\documentclass[a4paper,10pt,reqno]{amsart}
\usepackage{dsfont,marginnote,amsmath,amsfonts,amscd,amssymb,graphicx,mathrsfs,eufrak,bbding}
\usepackage[all,arc,curve,color,frame]{xy}
\usepackage[usenames]{color}

\usepackage[colorlinks]{hyperref}
\usepackage{tikz,mathrsfs}
\usetikzlibrary{arrows,decorations.pathmorphing,decorations.pathreplacing,positioning,shapes.geometric,shapes.misc,decorations.markings,decorations.fractals,calc,patterns}

\tikzset{>=stealth',
        cvertex/.style={circle,draw=black,inner sep=1pt,outer sep=3pt},
        vertex/.style={circle,fill=black,inner sep=1pt,outer sep=3pt},
        star/.style={circle,fill=yellow,inner sep=0.75pt,outer sep=0.75pt},
        tvertex/.style={inner sep=1pt,font=\scriptsize},
        gap/.style={inner sep=0.5pt,fill=white}}

\newcommand{\arrow}[2][20]
{
 \hspace{-5pt}
 \begin{tikzpicture}
  \node (A) at (0,0) {};
  \node (B) at (#1pt,0) {};
  \draw [#2] (A) -- (B);
 \end{tikzpicture}
 \hspace{-5pt}
}

\newtheorem{them}{Theorem}[section]

\theoremstyle{definition} 
\newtheorem{thm}[them]{Theorem}
\newtheorem{prop}[them]{Proposition}
\newtheorem{lemma}[them]{Lemma}
\newtheorem{defin}[them]{Definition}
\newtheorem{Tdefin}[them]{Temporary Definition}
\newtheorem{cor}[them]{Corollary}

\newtheorem{facts}[them]{Facts}
\newtheorem{fact}[them]{Fact}

\newtheorem{notation}[them]{Notation}
\newtheorem{example}[them]{Example}
\newtheorem{setting}[them]{Setting}
\newtheorem{examples}[them]{Examples}
\newtheorem{remark}[them]{Remark}

\newtheorem{ex}[them]{Exercise}

\numberwithin{equation}{section}

\newcommand{\p}{\mathfrak{p}}
\newcommand{\m}{\mathfrak{m}}

\renewcommand{\c}[1]{\mathcal{#1}}
\renewcommand{\u}[1]{\underline{#1}}

\renewcommand{\t}[1]{\textnormal{#1}}
\renewcommand{\tt}[1]{\mathtt{#1}}
\newcommand{\e}{\varepsilon}

\def\H{\mathop{\rm H_{}}\nolimits}
\def\RHom{\mathop{\rm {\bf R}Hom}\nolimits}
\def\Rep{\mathop{\rm Rep}\nolimits}

\def\det{\mathop{\rm det}\nolimits}
\def\op{\mathop{\rm op}\nolimits}
\def\GL{\mathop{\rm GL}\nolimits}
\def\thick{\mathop{\rm thick}\nolimits}
\def\Irr{\mathop{\rm Irr}\nolimits}
\def\SL{\mathop{\rm SL}\nolimits}
\def\uCM{\mathop{\underline{\rm CM}}\nolimits}
\def\umod{\mathop{\underline{\rm mod}}\nolimits}
\def\CM{\mathop{\rm CM}\nolimits}

\def\depth{\mathop{\rm depth}\nolimits}
\def\fl{\mathop{\sf fl}\nolimits}
\def\hgt{\mathop{\rm ht}\nolimits}
\def\mod{\mathop{\rm mod}\nolimits}

\def\coh{\mathop{\rm coh}\nolimits}
\def\Qcoh{\mathop{\rm Qcoh}\nolimits}
\def\Mod{\mathop{\rm Mod}\nolimits}
\def\fdmod{\mathop{\rm fdmod}\nolimits}
\def\refl{\mathop{\rm ref}\nolimits}
\def\proj{\mathop{\rm proj}\nolimits}

\def\pd{\mathop{\rm proj.dim}\nolimits}
\def\id{\mathop{\rm inj.dim}\nolimits}
\def\Hom{\mathop{\rm Hom}\nolimits}
\def\End{\mathop{\rm End}\nolimits}
\def\Ext{\mathop{\rm Ext}\nolimits}
\def\Tor{\mathop{\rm Tor}\nolimits}
\def\Tr{\mathop{\rm Tr}\nolimits}
\def\add{\mathop{\rm add}\nolimits}

\def\Ker{\mathop{\rm Ker}\nolimits}

\def\Sing{\mathop{\rm Sing}\nolimits}
\def\Supp{\mathop{\rm Supp}\nolimits}
\def\Ass{\mathop{\rm Ass}\nolimits}

\def\Spec{\mathop{\rm Spec}\nolimits}
\def\Max{\mathop{\rm Max}\nolimits}
\def\gl{\mathop{\rm gl.dim}\nolimits}

\def\D{\mathop{\rm{D}^{}}\nolimits}
\def\C{\mathop{\rm{C}^{}}\nolimits}
\def\Dsg{\mathop{\rm{D}_{\sf sg}}\nolimits}
\def\Db{\mathop{\rm{D}^b}\nolimits}
\def\Kb{\mathop{\rm{K}^b}\nolimits}
\def\Perf{\mathop{\mathfrak{Perf}}\nolimits}
\def\homdim{\mathop{\rm homdim}\nolimits}

\def\homdim{\mathop{\rm homdim}\nolimits}

\begin{document}
\begingroup
\centering 
\vspace*{0.12\textheight } 
{\large \textsc{MSRI Summer School}}\\[0.05\textheight] 
{\Huge LECTURES}\\[\baselineskip] 
{\Large ON}\\[\baselineskip]  
{\Huge NONCOMMUTATIVE}\\[\baselineskip] 
{\Huge RESOLUTIONS}\\[0.07\textheight]
{\scshape Michael Wemyss}\\ [0.12\textheight ] 
{\FiveStarOpenCircled}\\[0.5\baselineskip] 
{\small\scshape June 2012}\par 
\vfill\null 
\endgroup

\clearpage

\setcounter{section}{-1}

\section{Introduction}\label{L0}

The notion of a noncommutative crepant resolution (NCCR) was introduced by Van den Bergh \cite{VdBNCCR}, following his interpretation \cite{VdB1d} of work of Bridgeland \cite{Bridgeland} and Bridgeland--King--Reid \cite{BKR}.  Since then, NCCRs has appeared prominently in both the mathematics and physics literature as a general homological structure that underpins many topics currently of interest, for example moduli spaces, dimer models, curve counting Donaldson--Thomas invariants, spherical--type twists, the minimal model program and mirror symmetry. 

My purpose in writing these notes is to give an example based approach to some of the ideas and constructions for NCCRs, with latter sections focusing more on the explicit geometry and restricting mainly to dimensions two and three, rather than simply presenting results in full generality.  The participants at the MSRI Summer School had a wonderful mix of diverse backgrounds, so the content and presentation of these notes reflect this.  There are exercises scattered throughout the text, at various levels of sophistication, and also computer exercises that hopefully add to the intuition.

The following is a brief outline of the content of the notes.  In Lecture 1 we begin by outlining some of the motivation and natural questions for NCCRs through the simple example of the $\mathbb{Z}_3$ surface singularity.  which form the basis of our discussions.  This is done in the setting of two-dimensional Gorenstein quotient singularities for simplicity, although most things work much more generally.  We introduce the notion of Auslander algebras, and link to the idea of finite CM type.  We then introduce skew group rings and use this to show that a certain endomorphism ring in the running example has finite global dimension. 

Lecture 2 begins with the formal definitions of Gorenstein and CM rings, depth and CM modules, before giving the definition of a noncommutative crepant resolution (NCCR).  This comes in two parts, and the second part is motivated using some classical commutative algebra.  We then deal with uniqueness issues, showing that in dimension two NCCRs are unique up to Morita equivalence, whereas in dimension three they are unique up to derived equivalence.  We give examples to show that these results are the best possible.  Along the way, the three key technical results of the depth lemma, reflexive equivalence and the Auslander--Buchsbaum formula are formulated.

Lecture 3 breaks free of the algebraic shackles imposed in the previous two sections, by giving a brief overview of quiver GIT.  This allows us to extract geometry from NCCRs, and we illustrate this in examples of increasing complexity.

In Lecture 4 we go homological so as to give a language in which to compare the geometry to NCCRs.  We sketch some aspects of derived categories, and give an outline of tilting theory.  We then illustrate tilting explicitly in the examples from Lecture 3.  In a purely homological section we then relate the crepancy of birational morphisms to the condition $\End_R(M)\in\CM R$.  The section then considers  CY categories and algebras, and we prove that NCCRs are $d$-CY.  We formulate singular derived categories as a mechanism to relate the constructions involving CM modules to the CY property, and also  as an excuse to introduce AR duality, which links questions from Lecture 1 to AR sequences, which appear in Lecture 5.

Lecture 5 begins by overviewing the McKay Correspondence in dimension two.  It starts with the classical combinatorial version, before giving the Auslander version using AR sequences and the category of CM modules.  We then upgrade this and give the derived version, which homologically relates minimal resolutions to NCCRs for ADE surface singularities.  The last, and main, section gives the three-dimensional version, which is the original motivation for introducing NCCRs. We sketch the proof.

There is a short appendix (\S\ref{QuiverAppendix}) which gives very basic background on quiver representations, and sets the notation that is used in the examples and exercises.

\medskip
\noindent {\bf Acknowledgements.}  I would like to thank all the participants and co-organizers of the summer school for their questions, and for making the course such fun to teach.  In addition, thanks go to Pieter Belmans, Kosmas Diveris, Will Donovan, Martin Kalck, Joe Karmazyn, Boris Lerner, Alice Rizzardo and Toby Stafford for their many comments on previous drafts of these notes.

\newpage

\section{Motivation and First Examples}\label{L1}

\subsection{The Basic Idea}\label{L1 basic idea}  The classical method for resolving singularities is to somehow associate to the singularity $X=\Spec R$ an ideal $I$.  This becomes the centre of the blowup and we hope to resolve the singularity via the picture
\[
\begin{tikzpicture}
\node (1) at (0,2) {${\rm Bl}_{I}(X)$};
\node (1b) at (0,0) {$\Spec R$};
\node (2) at (2,1) {$I$};
\draw[->,decorate, 
decoration={snake,amplitude=.4mm,segment length=2mm,post length=1mm}] (1b) to  (2);
\draw[->,decorate, 
decoration={snake,amplitude=.4mm,segment length=2mm,post length=1mm}] (2) to (1);
\draw[->] (1) to (1b);  
\end{tikzpicture}
\]

Although this is entirely inside the world of commutative algebra,
we hope to obtain a better understanding of this process by
introducing noncommutative methods.  Instead of finding an ideal ${I}$, we want to, without referring to a resolution (i.e.\
the answer), produce a non-commutative ring $A$ from which we can extract resolution(s) of $X$.
\begin{eqnarray}
\begin{array}{c}
\begin{tikzpicture}
\node (1) at (0,2) {$\mathcal{M}$};
\node (1b) at (0,0) {$\Spec R$};
\node (2) at (2,1) {$A$};
\draw[->,decorate, 
decoration={snake,amplitude=.4mm,segment length=2mm,post length=1mm}] (1b) to (2);
\draw[->,decorate, 
decoration={snake,amplitude=.4mm,segment length=2mm,post length=1mm}] (2) to (1);
\draw[->] (1) to (1b);  
\end{tikzpicture}
\end{array}
\label{L1mot2}
\end{eqnarray}

Just as there are many different ideal sheaves which give the same blowup, there are in general many different non-commutative rings that can be used to resolve the singularity and so the subtlety comes through asking for the `best' one --- for a noncommutative ring to be called a noncommutative resolution it needs to satisfy some extra conditions.

The purpose of these lectures is to explain how to go about constructing such an $A$, and they will also outline some of the methods that are used to extract the geometry.  There are both algebraic and geometric consequences.  The main benefit of this noncommutative approach is that we equip the geometry with extra structure in the form of tautological bundles, which can then be used in various homological (and explicit) constructions.

\subsection{Motivation and Questions}

Here we input a finite subgroup $G$ of $\SL(2,\mathbb{C})$.  Then $G$ acts on $\mathbb{C}^2$, so it acts on $\mathbb{C}[[x,y]]$ via inverse transpose.  We define $R:=\mathbb{C}[[x,y]]^G$ and consider the quotient germ $\mathbb{C}^2/G=\Spec R$.  

\begin{example}\label{Z3 running}
The running example will be $G=\frac{1}{3}(1,2):=\left\langle g:=\left( \begin{smallmatrix} \varepsilon_3 &0\\ 0&\varepsilon_3^2  \end{smallmatrix}\right)\right\rangle$ where $\varepsilon_3$ is a primitive third root of unity.  Here $g$ sends $x\mapsto \varepsilon_3^2x$ and $y\mapsto \varepsilon_3y$, so it is clear that $x^3, xy, y^3$ are all invariants.  In fact they generate the invariant ring, so 
\[
R=\mathbb{C}[[x,y]]^{\frac{1}{3}(1,2)}=\mathbb{C}[[x^3,y^3,xy]]\cong \mathbb{C}[[a,b,c]]/ (ab-c^3).
\]
\end{example}

\begin{setting}\label{L1 setting}
Throughout the remainder of this section, $R$ will always denote $\mathbb{C}[[x,y]]^G$ for some $G\leq \SL(2,\mathbb{C})$.  We remark that experts can instead take their favourite complete local Gorenstein ring $R$ with $\dim R=2$, as all results remain true. Indeed most results still hold when $R$ is a complete local CM normal ring of dimension two, provided that $R$ has a canonical module.  Note that dropping the `complete local' is possible, but at the expense of making the language a bit  more technical, and the proofs much more so.
\end{setting}

Recall that if $M$ is a finitely generated $R$-module (written $M\in\mod R$), there exists a surjection $R^n\twoheadrightarrow M$ for some $n\in\mathbb{N}$, and the kernel is denoted $\Omega M$.  This is called the syzygy of $M$.  

\begin{Tdefin}\label{CM defin syz}
$M\in\mod R$ is called a Cohen--Macaulay (=CM) module if $M\cong\Omega(\Omega X)$ for some $X\in\mod R$. We denote the category of CM modules by $\CM R$.
\end{Tdefin}

We remark that the definition should be treated with caution, as assumptions are needed (which are satisfied in Setting~\ref{L1 setting} above) for it to be equivalent to the more standard definition that will be explained later (\ref{depth definition}).  Note that Temporary Definition~\ref{CM defin syz} ensures that $R\in\CM R$.  

\begin{example}\label{CM of Z3}
Let $G=\frac{1}{3}(1,2):=\left\langle\left( \begin{smallmatrix} \varepsilon_3 &0\\ 0&\varepsilon_3^2  \end{smallmatrix}\right)\right\rangle$ where $\varepsilon_3$ is a primitive third root of unity.  In this situation $R=\mathbb{C}[[x^3,y^3,xy]]\cong\mathbb{C}[[a,b,c]]/(ab-c^3)$.  We claim that $R$, together with the ideals $M_1:=(a,c)$ and $M_2:=(a,c^2)$, are all CM $R$-modules.  In fact, the situation is particularly nice since the calculation below shows that $\Omega M_1\cong M_2$ and $\Omega M_2\cong M_1$, hence $\Omega^2 M_1\cong M_1$ and $\Omega^2 M_2\cong M_2$.   

For the calculation, just note that we have a short exact sequence
\begin{eqnarray}
0\to (a,c^2)\xrightarrow{(-\frac{c}{a}\ inc)}R^2\xrightarrow{a\choose c}(a,c)\to 0\label{M1}
\end{eqnarray}
where the second map sends $(r_1,r_2)\mapsto r_1a+r_2c$, and the first map sends $ra+sc^2\mapsto ((ra+sc^2)(-\frac{c}{a}),ra+sc^2)=(-rc-sb,ra+sc^2)$.  In a similar way, we have an exact sequence
\begin{eqnarray}
0\to (a,c)\xrightarrow{(-\frac{c^2}{a}\ inc)}R^2\xrightarrow{a\choose c^2}(a,c^2)\to 0.\label{M2}
\end{eqnarray}
\end{example}

Notice that with our Temporary Definition~\ref{CM defin syz}, the above example shows that CM modules are in fact quite easy to produce --- just find a module, then syzygy twice.   Note at this stage it is not clear how many other CM modules there are in Example~\ref{CM of Z3}, never mind what this has to do with \S\ref{L1 basic idea} and the relationship to the geometry.

\begin{example}\label{quiver of Z3}
Continuing the above example, in the spirit of discovery let's compute $\End_R(R\oplus (a,c)\oplus (a,c^2))$.  Why we do this will only become clear afterwards.   We write each indecomposable module as a vertex
\[
\begin{tikzpicture} [bend angle=45, looseness=1.2]
\node[name=s,regular polygon, regular polygon sides=3, minimum size=2cm] at (0,0) {}; 
\node (C1) at (s.corner 1)  {$\scriptstyle (a,c)$};
\node (C2) at (s.corner 2)  {$\scriptstyle R$};
\node (C3) at (s.corner 3)  {$\scriptstyle (a,c^2)$};
\end{tikzpicture}
\]
Clearly we have the inclusions $(a,c^2)\subseteq(a,c)\subseteq R$ and so we have morphisms
\[
\begin{tikzpicture} [bend angle=45, looseness=1.2]
\node[name=s,regular polygon, regular polygon sides=3, minimum size=2cm] at (0,0) {}; 
\node (C1) at (s.corner 1)  {$\scriptstyle (a,c)$};
\node (C2) at (s.corner 2)  {$\scriptstyle R$};
\node (C3) at (s.corner 3)  {$\scriptstyle (a,c^2)$};
\draw [->,bend right] (C1) to node [left] {$\scriptstyle { inc}$} (C2);
\draw [->,bend right] (C3) to node [right] {$\scriptstyle { inc}$} (C1);
\end{tikzpicture}
\]
If we multiply an element of $R$ by $c$ we get an element of $(a,c)$, and similarly if we multiply an element of $(a,c)$ by $c$ we get an element of $(a,c^2)$.  Thus we add in 
\[
\begin{tikzpicture} [bend angle=45, looseness=1.2]
\node[name=s,regular polygon, regular polygon sides=3, minimum size=2cm] at (0,0) {}; 
\node (C1) at (s.corner 1)  {$\scriptstyle (a,c)$};
\node (C2) at (s.corner 2)  {$\scriptstyle R$};
\node (C3) at (s.corner 3)  {$\scriptstyle (a,c^2)$};
\draw[->] (C2) -- node  [gap] {$\scriptstyle c$} (C1); 
\draw[->] (C1) -- node  [gap] {$\scriptstyle c$} (C3);
\draw [->,bend right] (C1) to node [left] {$\scriptstyle { inc}$} (C2);
\draw [->,bend right] (C3) to node [right] {$\scriptstyle { inc}$} (C1);
\end{tikzpicture}
\]
We can multiply an element of $R$ by $a$ to get an element of $(a,c^2)$ and so obtain 
\[
\begin{tikzpicture} [bend angle=45, looseness=1.2]
\node[name=s,regular polygon, regular polygon sides=3, minimum size=2cm] at (0,0) {}; 
\node (C1) at (s.corner 1)  {$\scriptstyle (a,c)$};
\node (C2) at (s.corner 2)  {$\scriptstyle R$};
\node (C3) at (s.corner 3)  {$\scriptstyle (a,c^2)$};
\draw[->] (C2) -- node  [gap] {$\scriptstyle c$} (C1); 
\draw[->] (C1) -- node  [gap] {$\scriptstyle c$} (C3);
\draw [->,bend right] (C1) to node [left] {$\scriptstyle { inc}$} (C2);
\draw [->,bend right] (C2) to node [below] {$\scriptstyle a$} (C3);
\draw [->,bend right] (C3) to node [right] {$\scriptstyle { inc}$} (C1);
\end{tikzpicture}
\]
It is possible to multiply an element of $R$ by $a$ and get an element of $(a,c)$, but we don't draw it since this map is just the composition of the arrow $a$ with the arrow $inc$, so it is already taken care of.  The only morphism which is not so obvious is the map $(a,c^2)\to R$ given by $\frac{c}{a}$ (as in (\ref{M1})), thus this means that we have guessed the following morphisms between the modules:
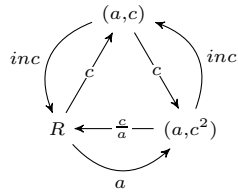
\[
\begin{tikzpicture} [bend angle=45, looseness=1.2]
\node[name=s,regular polygon, regular polygon sides=3, minimum size=2cm] at (0,0) {}; 
\node (C1) at (s.corner 1)  {$\scriptstyle (a,c)$};
\node (C2) at (s.corner 2)  {$\scriptstyle R$};
\node (C3) at (s.corner 3)  {$\scriptstyle (a,c^2)$};
\draw[->] (C2) -- node  [gap] {$\scriptstyle c$} (C1); 
\draw[->] (C1) -- node  [gap] {$\scriptstyle c$} (C3);
\draw[->] (C3) -- node  [gap,pos=0.4] {$\scriptstyle \frac{c}{a}$} (C2); 
\draw [->,bend right] (C1) to node [left] {$\scriptstyle { inc}$} (C2);
\draw [->,bend right] (C2) to node [below] {$\scriptstyle a$} (C3);
\draw [->,bend right] (C3) to node [right] {$\scriptstyle { inc}$} (C1);
\end{tikzpicture}
\]
It turns out that these are in fact all necessary morphisms, in that any other must be a linear combination of compositions of these.  This can be shown directly, but at this stage it is not entirely clear.  Is this algebra familiar? 
\end{example}
\noindent

The above simple example already illustrates some interesting phenomenon that later we will put on a more firm theoretical basis.  Indeed, the following five questions emerge naturally, and motivate much of the content of these notes.
\begin{enumerate}
\item[Q1.] Is there a systematic way of computing the quiver? We just guessed.
\item[Q2.] Are the three CM $R$-modules we guessed in Example~\ref{CM of Z3} all the indecomposable CM $R$-modules up to isomorphism? 
\item[Q3.] Is it a coincidence that $\Omega^2=\mathrm{Id}$ on the CM modules? 
\item[Q4.] Why are we only considering noncommutative rings that look like $\End_R(M)$?  It seems that \S\ref{L1 basic idea} allows for almost arbitrary rings. 
\item[Q5.] How do we extract the geometry from these noncommutative rings, as in (\ref{L1mot2})?
\end{enumerate}
In short, the answers are:
\begin{enumerate}
\item[A1.] Yes.  This is one of the things that Auslander--Reiten (=AR) theory does.  
\item[A2.] Yes.  The proof, using Auslander algebras, is remarkably simple.  This will lead to our main definition of noncommutative crepant resolutions (=NCCRs).
\item[A3.] No.  This is a special case of matrix factorizations, which appear for any hypersurface. Amongst other things, this leads to connections with Mirror Symmetry.
\item[A4.] Mainly for derived category reasons.  See \S\ref{L4} and remarks there.
\item[A5.] We use quiver GIT.  See \S\ref{L3}.
\end{enumerate}
The remainder of \S\ref{L1} will focus on Q2.

\subsection{Auslander Algebras and Finite Type} 
This section is purely algebraic, and showcases Auslander's philosophy that endomorphism rings of finite global dimension are important from a representation--theoretic viewpoint.  In algebraic geometry commutative rings with finite global dimension correspond precisely to non-singular varieties (see \S\ref{gldim section}), so Auslander's philosophy will guide us forward.
\begin{thm}\label{thm Auslander}(Auslander) Let $R$ be as in Setting~\ref{L1 setting}, and let $\mathcal{C}$ be a finite set of  indecomposable CM $R$-modules, such that $R\in\mathcal{C}$.  Then the following are equivalent\\
\t{(1)} $\gl\End_R(\bigoplus_{C\in\mathcal{C}}C)\leq 2$.\\
\t{(2)} $\mathcal{C}$ contains all indecomposable CM $R$-modules (up to isomorphism).
\end{thm}

In fact, as the proof below shows, \t{(1)}$\Rightarrow$\t{(2)} holds in arbitrary dimension, whereas \t{(2)}$\Rightarrow$\t{(1)} needs $\dim R=2$.   To prove \ref{thm Auslander} will require two facts.  The first is quite easy to prove, the second requires a little more technology.  Recall if $M\in\mod R$ we denote $\add M$ to be the collection of all direct summands of all finite direct sums of $M$.  If $\Lambda$ is a ring, then $\proj\Lambda:=\add \Lambda$, the category of projective $\Lambda$-modules. 
\begin{facts}\label{facts1and2} With notation as above,
\begin{enumerate}
\item\label{F1}  If $M\in\mod R$ contains $R$ as a summand, then the functor
\[
\Hom_R(M,-):\mod R\to\mod\End_R(M)
\]
is fully faithful, restricting to an equivalence $\add M\stackrel{\simeq}{\to}\proj\End_R(M)$.
\item\label{F2} Since $R$ is Gorenstein (or normal CM), CM $R$-modules are always reflexive, i.e.\ the natural map $X\to X^{**}=\Hom_R(\Hom_R(X,R),R)$ is an isomorphism.
\end{enumerate}
\end{facts}
With these facts, the proof of (1)$\Rightarrow$(2) is quite straightforward.  The proof (2)$\Rightarrow$(1) uses the depth lemma, which will be explained in \S\ref{L2}.

\medskip
\noindent
{\it Proof of \ref{thm Auslander}.}  Denote $M:=\bigoplus_{C\in\mathcal{C}}C$ and $\Lambda:=\End_R(M)$. 

\medskip
\noindent
(1)$\Rightarrow$(2) Suppose that $\gl\Lambda\leq 2$ and let $X\in\CM R$.  Consider a projective resolution $R^b\to R^a\to X^*\to 0$, then dualizing via $(-)^*=\Hom_R(-,R)$ and using \ref{facts1and2}(\ref{F2}) gives an exact sequence $0\to X\to R^a\to R^b$.  Applying $\Hom_R(M,-)$ then gives an exact sequence
\[
0\to \Hom_R(M,X)\to \Hom_R(M,R^a)\to\Hom_R(M,R^b).
\]
Both $\Hom_R(M,R^a)$ and $\Hom_R(M,R^b)$ are projective $\Lambda$-modules by \ref{facts1and2}(\ref{F1}), so since by assumption $\gl\Lambda\leq 2$ it follows that $\Hom_R(M,X)$ is a projective $\Lambda$-module. One last application of \ref{facts1and2}(\ref{F1}) shows that $X\in\add M$.

\medskip
\noindent
(2)$\Rightarrow$(1) Suppose that $\dim R=2$, and that $\mathcal{C}$ contains all indecomposable CM $R$-modules up to isomorphism.  Let $Y\in\mod\Lambda$, and consider the initial terms in a projective resolution
$P_1\stackrel{f}{\to} P_0\to Y\to 0$. By \ref{facts1and2}(\ref{F1}) there exists a morphism
$M_1\stackrel{g}{\to} M_0$ in $\add M$ such that
\[
(P_1\stackrel{f}{\to} P_0)=(\Hom_R(M,M_1)\stackrel{\cdot g}{\to}\Hom_R(M,M_0)).
\]
Put $X:=\Ker g$.  Since $\dim R=2$, by the depth lemma (see \ref{depth lemma} in the next section) we have $X\in\CM R$, so by assumption $X\in\mathcal{C}$.  Hence we have an exact sequence
\[
0\to X\to M_1\to M_0
\] 
with each term in $\add M$.  Simply applying $\Hom_R(M,-)$ gives an exact sequence
\[
0\to\Hom_R(M,X)\to\Hom_R(M,M_1)\to\Hom_R(M,M_0)\to Y\to0.
\]
Thus we have $\pd{}_\Lambda Y\le 2$.  Since this holds for all $Y$, $\gl\End_R(M)\leq 2$.  In fact, since $\End_R(M)$ has finite length modules, a global dimension of less than 2 would contradict the depth lemma.\qed

\medskip

Thus to answer Q2, by \ref{thm Auslander} we show that $\gl\End_R(R\oplus (a,c)\oplus (a,c^2))\leq 2$.  It is possible just to do this directly,  using the calculation in \ref{quiver of Z3}, but the next subsection gives a non-explicit proof.

\subsection{Skew Group Rings} Recall our setting $G\leq\SL(2,\mathbb{C})$ and $R=\mathbb{C}[[x,y]]^G$.  Since $R$ is defined as a quotient of the smooth space $\Spec \mathbb{C}[[x,y]]$ by $G$, the basic idea is that we should use the module theory of $\mathbb{C}G$, together with the module theory of $\mathbb{C}[[x,y]]$, to encode  some of the geometry of the quotient.  The (false in general) slogan is that `$G$-equivariant sheaves on $\mathbb{C}[[V]]$ encode the geometry of the resolution of $V/G$'.

\begin{example}\label{Z3 new coord}
Let $G=\frac{1}{3}(1,2)=\langle g\rangle$ as in \ref{CM of Z3}.  Consider the one-dimensional representations $\rho_0$, $\rho_1$ and $\rho_2$ of $G$, and denote their bases by $e_0$, $e_1$ and $e_2$.  Our convention is that $g$ acts on $e_i$ with weight $\varepsilon_3^i$.  Recall $g$ acts on the polynomial ring via $x\mapsto\varepsilon_3^{-1}x$ and $y\mapsto\varepsilon_3y$, hence $G$ acts on both side of the tensor $\mathbb{C}[[x,y]]\otimes_\mathbb{C}\rho_i$ and so we can consider the invariants $(\mathbb{C}[[x,y]]\otimes_\mathbb{C}\rho_i)^G$.  Note that since $G$ acts trivially on $\rho_0$, we have that $R=\mathbb{C}[[x,y]]^G=(\mathbb{C}[[x,y]]\otimes_\mathbb{C}\rho_0)^G$.  Denote $N_1:=(\mathbb{C}[[x,y]]\otimes_\mathbb{C}\rho_1)^G$ and $N_2:=(\mathbb{C}[[x,y]]\otimes_\mathbb{C}\rho_2)^G$.

Note that $x\otimes e_1$ belongs to $N_1$ since under the action of $G$, $x\otimes e_1\mapsto \varepsilon^{-1}_3x\otimes\varepsilon_3 e_1=x\otimes e_1$.  Similarly $y^2\otimes e_1\in N_1$.  In fact $x\otimes e_1$ and $y^2\otimes e_1$ generate $N_1$ as an $R$-module.  Similarly $x^2\otimes e_2$ and $y\otimes e_2$ generate $N_2$ as an $R$-module.  In fact, $N_2\cong M_1$ and $N_1\cong M_2$ where the $M_i$ are as in \ref{CM of Z3}, and in these new coordinates, dropping tensors we have
\[
\End_R(R\oplus M_1\oplus M_2)\cong
\End_R(R\oplus N_2\oplus N_1)\cong
\begin{array}{c}
\begin{tikzpicture} [bend angle=45, looseness=1.2]
\node[name=s,regular polygon, regular polygon sides=3, minimum size=2cm] at (0,0) {}; 
\node (C1) at (s.corner 1)  {$\scriptstyle N_2$};
\node (C2) at (s.corner 2)  {$\scriptstyle R$};
\node (C3) at (s.corner 3)  {$\scriptstyle N_1$};
\draw[->] (C2) -- node  [gap] {$\scriptstyle y$} (C1); 
\draw[->] (C1) -- node  [gap] {$\scriptstyle y$} (C3);
\draw[->] (C3) -- node  [gap] {$\scriptstyle y$} (C2); 
\draw [->,bend right] (C1) to node [left] {$\scriptstyle x$} (C2);
\draw [->,bend right] (C2) to node [below] {$\scriptstyle x$} (C3);
\draw [->,bend right] (C3) to node [right] {$\scriptstyle x$} (C1);
\end{tikzpicture}
\end{array}
\]
\end{example}

\begin{defin}\label{skew def}
For a $\mathbb{C}$-algebra $A$ and a finite group $G$ together with a group homomorphism $G\to {\rm Aut}_{\mathbb{C}\rm{-alg}}(A)$, we define the skew group ring $A\# G$ as follows:  as a vector space it is $A\otimes_\mathbb{C}\mathbb{C}G$, with multiplication defined as  
\[
(f_1\otimes g_1)(f_2\otimes g_2):=(f_1\cdot g_1(f_2))\otimes g_1g_2,
\]
for any $f_1, f_2\in A$ and $g_1,g_2\in G$, extended by linearity.
\end{defin}

In these notes we will always use \ref{skew def} in the setting where $A$ a power series (or polynomial) ring in finitely many variables, and $G$ a finite subgroup of $\GL(n,\mathbb{C})$. The following theorem is due to Auslander.
\begin{thm}\label{Auslander skew}
(Auslander) Let $G\leq\SL(n,\mathbb{C})$ be a finite subgroup and denote $S:=\mathbb{C}[[x_1,\hdots,x_n]]$ and $R:=\mathbb{C}[[x_1,\hdots,x_n]]^G$.  Then 
\[
S\# G\cong \End_R
\left(\bigoplus_{\rho\in\Irr G}((S\otimes\rho)^{G})^{\oplus\dim_\mathbb{C}\rho}\right).
\]
\end{thm}
We remark that the theorem also holds if $G$ is a subgroup of $\GL(n,\mathbb{C})$ which contains no complex reflections (in the sense of Bellamy's lectures \ref{complex reflection}) except the identity.

By \ref{Auslander skew}, in the running example \ref{Z3 new coord} we have an isomorphism 
\[
\End_R(R\oplus (a,c)\oplus (a,c^2))\cong \mathbb{C}[[x,y]]\#\tfrac{1}{3}(1,2).
\]
Thus, via \ref{thm Auslander}, to show that $\{ R, (a,c), (a,c^2)\}$ are all the indecomposable CM $R$-modules it suffices to prove that $\gl\mathbb{C}[[x,y]]\#\frac{1}{3}(1,2)\leq 2$.

Now if $M,N\in \mod S\# G$ then $G$ acts on $\Hom_S(M,N)$ by $(gf)(m):=g\cdot f(g^{-1}m)$ for all $g\in G$, $f\in\Hom_S(M,N)$ and $m\in M$. It is easy to check that 
\[
\Hom_{S\# G}(M,N)=\Hom_S(M,N)^G.
\]
Further, since taking $G$-invariants is exact (since $G$ is finite, and we are working over $\mathbb{C}$), this induces a functorial isomorphism
\[
\Ext^i_{S\# G}(M,N)=\Ext^i_S(M,N)^G
\]
for all $i\geq 0$. In particular, $\gl S\# G\le \gl S$ holds, and so in our setting we have $\gl \mathbb{C}[[x,y]]\#\frac{1}{3}(1,2)\leq \gl\mathbb{C}[[x,y]]=2$, as required.

\begin{remark}\label{gl dim equal}
The above can be strengthened to show that $\gl S\# G=\gl S$.
\end{remark}

\medskip
\medskip
\noindent
{\bf Credits:} The material in this section is now quite classical.  The ideas around \ref{thm Auslander} were originally developed for representation dimension of Artin algebras \cite{A71}, but they came across to CM modules following Auslander's version of the McKay correspondence \cite{Aus3}.  These ideas were pursued later by Iyama in his higher dimensional AR theory \cite{IyamaAR}, who first observed the link to NCCRs.  See also the paper by Leuschke \cite{Leuschke}.  Fact \ref{facts1and2}(\ref{F1}) is known as `projectivization', see e.g.\ \cite[II.2.1]{ARS}, and Fact \ref{facts1and2}(\ref{F2}) can be found in most commutative algebra textbooks.  Skew group rings are also a classical topic.  There are now many proofs of \ref{Auslander skew}, see for example \cite{EG}, \cite{Martin} and \cite{IT}.  Auslander's original proof is outlined in \cite{Y}.

\subsection{Exercises}

\begin{ex}\label{Ex1.12}
Let $R$ be a commutative ring and let $M$ be an $R$-module.  Define
\[
\End_{R}(M):=\{f:M\to M\mid f\t{ is an $R$-module homomorphism}  \}.
\]
\begin{enumerate}
\item Verify that $\End_{R}(M)$ is indeed a ring, and has the structure of an $R$-module.
\item Give an example of $R$ and $M$ for which $\End_{R}(M)$ is a commutative ring, and give an example for which $\End_{R}(M)$ is noncommutative.  Roughly speaking, given an $R$ and $M$ how often is the resulting endomorphism ring $\End_{R}(M)$ to be commutative?
\item We say that $M$ is a \emph{simple} $R$-module if the only submodules of $M$ are $\{ 0\}$ and $M$.  Prove that if $R$ is any ring and $M$ is a simple $R$-module then $\End_{R}(M)$ is a division ring.
\end{enumerate}
\end{ex}

\begin{ex}\label{Ex1.13}
As in \S\ref{L1}, consider the group 
\[
\frac{1}{r}(1,a):=\left\langle g:=\begin{pmatrix} \varepsilon_r &0\\ 0&\varepsilon_r^a  \end{pmatrix}\right\rangle
\]
where $\varepsilon_r$ is a primitive $r^{\rm th}$ root of unity. We assume that $r$ and $a$ are coprime, and denote the representations of $G$ by $\rho_0,\hdots,\rho_{r-1}$.
\begin{enumerate}
\item Show that $S_i:=(\mathbb{C}[[x,y]]\otimes \rho_i)^G\cong\{ f\in\mathbb{C}[[x,y]]\mid g\cdot f=\varepsilon_r^{i} f  \}$.  (The exact superscript on $\varepsilon$ will depend on conventions).
\item For $a=r-1$ (i.e.\ the group $G$ is inside $\SL(2,\mathbb{C})$), 
\begin{enumerate}
\item Determine $R=S_0$, and find generators for the $R$-modules $S_i$.  
\item Hence of otherwise, determine the quiver of $\End_R(\oplus_{i=0}^{r-1}S_i)$. 
\end{enumerate}
\item (This will be helpful for counterexamples later)  Consider $G=\frac{1}{3}(1,1)$ and $\frac{1}{5}(1,2)$.  For both of these examples,
\begin{enumerate}
\item Determine $R=S_0$, and find generators for the $R$-modules $S_i$. This should be quite different to (2)(a).
\item Hence or otherwise, determine the quiver of $\End_R(\oplus_{i=0}^{r-1}S_i)$.
\item Consider only the modules in (3)(a) that have two generators.  Sum them together, along with $R$.  Determine the quiver of the resulting endomorphism ring.
\end{enumerate}
\end{enumerate}
\end{ex}

\noindent{\bf Computer Exercises:}

\begin{ex}\label{Ex1.14}
(When is a ring CM?).  Consider the ring $R:=k[[a,b,c]]/(ab-c^3)$ from \S\ref{L1}, where $k$ has characteristic zero.  We code this into Singular as
\[
\begin{array}{l}
\tt{ >LIB``homolog.lib";}\\
\tt{>LIB``sing.lib";}\\
\tt{>ring\quad S=0,(a,b,c),ds;}
\end{array}
\]
The first two commands loads libraries that we will use.  The last command defines the power series rings $S:=k[[a,b,c]]$ (this is the $\tt{ds}$; use $\tt{dp}$ for the polynomial ring) in the variables $\tt{a,b,c}$.  Now
\[
\begin{array}{l}
\tt{ >ideal\quad i=ab-c3;}\\
\tt{>dim\_slocus(std(i));}
\end{array}
\]
The first command specifies the ideal $\tt{i}$ (for more than one generator, separate with commas e.g.\ $\tt{ >ideal\quad i=ab-c3, a4-b2;}$).  The second command asks for the dimensional of the singular locus of the variety cut out by the ideal $\tt{i}$.  Here the answer given is zero, which means it is an isolated singularity.
\[
\begin{array}{l}
\tt{ >qring\quad R=std(i);}
\end{array}
\]
This specifies our ring $R$ to be the factor $S/I$.  We now define the free rank one $R$-module $F:=R_R$ 
\[
\begin{array}{l}
\tt{ >module \quad F=[0];}
\end{array}
\]
and ask whether it is CM via 
\[
\begin{array}{l}
\tt{ >depth(F);}\\
\tt{ >dim(F);}
\end{array}
\]
If these two numbers agree, then the ring is CM. Using a similar procedure, calculate whether the following are CM rings.  
\begin{enumerate}
\item (Whitney umbrella) $\mathbb{C}[[u,v,x]]/(uv^2-x^2)$.
\item The ring of invariants of $\frac{1}{4}(1,1)$, i.e. $\mathbb{C}[[x^4,x^3y,x^2y^2,xy^3,y^4]]$.  This is isomorphic to $\mathbb{C}[[a,b,c,d,e]]$ factored by the $2\times 2$ minors of
\[
\begin{pmatrix}a&b&c&d\\ b&c&d&e  \end{pmatrix}
\]
\item $\mathbb{C}[[x^4,x^3y,xy^3,y^4]]$.  This is isomorphic to $\mathbb{C}[[a,b,d,e]]$ factored by the $2\times 2$ minors of
\[
\begin{pmatrix}a&b^2&be&d\\ b&ad&d^2&e  \end{pmatrix}
\]
\item $\mathbb{C}[[u,v,x,y]]/(uv-f(x,y))$ where $f(x,y)\in\mathbb{C}[[x,y]]$.
\item Try experimenting with other commutative rings.  Roughly, how often are they CM?
\end{enumerate}
\end{ex}

\begin{ex}\label{Ex1.15}
(When is a module CM?)  The procedure to determine whether a module is CM is similar to the above. Singular encodes modules as factors of free modules, so for example the ideal $M:=(a,c)$ in the ring $R=k[[a,b,c]]/(ab-c^3)$, i.e.\
\[
R^2 \xrightarrow{{c\ \ -b\choose -a\ \ c^2}}
R^2 \xrightarrow{(a\ \ c)} (a,c) \to 0
\]
is coded using the columns of the matrix as
\[
\begin{array}{l}
\tt{ > module\quad M = [c, -a], [-b, c2];} 
\end{array}
\]
Alternatively, to automatically work out the relations between $a$ and $c$, code
\[
\begin{array}{l}
\tt{ > module\quad Na = [a], [c];}\\ 
\tt{ > module\quad N =syz(Na);}\\
\end{array}
\]
since the first line codes the factor $R/(a,b)$, and the second takes the kernel of the natural map $R\to R/(a,b)$, hence giving $(a,b)$.  Now endomorphism rings are also easy to code, for example
\[
\begin{array}{l}
\tt{ > module\quad E=Hom(N,N);}\\ 
\end{array}
\]
The procedure for checking the depth and dimension of a module is exactly the same as in the previous example, namely
\[
\begin{array}{l}
\tt{ > depth(N);}\\
\tt{ > dim(N);}\\
\tt{ > depth(E);}\\
\tt{ > dim(E);} 
\end{array}
\]
\begin{enumerate}
\item $E_7$ surface singularity $\mathbb{C}[[x,y,z]]/ (x^3+xy^3+z^2)$.  Determine whether the following are CM modules, and whether their endomorphism rings are CM.
\begin{enumerate}
\item The quotient field $k=R/\m$, i.e.
\[
R^3 \xrightarrow{({x\ \ y \ \ z})}
R \to k \to 0
\]
\item The module $\Omega k$.  
\item The module $\Omega^2 k$.
\item The module given by 
\[
R^4 \xrightarrow{\left(\begin{smallmatrix}-z&y^2&0&x\\ xy& z& -x^2&0\\ 0&-x&-z&y\\ x^2&0&xy^2&z  \end{smallmatrix}\right)}
R^4 \to M \to 0
\]
\end{enumerate}
\item The ring $\mathbb{C}[[u,v,x,y]]/(uv-xy)$.  Determine whether the following are CM modules, and whether their endomorphism rings are CM.
\begin{enumerate}
\item The quotient field $k=R/\m$, i.e.
\[
R \xrightarrow{({u\ \ v \ \ x\ \ y})}
R \to k \to 0
\]
\item The module $\Omega k$.  
\item The module $\Omega^2 k$.
\item The module $\Omega^3 k$.
\item The modules $(u,x)$, $(u,y)$, and $(u^2,ux,x^2)$.
\end{enumerate}
\item As (2), but with $\mathbb{C}[[u,v,x,y]]/(uv-x^2y)$
\end{enumerate}
Note that provided $R$ is CM of dimension $d$, $\Omega^dX$ is CM for all $X\in\mod R$, so Singular can be used to produce many CM modules.  If $d\geq 3$, if $Y\in\CM R$ then it is quite rare that $\End_R(Y)\in\CM R$.
\end{ex}

\begin{ex}\label{Ex1.16}
(When is a ring Gorenstein?) If $(R,\m)$ is local of dimension $d$ and $k=R/\m$, then $\Ext^{d+1}(k,R)=0$ implies that $R$ is Gorenstein.  This can be coded using the techniques from above.  For each ring in \ref{Ex1.14} and \ref{Ex1.15} above, check whether it is Gorenstein.  Try also some other commutative rings.  Roughly, how many often are they Gorenstein?
\end{ex}

\newpage
\section{NCCRs and Uniqueness Issues}\label{L2}

In the last section we started with a ring $R:=\mathbb{C}[[a,b,c]]/(ab-c^3)$, and guessed a CM module $M:=R\oplus (a,c)\oplus (a,c^2)$ such that $\gl\End_R(M)= 2$.  To go further requires  more technology.  The following is the homological definition of depth, and the usual definition of a CM module.

\begin{defin}\label{depth definition}
If $(R,\m)$ is a local ring and $M\in\mod R$, we define the {\it depth} of $M$ to be
\[
\depth_RM:={\rm min}\{ i\geq 0\mid \Ext^i_R(R/\m,M)\neq 0\}.
\]
For $M\in\mod R$ it is always true that $\depth M\leq\dim M\leq\dim R\leq\dim_{R/\m}\m/\m^2<\infty$.  We say that $M$ is a (maximal) CM module if $\depth M=\dim R$, and in this case we write $M\in\CM R$.   We say that $R$ is a CM ring if $R_R\in\CM R$, and we say that $R$ is Gorenstein if it is CM and further $\id R<\infty$.
\end{defin}

The definition is stated to make it clear that Gorenstein rings are a special class of CM rings. It turns out that in fact $\id R<\infty$ implies that $R$ is CM, so the above definition can be simplified.  When $R$ is not necessarily local, we define $M\in\mod R$ to be CM by reducing to the local case, namely $M$ is defined to be CM if  $M_\m\in\CM R_\m$ for all $\m\in\Max R$.  

To show that \ref{depth definition} is equivalent to the temporary definition from the last section (at least in the setting there) will require the following easy lemma, which will turn out to be one of our main tools.

\begin{lemma}\label{depth lemma}
(The depth lemma)  Suppose that $(R,\m)$ is a local ring and let $0\to  A\to B\to C\to 0$ be a short exact sequence of finitely generated $R$-modules.  Then\\
(1) If $\depth B>\depth C$ then $\depth A=\depth C+1$.\\
(1) $\depth A\geq {\rm min}\{ \depth B, \depth C\}$.
\end{lemma}
\begin{proof}
This just follows by applying $\Hom_R(R/\m,-)$ and applying the definition of depth to the resulting long exact sequence.
\end{proof}

\begin{lemma}\label{CM defin same}
Suppose that $R$ is a local Gorenstein (or local normal CM) ring of dimension 2.  Then $\CM R=\{ \Omega^2X\mid X\in\mod R\}$.
\end{lemma}
\begin{proof}
($\subseteq$) Let $X\in\CM R$.  By fact \ref{facts1and2}(\ref{F2}), $X$ is reflexive.  Take a projective resolution $R^b\to R^a\to X^*\to 0$, then dualizing via $(-)^*=\Hom_R(-,R)$ and using the fact that $X$ is reflexive gives an exact sequence $0\to X\to R^a\to R^b$.  This shows that $X$ is a second syzygy.\\
($\supseteq$) If $X$ is a second syzygy, we have short exact sequences
\begin{eqnarray}
0\to X\to R^a\to C\to 0\label{s1}
\end{eqnarray}
\begin{eqnarray}
0\to C\to R^b\to D\to 0\label{s2}
\end{eqnarray}
We know that $\depth R=2$, and $0\leq \depth D\leq \dim R=2$.  We go through each of the three cases:
\begin{itemize}
\item If $\depth D=0$, then the depth lemma applied to (\ref{s2}) shows that $\depth C=1$.  The depth lemma applied to (\ref{s1}) then shows that $\depth X=2$, so $X\in\CM R$.
\item If $\depth D=1$, then the depth lemma applied to (\ref{s2}) shows that $\depth C=2$.  The depth lemma applied to (\ref{s1}) then shows that $\depth X=2$, so $X\in\CM R$.
\item If $\depth D=2$, then the depth lemma applied to (\ref{s2}) shows that $\depth C=2$.  The depth lemma applied to (\ref{s1}) then shows that $\depth X=2$, so $X\in\CM R$.
\end{itemize} 
In all cases, we deduce that $X\in\CM R$.
\end{proof}

\begin{lemma}
In the running example (\ref{quiver of Z3}) from the last section, namely $R=\mathbb{C}[[a,b,c]]/(ab-c^3)$ and $M:=R\oplus (a,c)\oplus (a,c^2)$,  we have $\End_R(M)\in\CM R$.
\end{lemma}
\begin{proof}
Just take a projective resolution $R^s\to R^t\to M\to 0$, and apply $\Hom_R(-,M)$ to obtain an exact sequence $0\to \End_R(M)\to \Hom_R(R,M)^s\to\Hom_R(R,M)^t$.  This is just  
$0\to \End_R(M)\to M^s\to M^t$.  Since $M\in\CM R$, both $M^s$ and $M^t$ have depth 2.  Repeating the argument in the proof of \ref{CM defin same}, using the depth lemma, shows that $\depth\End_R(M)=2=\dim R$.
\end{proof}

\subsection{Definition of NCCRs} The upshot so far is that in our running example $R=\mathbb{C}[[a,b,c]]/(ab-c^3)$ and $M:=R\oplus (a,c)\oplus (a,c^2)$, we have discovered that $\End_R(M)\in\CM R$ and further $\gl\End_R(M)=\dim R$.   

\begin{defin}
Let $R$ be a (equicodimensional normal) CM ring.  A noncommutative crepant resolution (=NCCR) of $R$ is by definition a ring of the form $\End_R(M)$ for some $M\in\refl R$, such that\\
(1) $\End_R(M)\in\CM R$.\\
(2) $\gl\End_R(M)=\dim R$.
\end{defin}

The first important remark is that although the definition is made in the CM setting, to get any relationship with the geometry it turns out to be necessary to require that $R$ is Gorenstein. So, although we can always do algebra in the CM setting, when we turn to geometry we will restrict to Gorenstein rings.

In the definition of NCCR, the first condition $\End_R(M)\in\CM R$ turns out to correspond to the geometric property of crepancy (a map $f\colon X\to Y$ is called crepant if $f^*\omega_X=\omega_Y$), but it is hard to explain this without the derived category, so we postpone explanation until \S\ref{L4}.  The second condition $\gl\End_R(M)=\dim R$ is explained below.

\subsection{Global Dimension, Krull Dimension and Smoothness}\label{gldim section}
As motivation, suppose that $V$ is an irreducible variety and that $\mathbb{C}[V]$ denotes its coordinate ring.  By the work of Auslander--Buchsbaum and Serre in the 1950s, it is known that for $(R,\m)$ a commutative noetherian local ring, $R$ is a regular local ring if and only if $\gl R<\infty$.  In fact,
\[
V \mbox{ is non-singular}\iff \gl \mathbb{C}[V]<\infty\iff \gl\mathbb{C}[V]=\dim \mathbb{C}[V].
\]
Thus as soon as the global dimension is finite, necessarily it is equal to $\dim \mathbb{C}[V]$.  When asking for the noncommutative analogue of smoothness, it is natural to hope that something similar happens.  However, as the exercises should demonstrate, the noncommutative world is not so well behaved.  

\begin{remark}
Suppose that $R$ is a CM ring, and $M\in\CM R$.  Then it is possible that $\gl\End_R(M)<\infty$ without $\gl\End_R(M)=\dim R$ (see Exercise~\ref{Ex2.18}).
\end{remark}

Thus in the noncommutative situation we have to make a choice, either we use $\gl\End_R(M)<\infty$ or $\gl\End_R(M)=\dim R$.  Which to choose?  To motivate, consider the resolution of the cone singularity
\medskip
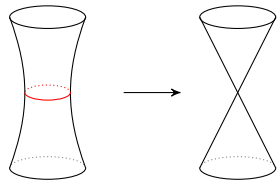
\[
\begin{tikzpicture} [bend angle=20, looseness=1,scale=1]
%\filldraw [red] (0,0) circle (0.5pt);
\node (p1) at (-3,0) {};
\draw (0,1) ellipse (0.5cm and 0.15cm);
\draw[gray,densely dotted] (0.5,-1) arc (0:180:0.5cm and 0.15cm);
\draw (0.5,-1) arc (0:-180:0.5cm and 0.15cm);  
\draw (-0.5,1) -- (0.5,-1);
\draw (0.5,1) -- (-0.5,-1);
\draw (-2.5,1) ellipse (0.5cm and 0.15cm);
\draw[gray,densely dotted] (-2,-1) arc (0:180:0.5cm and 0.15cm);
\draw (-2,-1) arc (0:-180:0.5cm and 0.15cm);  
\draw [bend left] (-3,1) to node (m2a) {} (-3,-1);
\draw [bend right] (-2,1) to node (m1a){} (-2,-1);
\coordinate (m1) at (m1a);
\coordinate (m2) at (m2a);
\draw[->] (-1.5,0) -- node[above] {} (-0.75,0);
\draw[densely dotted,red] (m1a) arc (0:180:0.3cm and 0.1cm);
\draw[red] (m1a) arc (0:-180:0.3cm and 0.1cm); 
\end{tikzpicture} 
\]
Imagine an ant standing at some point on the cooling tower.  It wouldn't know precisely which point it is at, since each point is indistinguishable from every other point.  Since points correspond (locally) to simple modules, every simple module should thus be expected to behave in the same way.  

Now if we have $\Lambda:=\End_R(M)\in\CM R$, by the depth lemma necessarily any simple $\Lambda$-module $S$ has $\pd_\Lambda S\geq \dim R$.  We don't want erratic behaviour like the projective dimension jumping (as in Exercise~\ref{Ex2.18}), so we choose the $\gl\End_R(M)=\dim R$ definition to ensure homogeneity.

\begin{remark}
When $R$ is Gorenstein, $M\in\refl R$ such that $\End_R(M)\in\CM R$ (which is satisfied in the geometric setting in \S\ref{L4}), by \ref{gl ok} below $\gl\End_R(M)<\infty\iff\gl\End_R(M)=\dim R$.  Thus in the main geometric setting of interest, homogeneity of the projective dimension of the simples is not an extra condition.
\end{remark}

\subsection{NCCRs are Morita Equivalent in Dimension 2}   In the algebraic geometric theory of surfaces, there exists a \emph{minimal} resolution through which all others factor.  This is unique up to isomorphism.  We can naively ask whether the same is true for NCCRs.  It is not, for stupid reasons:
\begin{example}
Consider the ring $R=\mathbb{C}[x,y]$.  Then both $\End_R(R)\cong R$ and $M_2(R)\cong \End_R(R\oplus R)$ are NCCRs of $R$.  They are clearly not isomorphic. 
\end{example}

It is well known (see also the exercises) that $R$ and $M_2(R)$ are Mortia equivalent, meaning that $\mod R\simeq \mod M_2(R)$ as categories.  Thus, even in dimension two, the best we can hope for is that NCCRs are unique up to Morita equivalence.

Recall that if $R$ is a domain with field of fractions $F$, then $a\in F$ is called \emph{integral over $R$} if it is the root of a monic polynomial in $R[X]$.  Clearly we have $R\subseteq \{ a\in F\mid a\textnormal{ is integral over }R\}$.  We say that $R$ is \emph{normal} if equality holds.  The key reason we are going to assume that $R$ is normal is the following fact \ref{facts3and4}(\ref{F3}), which will act as the replacement for our previous fact \ref{facts1and2}(\ref{F1}) (note that \ref{facts1and2}(\ref{F1}) required that $R\in\add M$).

In the following there is a condition on the existence of a canonical module, which is needed for various technical commutative algebra reasons.  In all the geometric situations we will be interested in (or when $R$ is Gorenstein) a canonical module does exist.

\begin{facts}\label{facts3and4} Suppose that $(R,\m)$ is a local CM normal domain of dimension $d$ with a canonical module, and let $M\in\refl R$. 
\begin{enumerate}
\item\label{F3} (Reflexive equivalence) $M$ induces equivalences of categories
\[
\begin{tikzpicture}
\node (a1) at (0,0) {$\refl R$};
\node (a2) at (4,0) {$\refl_R \End_R(M)$};
\node (b1) at (0,-1) {$\add R$};
\node (b2) at (4,-1) {$\proj\End_R(M) $};
\draw[->] (a1) -- node[above] {$\scriptstyle\Hom_R(M,-)$} node [below] {$\scriptstyle\sim$} (a2);
\draw[->] (b1) -- node[above] {$\scriptstyle\Hom_R(M,-)$} node [below] {$\scriptstyle\sim$} (b2);
\draw[right hook->] (b1) to (a1);
\draw[right hook->] (b2) to (a2);
\end{tikzpicture}
\]
where $\refl_R \End_R(M)$ denotes the category of those $\End_R(M)$ modules which are reflexive when considered as $R$-modules.
\item\label{F4} (The Auslander--Buchsbaum Formula) 
\begin{enumerate}
\item If $\Lambda:=\End_R(M)$ is a NCCR, then for all $X\in\mod\Lambda$ we have 
\[
\depth_R X+\pd_\Lambda X=\dim R.
\]  
\item If $R$ is Gorenstein and $\Lambda:=\End_R(M)\in\CM R$, then for all $X\in\mod\Lambda$ with $\pd_\Lambda M<\infty$ we have 
\[
\depth_R X+\pd_\Lambda X=\dim R.
\]  
\end{enumerate}
\end{enumerate}
\end{facts}
\noindent
The special case $M=R$ in (\ref{F4})(b), namely $\Lambda:=\End_R(R)\cong R$, gives the classical Auslander--Buchsbaum formula. As a first application, we have:

\begin{lemma}\label{gl ok}
Suppose that $R$ is a local Gorenstein normal domain, $M\in\CM R$ with $\End_R(M)\in\CM R$.  Then $\gl\End_R(M)<\infty$ if and only if $\gl\End_R(M)=\dim R$.
\end{lemma}
\begin{proof}
Global dimension can be computed as the supremum of the $\End_R(M)$-modules that have finite length as $R$-modules.\\
($\Rightarrow$)  Suppose that $\gl\End_R(M)<\infty$.  Each finite length $\End_R(M)$-module has depth zero, and by assumption has finite projective dimension. Hence by Auslander--Buchsbaum, each finite length module has projective dimension equal to $\dim R$, so $\gl\End_R(M)=\dim R$.
\end{proof}
The second application of Auslander--Buchsbaum is our first uniqueness theorem.

\begin{thm}\label{NCCR dim 2}
Let $(R,\m)$ be a local CM normal domain of dimension $2$ with a canonical module.   If $R$ has a NCCR, then all NCCRs of $R$ are Morita equivalent.
\end{thm}
\begin{proof}
Let $\End_{R}(M)$ and $\End_R(N)$ be NCCRs.  Consider $X\in\refl_R\End_R(M)$.  We know that $\depth X\geq 2$ by the depth lemma (exactly as in the proof of \ref{CM defin same}).  By Auslander--Buchsbaum (\ref{facts3and4}(\ref{F4})) we conclude that $X$ is a projective $\End_R(M)$-module.  This shows that $\refl_R\End_R(M)=\proj\End_R(M)$. By \ref{facts3and4}(\ref{F3}), this in turn implies that $\refl R=\add M$.

Repeating the argument with $\End_R(N)$ shows that $\refl R=\add N$, so combining we see that $\add M=\add N$.  From here it is standard that $\End_{R}(M)$ and $\End_{R}(N)$ are Morita equivalent, via the progenerator $\Hom_R(M,N)$.
\end{proof}
We remark that all these theorems hold in the non-local setting, provided that we additionally assume that $R$ is equicodimensional (i.e.\ $\dim R_\m=\dim R$ for all $\m\in\Max R$).  This assumption allows us to reduce to the local case without the dimension dropping, and so the global--local arguments work nicely.

\begin{remark}
Theorem~\ref{NCCR dim 2} only gives uniqueness, it does not give existence.  Indeed, NCCRs do not exist for all local CM normal domain of dimension $2$, since as a consequence of \ref{thm Auslander} if such an $R$ admits an NCCR, necessarily it must have finite CM type.  If we work over $\mathbb{C}$, another theorem of Auslander (see \cite[\S11]{Y}) says that the only such $R$ are the two-dimensional quotient singularities.    
\end{remark}

\subsection{NCCRs are Derived Equivalent in Dimension 3}  In dimension three, the situation is more complicated, but can still be controlled.  In algebraic geometry, when passing from surfaces to 3-folds we (often) replace the idea of a minimal resolution by a crepant resolution, and these are definitely not unique up to isomorphism.  However, by a result of Bridgeland, all crepant resolutions of a given $\Spec R$ are unique up to derived equivalence.  Using this as motivation, we thus ask whether all NCCRs for a given $R$ are derived equivalent. 

As a first remark, this is the best that we can hope for.  In contrast to the previous subsection, NCCRs are definitely not unique up to Morita equivalence in dimension three, as the next example demonstrates.

\begin{example}
Consider the ring $R=\mathbb{C}[[u,v,x,y]]/(uv-x^2y^2)$.  In this example, although it might not be immediately obvious, we have that $\End_R(R\oplus (u,x)\oplus (u,xy)\oplus (u,xy^2))$ and $\End_R(R\oplus (u,x)\oplus (u,xy)\oplus (u,x^2y))$ are both NCCRs.  This can be proved in a variety of ways, for example using dimers and toric geometry \cite{Nag}, arguing directly with tilting bundles \cite[\S5]{IW5}, or by commutative algebra and Calabi--Yau reduction \cite{IW6}.  Regardless, the two NCCRs above can be presented as
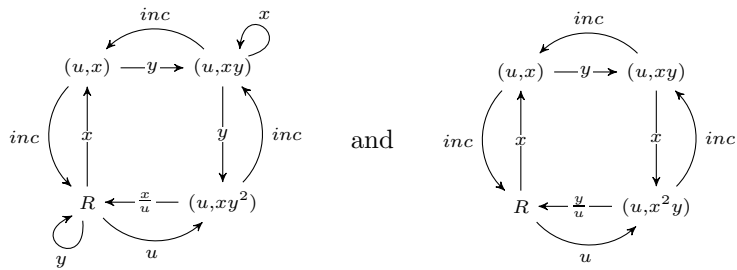
\[
\begin{array}{ccc}
\begin{array}{c}
{\begin{tikzpicture} [bend angle=45,looseness=1]
\node[name=s,regular polygon, regular polygon sides=4, minimum size=2.5cm] at (0,0) {}; 
\node (C1) at (s.corner 1)  {$\scriptstyle (u,xy)$};
\node (C2) at (s.corner 2)  {$\scriptstyle (u,x)$};
\node (C4) at (s.corner 4)  {$\scriptstyle (u,xy^2)$};
\node (C3) at (s.corner 3)  {$\scriptstyle R$};
\draw[->] (C4) -- node[gap,pos=0.45] {$\scriptstyle \frac{x}{u}$} (C3); 
\draw[->] (C3) -- node[gap] {$\scriptstyle x$} (C2); 
\draw[->] (C2) -- node[gap] {$\scriptstyle y$} (C1); 
\draw[->] (C1) -- node[gap] {$\scriptstyle y$} (C4);
\draw [->,bend right] (C1) to node[above]  {$\scriptstyle inc$} (C2);
\draw [->,bend right] (C2) to node[left]  {$\scriptstyle inc$} (C3);
\draw [->,bend right] (C3) to node[below]  {$\scriptstyle u$} (C4);
\draw [->,bend right] (C4) to node[right]  {$\scriptstyle inc$} (C1);
\node (C1a) at ($(s.corner 1)+(30:7pt)$) {};
\draw[<-]  (C1a) edge [in=15,out=85,loop,looseness=12] node[above] {$\scriptstyle x$} (C1a);
\node (C3a) at ($(s.corner 3)+(-120:4pt)$) {};
\draw[<-]  (C3a) edge [in=-85,out=-155,loop,looseness=12] node[below] {$\scriptstyle y$} (C3a);
\end{tikzpicture}}\end{array}&
\textnormal{and}&
\begin{array}{c}
{\begin{tikzpicture} [bend angle=45,looseness=1]
\node[name=s,regular polygon, regular polygon sides=4, minimum size=2.5cm] at (0,0) {}; 
\node (C1) at (s.corner 1)  {$\scriptstyle (u,xy)$};
\node (C2) at (s.corner 2)  {$\scriptstyle (u,x)$};
\node (C4) at (s.corner 4)  {$\scriptstyle (u,x^2y)$};
\node (C3) at (s.corner 3)  {$\scriptstyle R$};
\draw[->] (C4) -- node[gap,pos=0.45] {$\scriptstyle \frac{y}{u}$} (C3); 
\draw[->] (C3) -- node[gap] {$\scriptstyle x$} (C2); 
\draw[->] (C2) -- node[gap] {$\scriptstyle y$} (C1); 
\draw[->] (C1) -- node[gap] {$\scriptstyle x$} (C4);
\draw [->,bend right] (C1) to node[above]  {$\scriptstyle inc$} (C2);
\draw [->,bend right] (C2) to node[left]  {$\scriptstyle inc$} (C3);
\draw [->,bend right] (C3) to node[below]  {$\scriptstyle u$} (C4);
\draw [->,bend right] (C4) to node[right]  {$\scriptstyle inc$} (C1);
\end{tikzpicture}}\end{array}
\end{array}
\]
respectively, and they are not Morita equivalent (for example, examine the Ext groups of the vertex simples).
\end{example}

As in the previous subsection, the following theorem holds in the more general setting where $R$ is equicodimensional, but the proof is given in the local case for simplicity.

\begin{thm}\label{NCCRs Db}
Suppose that $(R,\m)$ is a normal CM domain with a canonical module, such that $\dim R=3$.  Then all NCCRs of $R$ are derived equivalent.
\end{thm}

The proof has very little to do with derived categories, and can be understood without even knowing the definition, given the knowledge that classical tilting modules induce derived equivalences.  
\begin{defin}\label{tiltingdef}
Let $\Lambda$ be a ring.  Then $T\in\mod \Lambda$ is called a \emph{classical partial tilting module} if $\pd_{\Lambda}T\leq 1$ and $\Ext^{1}_{\Lambda}(T,T)=0$.  If further there exists an exact sequence
\[
0\to \Lambda\to T_{0}\to T_{1}\to 0
\]
with each $T_{i}\in\add T$, we say that $T$ is a \emph{classical tilting module}.
\end{defin}

If $T$ is a classical tilting $\Lambda$-module, it is standard that there is a derived equivalence between $\Lambda$ and $\End_\Lambda(T)$.   Using only this and the facts we have already, we can now prove \ref{NCCRs Db}.
\begin{proof}
Suppose that $\End_{R}(M)$ and $\End_{R}(N)$ are NCCRs of $R$.  Our strategy is to prove that $T:=\Hom_{R}(M,N)$ is a tilting $\Lambda:=\End_{R}(M)$-module.  By the remark above, this then shows that $\End_R(M)$ and $\End_{\End_R(M)}(T)$ are derived equivalent.  Since $\End_{\End_R(M)}(T)\cong \End_R(N)$ by fact \ref{facts3and4}(\ref{F3}), this will then show that $\End_R(M)$ and $\End_R(N)$ are derived equivalent.

\smallskip
\noindent
(1) We first show  $\pd_{\Lambda}T\leq 1$.  This is really just Auslander--Buchsbaum.  Take $R^a\to R^b\to M\to 0$ and apply $\Hom_R(-,N)$ to obtain
\begin{eqnarray}
0\to T\to \Hom_R(R,N)^b\to \Hom_R(R,N)^a\label{T depth}
\end{eqnarray}
Since $\depth N\geq 2$ (since $N$ is reflexive), by the depth lemma applied to (\ref{T depth}) we have $\depth T\geq  2$. Hence by Auslander--Buchsbaum (\ref{facts3and4}(\ref{F4})) $\pd_{\Lambda}T\leq 1$. 

\smallskip 
\noindent
(2) We next show that $\Ext^1_\Lambda(T,T)=0$.  This is really just the depth lemma, and using localization to induct.  For all primes $\p$ with $\hgt\p=2$, $T_{\p}\in\CM R_{\p}$ (since $T_{\p}\in\refl R_{\p}$, but for Gorenstein surfaces $\refl R_\p=\CM R_\p$, as in \S\ref{L1}).  Hence by Auslander-Buchsbaum applied to $\Lambda_\p$, it follows that $T_{\p}$ is a projective $\Lambda_{\p}$-module for all such primes.  This in turn shows that $R$-module $\Ext^{1}_{\Lambda}(T,T)$ is supported only on the maximal ideal, hence has finite length.  In particular, provided that $\Ext^{1}_{\Lambda}(T,T)$ is non-zero, there is an injection $R/\m\hookrightarrow \Ext^{1}_{\Lambda}(T,T)$ and so by the definition of depth, necessarily $\depth\Ext^{1}_{\Lambda}(T,T)=0$.  But on the other hand $\End_{\Lambda}(T)\cong\End_{R}(N)\in\CM R$ (the isomorphism is as above, by \ref{facts3and4}(\ref{F3})) and so the depth lemma applied to 
\[
0\to \Hom_\Lambda(T,T)\cong\End_R(N)\to\Hom_\Lambda(\Lambda^a,T)\to\Hom_\Lambda(\Omega T,T)\to \Ext^1_\Lambda(T,T)\to 0
\]
actually forces $\depth \Ext^{1}_{\Lambda}(T,T)>0$.  This is a contradiction, unless $\Ext^{1}_{\Lambda}(T,T)=0$.

\smallskip
\noindent
(3) We lastly show that there is an exact sequence $0\to \Lambda\to T_{0}\to T_{1}\to 0$ with each $T_{i}\in\add T$.  This  involves a duality trick.  Denote $(-)^*:=\Hom_R(-,R)$, then certainly $\Gamma:=\End_{R}(N^{*})$ is also a NCCR.  

Consider a projective $\Gamma$-module $P$ surjecting as $P\stackrel{\psi}{\twoheadrightarrow}\mathbb{F}M^{*}$, where $\mathbb{F}=\Hom_{R}(N^{*},-)$.  By reflexive equivalence, we know that $P=\mathbb{F}N_0^*$ for some $N_0^*\in\add N^*$, and further $\psi=\mathbb{F}f$ for some $f:N_0^*\to M^*$.  Taking the kernel of $f$, this all means that we have an exact sequence
\[
0\to K\to N_{0}^{*}\stackrel{f}{\to} M^{*}
\]
such that
\begin{eqnarray}
0\to\mathbb{F}K\to \mathbb{F}N_{0}^{*}\to\mathbb{F}M^{*}\to 0\label{rem2}
\end{eqnarray}
is exact.   By the depth lemma $\mathbb{F}K\in\CM R$ and so by Auslander-Buchsbaum $\mathbb{F}K$ is a projective $\Gamma$-module.  Thus $K\in\add N^{*}$ by reflexive equivalence; say $K=N_{1}^{*}$. 

Now $\Ext^{1}_{\Gamma}(\mathbb{F}M^{*},\mathbb{F}M^{*})=0$ by applying the argument in (2) to $\Gamma$.  Thus applying $\Hom_{\Gamma}(-,\mathbb{F}M^{*})$ to (\ref{rem2}) gives us the following commutative diagram
\[
{\SelectTips{cm}{10}
\xy%0;/r.85pc/:
(0,0)*+{0}="0",
(20,0)*+{\Hom_{\Gamma}(\mathbb{F}M^{*},\mathbb{F}M^{*})}="1",
(55,0)*+{\Hom_{\Gamma}(\mathbb{F}N_{0}^{*},\mathbb{F}M^{*})}="2",
(90,0)*+{\Hom_{\Gamma}(\mathbb{F}N_{1}^{*},\mathbb{F}M^{*})}="3",
(110,0)*+{0}="4",
(0,-10)*+{0}="0a",
(20,-10)*+{\Hom_{R}(M^{*},M^{*})}="1a",
(55,-10)*+{\Hom_{R}(N_{0}^{*},M^{*})}="2a",
(90,-10)*+{\Hom_{R}(N_{1}^{*},M^{*})}="3a",
(110,-10)*+{0}="4a",
\ar"0";"1"
\ar"1";"2"
\ar"2";"3"
\ar"3";"4"
\ar"0a";"1a"
\ar"1a";"2a"
\ar"2a";"3a"
\ar@{=}"1";"1a"
\ar@{=}"2";"2a"
\ar@{=}"3";"3a"
\ar"3a";"4a"
\endxy}
\]
where the top row is exact.  Hence the bottom row is exact.  Since $(-)^{*}:\refl R\to \refl R$ is a duality, this means that
\[
0\to\Hom_{R}(M,M)\to\Hom_{R}(M,N_{0})\to\Hom_{R}(M,N_{1})\to 0
\]
is exact.  But this is simply 
\[
0\to\Lambda\to T_{0}\to T_{1}\to 0
\]
with each $T_{i}\in \add T$.  Hence $T$ is a tilting $\Lambda$-module.
\end{proof}

\medskip
\medskip
\noindent
{\bf Credits:} The material on depth and maximal CM modules is well-known and can be found in most commutative algebra books.  The definition of an NCCR is due to Van den Bergh \cite{VdBNCCR}, modelled both on the skew group ring and also on his interpretation \cite{VdB1d} of the flops paper of Bridgeland  \cite{Bridgeland}.  The idea that finite global dimension is not enough and we need homogeneity appears in the `homologically homogeneous' rings of Brown--Hajarnavis \cite{BrownHar}, and also as `non-singular orders' in the language of Auslander \cite{Aus78,Aus2}.

Reflexive equivalence is also well--known, it appears in Reiten--Van den Bergh \cite{RV89}.  The Auslander--Buchsbaum formula in the commutative setting is much more general than the version presented here, and first appeared in \cite{AB}.  The proof can be found in most commutative algebra or homological algebra books.  The noncommutative version of the Auslander--Buchsbaum formula for NCCRs in \ref{facts3and4}(\ref{F4})(a) was first established by Iyama--Reiten \cite{IR}, whereas the version presented in \ref{facts3and4}(\ref{F4})(b), which  is valid in the infinite global dimension case, is taken from \cite{IW3}.  The `correct' setting is that of a Gorenstein $R$-order.

Lemma~\ref{gl ok} appears in \cite{VdBNCCR}, but it is also explained by Iyama--Reiten \cite{IR} and Dao--Huneke \cite{DH}.  The fact that NCCRs are unique in dimension two was well--known to experts, but is only written down in \cite{IW4}.  The fact that NCCRs are all derived equivalent in dimension three when the base ring $R$ is Gorenstein is due to Iyama--Reiten \cite{IR}, but presented here is the simplified proof from \cite{IW4} since it holds in the more general CM setting.

\subsection{Exercises}

\begin{ex}\label{Ex2.16}
(Common examples and counterexamples for surfaces). Consider the following quivers with relations $(Q,R)$
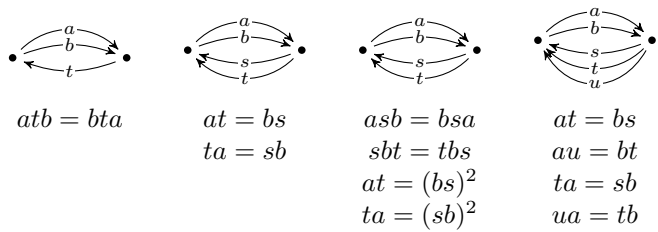
\[
\begin{array}{cccc}
\begin{array}{c}
\begin{tikzpicture} [bend angle=45, looseness=1]
\node (C1) at (0,0) [vertex] {};
\node (C2) at (1.5,0)  [vertex] {};
\draw [->,bend left] (C1) to node[gap]  {$\scriptstyle a$} (C2);
\draw [->,bend left=20,looseness=1] (C1) to node[gap]  {$\scriptstyle b$} (C2);
\draw [->,bend left=20,looseness=1] (C2) to node[gap]  {$\scriptstyle t$} (C1);
\end{tikzpicture}
\end{array}
&
\begin{array}{c}
\begin{tikzpicture} [bend angle=45, looseness=1]
\node (C1) at (0,0) [vertex] {};
\node (C2) at (1.5,0)  [vertex] {};
\draw [->,bend left] (C1) to node[gap]  {$\scriptstyle a$} (C2);
\draw [->,bend left=20,looseness=1] (C1) to node[gap]  {$\scriptstyle b$} (C2);
\draw [->,bend left] (C2) to node[gap]  {$\scriptstyle t$} (C1);
\draw [->,bend left=20,looseness=1] (C2) to node[gap]  {$\scriptstyle s$} (C1);
\end{tikzpicture}
\end{array}
&
\begin{array}{c}
\begin{tikzpicture} [bend angle=45, looseness=1]
\node (C1) at (0,0) [vertex] {};
\node (C2) at (1.5,0)  [vertex] {};
\draw [->,bend left] (C1) to node[gap]  {$\scriptstyle a$} (C2);
\draw [->,bend left=20,looseness=1] (C1) to node[gap]  {$\scriptstyle b$} (C2);
\draw [->,bend left] (C2) to node[gap]  {$\scriptstyle t$} (C1);
\draw [->,bend left=20,looseness=1] (C2) to node[gap]  {$\scriptstyle s$} (C1);
\end{tikzpicture}
\end{array}
&
\begin{array}{c}
\begin{tikzpicture} [bend angle=45, looseness=1]
\node (C1) at (0,0) [vertex] {};
\node (C2) at (1.5,0)  [vertex] {};
\draw [->,bend left] (C1) to node[gap]  {$\scriptstyle a$} (C2);
\draw [->,bend left=20,looseness=1] (C1) to node[gap]  {$\scriptstyle b$} (C2);
\draw [->,bend left] (C2) to node[gap]  {$\scriptstyle t$} (C1);
\draw [->,bend left=20,looseness=1] (C2) to node[gap]  {$\scriptstyle s$} (C1);
\draw [->,bend left=60,looseness=1.3] (C2) to node[gap]  {$\scriptstyle u$} (C1);
\end{tikzpicture}
\end{array}
\\
atb=bta
&
at=bs
&
asb=bsa
&
at=bs
\\
&
ta=sb
&
sbt=tbs
&
au=bt
\\
&
&
at=(bs)^2
&
ta=sb
\\
&
&
ta=(sb)^2
&
ua=tb
\end{array}
\]
For each $\Lambda:=kQ/R$, 
\begin{enumerate}
\item Determine the centre $Z(\Lambda)$. Is it CM?  Is it Gorenstein? Is it smooth?
\item Is $\Lambda\cong\End_{Z(\Lambda)}(M)$ for some $M\in Z(\Lambda)$?  If so, is $M\in\CM Z(\Lambda)$?
\item Is $\Lambda\in\CM Z(\Lambda)$?
\item What are the projective dimension of the vertex simples?
\item In the situation when the vertex simples have infinite projective dimension, is there anything remarkable about their projective resolutions?
\item Using (2) and (4), compute $\gl\Lambda$.
\item Using (2) and (6), which $\Lambda$ are NCCRs over $Z(\Lambda)$?
\item (harder) To which spaces are the $\Lambda$ derived equivalent? (aside:\ does this explain (5)?)
\end{enumerate}
\end{ex}

\begin{ex}\label{Ex2.17}
(Example of a non-local NCCR in dimension two). Consider $(Q,R)$ given by
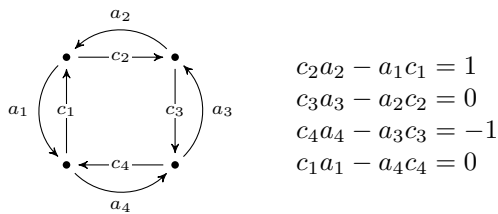
\[
\begin{array}{cc}
\begin{array}{c}
{\begin{tikzpicture} [bend angle=45,looseness=1]
\node[name=s,regular polygon, regular polygon sides=4, minimum size=2cm] at (0,0) {}; 
\node (C1) at (s.corner 1) [vertex] {};
\node (C2) at (s.corner 2) [vertex] {};
\node (C4) at (s.corner 4) [vertex] {};
\node (C3) at (s.corner 3) [vertex] {};
\draw[->] (C4) -- node[gap] {$\scriptstyle c_4$} (C3); 
\draw[->] (C3) -- node[gap] {$\scriptstyle c_1$} (C2); 
\draw[->] (C2) -- node[gap] {$\scriptstyle c_2$} (C1); 
\draw[->] (C1) -- node[gap] {$\scriptstyle c_3$} (C4);
\draw [->,bend right] (C1) to node[above]  {$\scriptstyle a_2$} (C2);
\draw [->,bend right] (C2) to node[left]  {$\scriptstyle a_1$} (C3);
\draw [->,bend right] (C3) to node[below]  {$\scriptstyle a_4$} (C4);
\draw [->,bend right] (C4) to node[right]  {$\scriptstyle a_3$} (C1);
\end{tikzpicture}}
\end{array}&
\begin{array}{l}
c_2a_2-a_1c_1=1\\
c_3a_3-a_2c_2=0\\
c_4a_4-a_3c_3=-1\\
c_1a_1-a_4c_4=0
\end{array}
\end{array}
\]
Denote $\Lambda:=kQ/R$ and let $Z:=\mathbb{C}[u,v,x]/(uv-x^2(x-1)^2)$.  Note that $Z$ has only two singular points, which locally are just the $\frac{1}{2}(1,1)$ surface singularity.
\begin{enumerate}
\item Show that $kQ/R\cong \End_Z(Z\oplus (u,x-1)\oplus (u,x(x-1))\oplus (u,x^2(x-1)) )$.
\item Using (1), deduce that $\Lambda$ is a NCCR.
\item Find some algebras that are Morita equivalent to $\Lambda$. 
\end{enumerate}
\end{ex}

\begin{ex}\label{Ex2.18}
(Left and right modules can matter) Consider the algebra $\Lambda$ given by
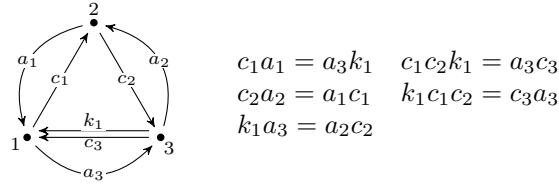
\[
\begin{array}{cc}
\begin{array}{c}
\begin{tikzpicture} [bend angle=45, looseness=1.2]
%\draw[blue,densely dashed] (0,0) circle(1cm); 
\node[name=s,regular polygon, regular polygon sides=3, minimum size=2cm] at (0,0) {}; 
\node (C1) at (s.corner 1) [vertex] {};
\node (C2) at (s.corner 2) [vertex] {};
\node (C3) at (s.corner 3) [vertex] {};
\node (c1) at ($(C1) + (90:5pt)$) {$\scriptstyle 2$};
\node (c2) at ($(C2) + (-150:5pt)$) {$\scriptstyle 1$};
\node (c3) at ($(C3) + (-50:5pt)$) {$\scriptstyle 3$};
%standard x arrows that don't need to move
\draw[->] (C3) -- node[gap, below=0pt] {$\scriptstyle c_3$} (C2);
\draw[->] (C2) -- node[gap] {$\scriptstyle c_1$} (C1);
\draw[->] (C1) -- node[gap] {$\scriptstyle c_2$} (C3);
%standard y arrows
\draw [->,bend right] (C1) to node[gap] {$\scriptstyle a_1$} (C2);
\draw [->,bend right] (C2) to node[gap] {$\scriptstyle a_3$} (C3);
\draw [->,bend right] (C3) to node[gap] {$\scriptstyle a_2$} (C1);
%extra arrows plus black one that need to move
\draw[->] (C3)++(150:5pt) -- node[gap, above=0pt] {$\scriptstyle k_1$} ($(C2) + (30:5pt)$);
\end{tikzpicture} 
\end{array}
& 
\begin{array}{ll}
c_1a_1=a_3k_1 & c_1c_2k_1=a_3c_3\\
c_2a_2=a_1c_1 &k_1c_1c_2=c_3a_3\\
k_1a_3=a_2c_2&
\end{array}
\end{array}
\]
This should be familiar from Exercise~\ref{Ex1.13}.  
Denote the vertex simples by $S_1$, $S_2$ and $S_3$.
\begin{enumerate}
\item Show that as left $\Lambda$-modules $\pd S_1=\pd S_2=2$ whilst $\pd S_3=3$.
\item Show that as right $\Lambda$-modules $\pd S_1=3$ whilst $\pd S_2=\pd S_3=2$.
\item In this example $\Lambda\cong \End_{Z(\Lambda)}(M)$ where $Z(\Lambda)$ is a CM ring, and $M\in\CM Z(\Lambda)$.  Why does this not contradict Auslander--Buchsbaum?
\end{enumerate}
\end{ex}

\begin{ex}\label{Ex2.19}
(Common examples for 3-folds). Consider the following quivers with relations $(Q,R)$
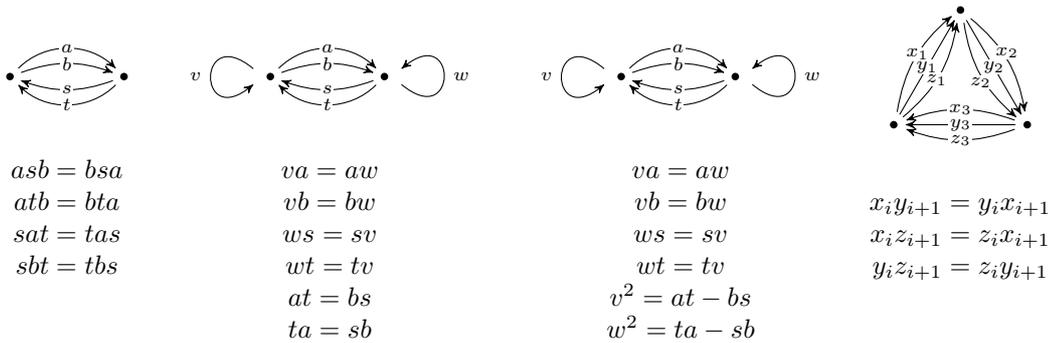
\[
\begin{array}{cccc}
&\\
\begin{array}{c}
\begin{tikzpicture} [bend angle=45, looseness=1]
\node (C1) at (0,0) [vertex] {};
\node (C2) at (1.5,0)  [vertex] {};
\draw [->,bend left] (C1) to node[gap]  {$\scriptstyle a$} (C2);
\draw [->,bend left=20,looseness=1] (C1) to node[gap]  {$\scriptstyle b$} (C2);
\draw [->,bend left] (C2) to node[gap]  {$\scriptstyle t$} (C1);
\draw [->,bend left=20,looseness=1] (C2) to node[gap]  {$\scriptstyle s$} (C1);
\end{tikzpicture}
\end{array}
& 
\begin{array}{c}
\begin{tikzpicture} [bend angle=45, looseness=1]
\node (C1) at (0,0) [vertex] {};
\node (C2) at (1.5,0)  [vertex] {};
\node (C1a) at (-0.1,0)  {};
\node (C2a) at (1.6,0) {};
\draw [->,bend left] (C1) to node[gap]  {$\scriptstyle a$} (C2);
\draw [->,bend left=20,looseness=1] (C1) to node[gap]  {$\scriptstyle b$} (C2);
\draw [->,bend left] (C2) to node[gap]  {$\scriptstyle t$} (C1);
\draw [->,bend left=20,looseness=1] (C2) to node[gap]  {$\scriptstyle s$} (C1);
\draw[<-]  (C1a) edge [in=140,out=-140,loop,looseness=12] node[left] {$\scriptstyle v$} (C1a);
\draw[->]  (C2a) edge [in=40,out=-40,loop,looseness=12] node[right] {$\scriptstyle w$} (C2a);
\end{tikzpicture}
\end{array}
&
\begin{array}{c}
\begin{tikzpicture} [bend angle=45, looseness=1]
\node (C1) at (0,0) [vertex] {};
\node (C2) at (1.5,0)  [vertex] {};
\node (C1a) at (-0.1,0)  {};
\node (C2a) at (1.6,0) {};
\draw [->,bend left] (C1) to node[gap]  {$\scriptstyle a$} (C2);
\draw [->,bend left=20,looseness=1] (C1) to node[gap]  {$\scriptstyle b$} (C2);
\draw [->,bend left] (C2) to node[gap]  {$\scriptstyle t$} (C1);
\draw [->,bend left=20,looseness=1] (C2) to node[gap]  {$\scriptstyle s$} (C1);
\draw[<-]  (C1a) edge [in=140,out=-140,loop,looseness=12] node[left] {$\scriptstyle v$} (C1a);
\draw[->]  (C2a) edge [in=40,out=-40,loop,looseness=12] node[right] {$\scriptstyle w$} (C2a);
\end{tikzpicture}
\end{array}
&
\begin{array}{c}
\begin{tikzpicture} [bend angle=20, looseness=1]
%\draw[blue,densely dashed] (0,0) circle(1cm); 
\node[name=s,regular polygon, regular polygon sides=3, minimum size=2cm] at (0,0) {}; 
\node (C1) at (s.corner 1) [vertex] {};
\node (C2) at (s.corner 2) [vertex] {};
\node (C3) at (s.corner 3) [vertex] {};

%extra arrows 
\draw[->, bend right] (C3)++(160:5pt) to node[gap] {$\scriptstyle x_3$} ($(C2) + (20:5pt)$);
\draw[->, bend left] (C3)++(-160:5pt) to node[gap] {$\scriptstyle z_3$} ($(C2) + (-20:5pt)$);
%extra arrows
\draw[->, bend left] (C2)++(75:5pt) to node[gap, pos=0.55] {$\scriptstyle x_1$} ($(C1) + (-145:4pt)$);
\draw[->, bend right] (C2)++(40:4.5pt) to node[gap, pos=0.45] {$\scriptstyle z_1$} ($(C1) + (-105:5pt)$);
%extra arrows
\draw[->, bend left] (C1)++(-40:5pt) to node[gap, above=0pt] {$\scriptstyle x_2$} ($(C3) + (100:5pt)$);
\draw[->, bend right] (C1)++(-75:5pt) to node[gap, below=0pt] {$\scriptstyle z_2$} ($(C3) + (145:5pt)$);
%standard y arrows that don't need to move
\draw[->] (C3) -- node[gap] {$\scriptstyle y_3$} (C2);
\draw[->] (C2) -- node[gap] {$\scriptstyle y_1$} (C1);
\draw[->] (C1) -- node[gap] {$\scriptstyle y_2$} (C3);
\end{tikzpicture} 
\end{array}
\\
asb=bsa
&
va=aw
&
va=aw
&

\\
atb=bta
&
vb=bw
&
vb=bw
&
x_iy_{i+1}=y_ix_{i+1}
\\
sat=tas
&
ws=sv
&
ws=sv
&
x_iz_{i+1}=z_ix_{i+1}
\\
sbt=tbs
&
wt=tv
&
wt=tv
&
y_iz_{i+1}=z_iy_{i+1}
\\
&
at=bs
&
v^2=at-bs
&
\\
&
ta=sb
&
w^2=ta-sb
&\\
&
\end{array}
\]
where in the last example the relations are taken over all $1\leq i\leq 3$, with the subscripts taken mod 3 if necessary.  For each $\Lambda=kQ/R$, 
\begin{enumerate}
\item Check that the relations can be packaged as a superpotential.
\item  Determine the centre $Z(\Lambda)$.  It should be a three-dimensional Gorenstein ring (if necessary, check using Singular).
\item (harder) Show that $\Lambda$ is an NCCR over $Z(\Lambda)$.
\item (much harder) To which space(s) are the $\Lambda$ derived equivalent? (aside:\ the first three examples capture a certain geometric phenomenon regarding curves in 3-folds.  Which phenomenon?) 
\end{enumerate}
\end{ex}

\noindent {\bf Computer Exercises:}
\begin{ex}\label{Ex2.20}
(CM via Ext groups). Let $M,N\in\CM R$, where $R$ is a CM ring. This exercise will explore whether the property $\Hom_R(M,N)\in\CM R$ can be characterized in terms of Ext groups.  To calculate the depth of a $\Ext^1_R(M,N)$ and $\Ext^2_R(M,N)$, use 
\[
\begin{array}{l}
\tt{ > depth(Ext(1,M,M));} \\
\tt{ > depth(Ext(2,M,N));} 
\end{array}
\]
\begin{enumerate}
\item  We restrict to $\dim R=3$ with the assumption that $R$ is CM, with an isolated singularity. 
\begin{enumerate}
\item Let $R=\mathbb{C}[[u,v,x,y]]/(uv-xy)$. Consider the modules $R$, $(u,x)$, $(u,y)$, $R\oplus (u,x)$, $R\oplus (u,y)$ and $R\oplus (u,x)\oplus (u,y)$.  Check they are all CM.  For each $M$, compute whether $\End_R(M)\in\CM R$, and compute $\Ext^1_R(M,M)$.
\item Let $R=\mathbb{C}[[u,v,x,y]]/(uv-(x+y)(x+2y)(x+3y))$. To ease notation set $f_a=x+ay$.  Consider the modules $R$, $(u,f_1)$, $(u,f_2)$, $(u,f_2)$, $(u,f_1f_2)$, $(u,f_1f_3)$, $(u,f_2f_3)$ and all direct sum combinations. All are CM. For each $M$, compute whether $\End_R(M)\in\CM R$, and compute $\Ext^1_R(M,M)$.
\end{enumerate}
\item  Now $\dim R=3$ with $R$ CM, but the singular locus is no longer isolated (for example check using Singular). 
\begin{enumerate}
\item Let $R=\mathbb{C}[[u,v,x,y]]/(uv-x^2y)$.  Consider the modules $R$, $(u,x)$, $(u,y)$, $(u,x^2)$, $(u,xy)$ and all direct sum combinations.  All are CM.  For each $M$, compute whether $\End_R(M)\in\CM R$, and compute $\Ext^1_R(M,M)$.
\end{enumerate}
\item Is there a pattern from (1) and (2)?  Prove this relationship --- it should really only involve the depth lemma.
\end{enumerate}
\end{ex}

\begin{ex}\label{Ex2.21}
(How to determine invariant rings).  Computing invariant rings is generally quite a grim task. This exercise will show how to get a computer to (i) compute the generators of the invariant ring (ii) compute all the relations between them.
\begin{enumerate}
\item Consider $G=\left\langle g:=\left( \begin{smallmatrix} -1 &0\\ 0&-1  \end{smallmatrix}\right)\right\rangle$.  The invariants are calculated using the code
\[
\begin{array}{l}
\tt{>LIB``finvar.lib";}\\
\tt{>ring\quad S=complex,(x,y),dp;}\\
\tt{>matrix A[2][2]=-1,0,0,-1;}\\
\tt{>list\quad L=group\_reynolds(A);}\\
\tt{>matrix\quad T=invariant\_algebra\_reynolds(L[1],1);}\\
\tt{>print(T);}
\end{array}
\]
The output should be $y^2,xy,x^2$, which we know generate.  To calculate this in terms of generators and relations, we code
\[
\begin{array}{l}
\tt{>string\quad newring=``E";}\\
\tt{>orbit\_variety(T,newring);}\\
\tt{>print(G);}\\
\tt{>basering;}
\end{array}
\]
The output is $\tt{y(2)^\wedge 2-y(1)*y(3)}$, which is the equation we already know.  Using this presentation, we can now plug it into Exercises~\ref{Ex1.14} and \ref{Ex1.16} and ask whether the invariant ring is CM, or even Gorenstein.
\item The next group is
\[
BD_8:=\left\langle \begin{pmatrix} i&0\\0 &-i \end{pmatrix},\begin{pmatrix} 0&1\\-1&0  \end{pmatrix}\right\rangle.
\]
To code two generators requires
\[
\begin{array}{l}
\tt{>LIB``finvar.lib";}\\
\tt{>ring\quad S=complex,(x,y),dp;}\\
\tt{>matrix\quad A[2][2]=i,0,0,-i;}\\
\tt{>matrix\quad B[2][2]=0,1,1,0;}\\
\tt{>list\quad L=group\_reynolds(A,B);}\\
\tt{>matrix\quad T=invariant\_algebra\_reynolds(L[1],1);}\\
\tt{>print(T);}
\end{array}
\]
By using the same method as in (1), find the generators and relations.  Try to prove this without using the computer.
\item Try the same question with the following groups.  
\begin{enumerate}
\item $\frac{1}{4}(1,1):=\left\langle \begin{pmatrix} i&0\\0 &i \end{pmatrix} \right\rangle$
\item $\frac{1}{4}(1,1,2):=\left\langle \begin{pmatrix} i&0&0\\0 &i&0\\0&0&-1 \end{pmatrix} \right\rangle$
\item $\mathbb{Z}_2\times\mathbb{Z}_2:=\left\langle \begin{pmatrix} -1&0&0\\0 &-1&0\\0&0&1 \end{pmatrix}, \begin{pmatrix} 1&0&0\\0 &-1&0\\0&0&-1 \end{pmatrix} \right\rangle$
\end{enumerate}
Is there a pattern as to when the invariant ring is Gorenstein, and when it is CM?
\end{enumerate}
\end{ex}

\newpage
\section{From Algebra to Geometry: Quiver GIT}\label{L3}
Section \ref{L1} and \ref{L2} contains only algebra, and uses geometry only to motivate some of the results.  In this section the process begins to reverse, and we will begin extracting geometry (and obtain geometric theorems) starting from NCCRs.

The setup is that $Z$ is a Gorenstein normal domain, and we continue our original motivation of trying to resolve $\Spec Z$ via the picture
\[
\begin{array}{c}
\begin{tikzpicture}
\node (1) at (0,2) {$\mathcal{M}$};
\node (1b) at (0,0) {$\Spec Z$};
\node (2) at (2,1) {$A$};
\draw[->,decorate, 
decoration={snake,amplitude=.4mm,segment length=2mm,post length=1mm}] (1b) to (2);
\draw[->,decorate, 
decoration={snake,amplitude=.4mm,segment length=2mm,post length=1mm}] (2) to (1);
\draw[->] (1) to (1b);  
\end{tikzpicture}
\end{array}
\]
The first arrow is the process of associating to $Z$ an NCCR $A:=\End_Z(M)$. We have changed notation to $Z$ (from $R$), since it is always the centre of $A$.  We remark that NCCRs do not exist in general, even for easy examples like $\mathbb{C}[[u,v,x,y]]/(uv-x(x^2+y^7))$.  The content of this section is to explain the second arrow in the above diagram, i.e.\ how to extract the geometry from the noncommutative ring $A$.
 
To simplify the exposition, although it is not strictly necessary, we will assume that we have written $A$ as a quiver with relations $A=kQ/R$ (see \S\ref{QuiverAppendix} for a brief overview).  We are going to define various moduli spaces of finite dimensional representations, and to do this requires geometric invariant theory (GIT).  

For a fixed dimension vector $\alpha$ we may consider all representations of $A=kQ/R$ with dimension vector $\alpha$, namely
\[
\mathcal{R}:=\Rep(A,\alpha)=\{ \t{representations of }A\,\, \t{of dimension } \alpha \}.
\]
This is an affine variety, so denote the co-ordinate ring by $k[\mathcal{R}]$.  The variety, and hence the co-ordinate ring, carries a natural action of $G:=\prod_{i\in Q_0} \t{GL}(\alpha_i,k)$, where $Q_0$ denotes the set of vertices and $\t{GL}(\alpha_i,k)$ denotes the group of invertible $\alpha_i\times\alpha_i$ matrices with entries in $k$.  The action is via conjugation; $g$ acts on an arrow $a$ as $g\cdot a=g_{t(a)}^{-1}ag_{h(a)}$.  It is actually an action by $\t{PGL}$, since the diagonal one-parameter subgroup $\Delta=\{ (\lambda 1,\cdots,\lambda 1):\lambda\in k^{*} \}$ acts trivially.  By linear algebra the \emph{isomorphism classes} of representations of $A=kQ/R$ are in natural one-to-one correspondence with the orbits of this action.  

To understand the space of isomorphism classes is normally an impossible problem (for example it might have wild quiver type), so we want to throw away some representations and take what is known as a GIT quotient.  The key point is that to make a GIT quotient requires an addition piece of data in the form of a character $\chi$ of $G$.  

The characters $\chi$ of $G=\prod_{i\in Q_0} \t{GL}(\alpha_i,k)$ are known to be just the powers of the determinants
\[
\chi(g)=\prod_{i\in Q_0}\det(g_i)^{\theta_i}
\]
for some collection of integers $\theta_i\in \mathbb{Z}^{Q_0}$. Since such a $\chi$ determines and is determined by the $\theta_i$, we usually denote $\chi$ by $\chi_{\theta}$.  We now consider the map
\[
\begin{array}{rcl}
\theta:\tt{fdmod}A&\rightarrow& \mathbb{Z}\\
 M&\mapsto&  \sum_{i\in Q_0}\theta_i\,\t{dim}M_i ,
\end{array}
\]
which is additive on short exact sequences and so induces a map $K_0(\tt{fdmod}A)\rightarrow \mathbb{Z}$.  

We assume that our character satisfies $\chi_{\theta}(\Delta)=\{ 1\}$ (for experts --- this is needed to use Mumford's numerical criterion in \cite[2.5]{King}).  It not too hard to see that this condition translates into $\sum_{i\in Q_0}\theta_i\alpha_i=0$ and so for these $\chi_{\theta}$, $\theta(M)=0$ whenever $M$ has dimension vector $\alpha$.

We arrive at the key definition \cite[1.1]{King}.
\begin{defin}\label{abelianstab}
Let $\mathcal{A}$ be an abelian category, and $\theta:K_0(\mathcal{A})\rightarrow \mathbb{Z}$ an additive function.  We call $\theta$ a character of $\mathcal{A}$.  An object $M\in\mathcal{A}$ is called $\theta$-semistable if $\theta(M)=0$ and every subobject $M^\prime\subseteq M$ satisfies $\theta(M^\prime)\geq 0$.  Such an object $M$ is called $\theta$-stable if the only subobjects $M^\prime$ with $\theta(M^\prime)=0$ are $M$ and $0$.  We call $\theta$ generic if every $M$ which is $\theta$-semistable is actually $\theta$-stable.
\end{defin}
For $A=kQ/R$ as before, we are interested in the above definition for the case $\mathcal{A}=\tt{fdmod}A$.  We shall see how this works in practice in the next section. The reason King gave the above definition is that it is equivalent to the other notion of stability from GIT, which we now describe.

As stated above, $\mathcal{R}$ is an affine variety with an action of a linearly reductive group $G=\prod_{i\in Q_0} \t{GL}(\alpha_i,k)$.  Since $G$ is reductive, we have a quotient
\[
\mathcal{R}\rightarrow \mathcal{R}{/\!\!/}G=\Spec k[\mathcal{R}]^G
\]  
which is dual to the inclusion $k[\mathcal{R}]^G\rightarrow k[\mathcal{R}]$.  The reductiveness of the group ensures that $k[\mathcal{R}]^G$ is a finitely generated $k$-algebra, and so $\Spec k[\mathcal{R}]^G$ is a variety, not just a scheme.  

To make a GIT quotient we have to add to this picture the extra data of $\chi$, some character of $G$.
\begin{defin}
$f\in k[\mathcal{R}]$ is a semi-invariant of weight $\chi$ if $f(g\cdot x)=\chi(g)f(x)$ for all $g\in G$ and all $x\in \mathcal{R}$.  We write the set of such $f$ as $\mathcal{R}^{G,\chi}$.  We define
\[
\mathcal{R}{/\!\!/}_{\!\!\chi} G:=\t{Proj}\left(\bigoplus_{n\geq 0} k[\mathcal{R}]^{G,\chi^n} \right)
\]
\end{defin}
\begin{defin}
(1)  $x\in\mathcal{R}$ is called $\chi$-semistable (in the sense of GIT) if there exists some semi-invariant $f$ of weight $\chi^n$ with $n>0$ such that $f(x)\neq 0$, otherwise $x\in\mathcal{R}$ is called unstable.\\
(2) $x\in\mathcal{R}$ is called $\chi$-stable (in the sense of GIT) if it is $\chi$-semistable, the $G$ orbit containing $x$ is closed in $\mathcal{R}^{ss}$ and further the stabilizer of $x$ is finite.
\end{defin}
The set of semistable points $\mathcal{R}^{ss}$ forms an open subset of $\mathcal{R}$; in fact we have a morphism
\[
q:\mathcal{R}^{ss}\rightarrow\mathcal{R}{/\!\!/}_{\!\!\chi} G
\]
which is a good quotient.  In fact $q$ is a geometric quotient on the stable locus $\mathcal{R}^{s}$, meaning that $\mathcal{R}^s{/\!\!/}_{\!\!\chi} G$ really is an orbit space. 

The point in the above discussion is the following result, which says that the two notions coincide.
\begin{thm} (King)
Let $M\in\t{Rep}(A,\alpha)=\mathcal{R}$, choose $\theta$ as in Definition~\ref{abelianstab}.  Then $M$ is $\theta$-semistable (in the categorical sense of Definition~\ref{abelianstab}) if and only if $M$ is $\chi_{\theta}$-semistable (in the sense of GIT).  The same holds replacing semistability with stability.  If $\theta$ is generic, then $\mathcal{R}{/\!\!/}_{\!\!\chi} G$ parametrizes the $\theta$-stable modules up to isomorphism.
\end{thm}

Thus we use the machinery from the GIT side to define for quivers the following:
\begin{defin}
For $A=kQ/R$ choose dimension vector $\alpha$ and character $\theta$ satisfying $\sum_{i\in Q_0}\alpha_i\theta_i=0$.  Denote $\t{Rep}(A,\alpha)=\mathcal{R}$ and  $G=\prod_{i\in Q_0} \t{GL}(\alpha_i,k)$.  We define 
\[
\c{M}^{ss}_{\theta}(A,\alpha):=\mathcal{R}{/\!\!/}_{\!\!\chi_\theta} G:=\t{Proj}\left(\bigoplus_{n\geq 0} k[\mathcal{R}]^{G,\chi^n} \right)
\]
and call it the moduli space of $\theta$-semistable representations of dimension vector $\alpha$.
 \end{defin} 
 This is by definition projective over the ordinary quotient $\mathcal{R}{/\!\!/} G=\Spec k[\mathcal{R}]^{G}$. Hence for example if $k[\mathcal{R}]^{G}=k$ then $\c{M}^{ss}_{\theta}(A,\alpha)$ is a projective variety, but in our setting this will not be the case.
 
\begin{remark}\label{remark warning for later}
If $\Spec Z$ is a singularity that we would like to resolve, ideally we would like the zeroth piece $k[\mathcal{R}]^G$ to be $Z$, since then the moduli space is projective over $\Spec Z$.  However, even in cases where we use NCCRs to resolve singularities, $k[\mathcal{R}]^G$ might not be $Z$ (see Exercises 3).
\end{remark}

Note that $\c{M}^{ss}_{\theta}(A,\alpha)$ may be empty, and in fact it is often very difficult to determine whether this is true or not.  In our NCCR setting, this won't be a problem since by choice of dimension vector later it will always contain the Azumaya locus.  Even when $\c{M}^{ss}_{\theta}(A,\alpha)$ is not empty, computing it explicitly can often be hard.

We remark that $\c{M}^{ss}_{\theta}(A,\alpha)$ is a moduli space in the strict sense that it represents a functor.  This functorial viewpoint is very important, but in these notes we gloss over it, and the other related technical issues.

\subsection{Examples}
We want to input an NCCR, and by studying the moduli we hope to output some crepant resolution.  The problem is that so far we have no evidence that this is going to work (!), so in this section we will explicitly compute some moduli spaces to check that this strategy is not entirely unreasonable.  

We assume that we have already presented our NCCR as $A=kQ/R$.  We follow the exposition from King:
\begin{quotation}
``To specify such a moduli space we must give a dimension vector $\alpha$ and a weight vector (or `character') $\theta$ satisfying $\sum_{i\in Q_0}\theta_i\alpha_i=0$.  The moduli space of $\theta$-stable $A$-modules of dimension vector $\alpha$ is then the parameter space for those $A$-modules which have no proper submodules with any dimension vector $\beta$ for which $\sum_{i\in Q_0}\theta_i\beta_i\leq 0$.''
\end{quotation}

Even although there are two choices needed to create a moduli space, namely $\alpha$ and $\theta$, for NCCRs the $\alpha$ is given naturally by the ranks of the reflexive modules that we have summed together to create the NCCR (see \S\ref{last section} later). 

Before computing examples with NCCRs, we begin with something easier.
\begin{example}\label{P1 GIT}
\t{Consider the quiver}
\[
\begin{array}{c}
\begin{tikzpicture} 
\node (C1) at (0,0) [vertex] {};
\node (C2) at (1,0) [vertex] {};
\draw [->] ($(C1)+(30:4pt)$) -- ($(C2)+(150:4pt)$);
\draw [->] ($(C1)+(-30:4pt)$) -- ($(C2)+(210:4pt)$);
\end{tikzpicture}
\end{array}
\]
with no relations.  Choose $\alpha=(1,1)$ and $\theta=(-1,1)$.  With these choices, since $\sum \theta_i\alpha_i=0$ we can form the moduli space.  Now a representation of dimension vector $\alpha=(1,1)$ is $\theta$-semistable by definition if $\theta(M^\prime)\geq 0$ for all subobjects $M^\prime$.  But the only possible subobjects in this example are of dimension vector $(0,0)$,  $(0,1)$ and $(1,0)$, and $\theta$ is $\geq 0$ on all but the last (in fact its easy to see that $\theta$ is generic in this example).  Thus a representation of dimension vector $(1,1)$ is $\theta$-semistable if and only if it has no submodules of dimension vector $(1,0)$.    Now take an arbitrary representation $M$ of dimension vector $(1,1)$  
\[
\begin{array}{c}
\begin{tikzpicture} 
\node at (-0.74,0) {$M=$};
\node (C1) at (0,0)  {$\mathbb{C}$};
\node (C2) at (1.1,0)  {$\mathbb{C}$};
\draw [->] ($(C1)+(25:6pt)$) -- node [gap] {$\scriptstyle a$} ($(C2)+(155:6pt)$);
\draw [->] ($(C1)+(-25:6pt)$) -- node [gap] {$\scriptstyle b$} ($(C2)+(205:6pt)$);
\end{tikzpicture}
\end{array}
\]
Notice that $M$ has a submodule of dimension vector $(1,0)$ if and only if $a=b=0$, since the diagram
\[
\begin{array}{c}
\begin{tikzpicture} 
\node (C1) at (0,0)  {$\mathbb{C}$};
\node (C2) at (1.1,0)  {$\mathbb{C}$};
\draw [->] ($(C1)+(25:6pt)$) -- node [gap] {$\scriptstyle a$} ($(C2)+(155:6pt)$);
\draw [->] ($(C1)+(-25:6pt)$) -- node [gap] {$\scriptstyle b$} ($(C2)+(205:6pt)$);
\node (D1) at (0,-1)  {$\mathbb{C}$};
\node (D2) at (1.1,-1)  {$0$};
\draw [->] ($(D1)+(25:6pt)$) -- node [gap] {$\scriptstyle 0$} ($(D2)+(155:6pt)$);
\draw [->] ($(D1)+(-25:6pt)$) -- node [gap] {$\scriptstyle 0$} ($(D2)+(205:6pt)$);
\draw[->] (D1) -- node[left] {$\scriptstyle \cong$} (C1);
\draw[->] (D2) -- node[right] {$\scriptstyle 0$} (C2);
\end{tikzpicture}
\end{array}
\]
must commute.   Thus by our choice of stability $\theta$,
\begin{center}
$M$ is $\theta$-semistable $\iff$ $M$ has no submodule of dim vector $(1,0)$ $\iff$ $a\neq 0$ or $b\neq 0$.
\end{center}
and so we see that the semistable objects parametrize $\mathbb{P}^1$ via the ratio $(a:b)$, so the moduli space is just $\mathbb{P}^1$.  Another way to see this: we have two open sets, one corresponding to $a\neq 0$ and the other to $b\neq 0$.  After changing basis we can set them to be the identity, and so we have
\[
\begin{array}{cc}
U_0=\{\begin{array}{c}
\begin{tikzpicture} 
\node (C1) at (0,0)  {$\mathbb{C}$};
\node (C2) at (1.1,0)  {$\mathbb{C}$};
\draw [->] ($(C1)+(25:6pt)$) -- node [gap] {$\scriptstyle 1$} ($(C2)+(155:6pt)$);
\draw [->] ($(C1)+(-25:6pt)$) -- node [gap] {$\scriptstyle b$} ($(C2)+(205:6pt)$);
\end{tikzpicture}
\end{array}\mid b\in\mathbb{C}\}&U_1=\{
\begin{array}{c}
\begin{tikzpicture} 
\node (C1) at (0,0)  {$\mathbb{C}$};
\node (C2) at (1.1,0)  {$\mathbb{C}$};
\draw [->] ($(C1)+(25:6pt)$) -- node [gap] {$\scriptstyle a$} ($(C2)+(155:6pt)$);
\draw [->] ($(C1)+(-25:6pt)$) -- node [gap] {$\scriptstyle 1$} ($(C2)+(205:6pt)$);
\end{tikzpicture}
\end{array}\mid a\in\mathbb{C}\}
\end{array}
\]
Now the gluing is given by, whenever $U_0\ni b\neq 0$
\[
U_0\ni b=\begin{array}{c}
\begin{tikzpicture} 
\node (C1) at (0,0)  {$\mathbb{C}$};
\node (C2) at (1.1,0)  {$\mathbb{C}$};
\draw [->] ($(C1)+(25:6pt)$) -- node [gap] {$\scriptstyle 1$} ($(C2)+(155:6pt)$);
\draw [->] ($(C1)+(-25:6pt)$) -- node [gap] {$\scriptstyle b$} ($(C2)+(205:6pt)$);
\end{tikzpicture}
\end{array}
\cong
\begin{array}{c}
\begin{tikzpicture} 
\node (C1) at (0,0)  {$\mathbb{C}$};
\node (C2) at (1.1,0)  {$\mathbb{C}$};
\draw [->] ($(C1)+(25:6pt)$) -- node [gap] {$\scriptstyle b^{-1}$} ($(C2)+(155:6pt)$);
\draw [->] ($(C1)+(-25:6pt)$) -- node [gap] {$\scriptstyle 1$} ($(C2)+(205:6pt)$);
\end{tikzpicture}
\end{array}=b^{-1}\in U_1
\]
which is evidently just $\mathbb{P}^1$.
\end{example}

Although the following is not an NCCR, it has already appeared on the example sheets (Common Examples 1).  Whenever resolving singularities, it is traditional to begin by blowing up the origin of $\mathbb{C}^2$.
\begin{example}\label{blowup GIT}
\t{Consider the quiver with relations
\[
\begin{array}{cc}
\begin{array}{c}
\begin{tikzpicture} [bend angle=45, looseness=1]
\node (C1) at (0,0) [vertex] {};
\node (C2) at (1.5,0)  [vertex] {};
\draw [->,bend left] (C1) to node[gap]  {$\scriptstyle a$} (C2);
\draw [->,bend left=20,looseness=1] (C1) to node[gap]  {$\scriptstyle b$} (C2);
\draw [->,bend left=20,looseness=1] (C2) to node[gap]  {$\scriptstyle t$} (C1);
\end{tikzpicture}
\end{array} &
atb=bta
\end{array}
\]
and again choose dimension vector $(1,1)$ and stability $\theta_0=(-1,1)$.  Exactly as above if
\[
M=\begin{array}{c}
\begin{tikzpicture} [bend angle=40, looseness=1]
\node (C1) at (0,0)  {$\mathbb{C}$};
\node (C2) at (1.5,0)  {$\mathbb{C}$};
\draw [->,bend left] (C1) to node[gap]  {$\scriptstyle a$} (C2);
\draw [->,bend left=20,looseness=1] (C1) to node[gap]  {$\scriptstyle b$} (C2);
\draw [->,bend left=20,looseness=1] (C2) to node[gap]  {$\scriptstyle t$} (C1);
\end{tikzpicture}
\end{array}
\]
then 
\begin{center}
$M$ is $\theta$-semistable $\iff$ $M$ has no submodule of dim vector $(1,0)$ $\iff$ $a\neq 0$ or $b\neq 0$.
\end{center}
For the first open set in the moduli $U_0$ (when $a\neq 0$): after changing basis so that $a=1$ we see that the open set is parameterized by the two scalars $b$ and $t$ subject to the single relation (substituting $a=1$ into the quiver relations) $tb=bt$.  But this always holds, thus the open set $U_0$ is just $\mathbb{C}^2$ with co-ordinates $b,t$.  We write this as $\mathbb{C}^2_{b,t}$.  Similarly for the other open set:
\[
\begin{array}{cc}
\begin{array}{c}
\begin{tikzpicture} [bend angle=40, looseness=1]
\node (C1) at (0,0)  {$\mathbb{C}$};
\node (C2) at (1.5,0)  {$\mathbb{C}$};
\draw [->,bend left] (C1) to node[gap]  {$\scriptstyle 1$} (C2);
\draw [->,bend left=20,looseness=1] (C1) to node[gap]  {$\scriptstyle b$} (C2);
\draw [->,bend left=20,looseness=1] (C2) to node[gap]  {$\scriptstyle t$} (C1);
\end{tikzpicture}
\end{array}
&
\begin{array}{c}
\begin{tikzpicture} [bend angle=40, looseness=1]
\node (C1) at (0,0)  {$\mathbb{C}$};
\node (C2) at (1.5,0)  {$\mathbb{C}$};
\draw [->,bend left] (C1) to node[gap]  {$\scriptstyle a$} (C2);
\draw [->,bend left=20,looseness=1] (C1) to node[gap]  {$\scriptstyle 1$} (C2);
\draw [->,bend left=20,looseness=1] (C2) to node[gap]  {$\scriptstyle t$} (C1);
\end{tikzpicture}
\end{array}
\\
&\\
U_0=\mathbb{C}^2_{b,t}&U_1=\mathbb{C}^2_{a,t}. 
\end{array}
\]
Now the gluing is given by, whenever $b\neq 0$
\[
U_0\ni(b,t)=\begin{array}{c}
\begin{tikzpicture} [bend angle=40, looseness=1]
\node (C1) at (0,0)  {$\mathbb{C}$};
\node (C2) at (1.5,0)  {$\mathbb{C}$};
\draw [->,bend left] (C1) to node[gap]  {$\scriptstyle 1$} (C2);
\draw [->,bend left=20,looseness=1] (C1) to node[gap]  {$\scriptstyle b$} (C2);
\draw [->,bend left=20,looseness=1] (C2) to node[gap]  {$\scriptstyle t$} (C1);
\end{tikzpicture}
\end{array}
\cong\begin{array}{c}
\begin{tikzpicture} [bend angle=40, looseness=1]
\node (C1) at (0,0)  {$\mathbb{C}$};
\node (C2) at (1.5,0)  {$\mathbb{C}$};
\draw [->,bend left] (C1) to node[gap]  {$\scriptstyle b^{-1}$} (C2);
\draw [->,bend left=20,looseness=1] (C1) to node[gap]  {$\scriptstyle 1$} (C2);
\draw [->,bend left=20,looseness=1] (C2) to node[gap]  {$\scriptstyle bt$} (C1);
\end{tikzpicture}
\end{array}
=(b^{-1},bt)\in U_1
\]
and so we see that this is just the blowup of the origin of $\mathbb{C}^2$.}
\end{example}

With the two above examples in hand, we can now compute an example where an NCCR gives a crepant resolution. There are many more examples in the exercises.  We begin in dimension two.

\begin{example}\label{Z3 GIT}
We return to the running example from \S\ref{L1} and \S\ref{L2}, namely $R:=\mathbb{C}[a,b,c]/(ab-c^3)$.  We now know that $\End_R(R\oplus (a,c)\oplus (a,c^2))$ is a NCCR, and we can believe (from \ref{quiver of Z3}) that it is presented as
\[
\begin{array}{cc}
&\\
\begin{array}{c}
\begin{tikzpicture} [bend angle=45, looseness=1.2]
\node[name=s,regular polygon, regular polygon sides=3, minimum size=2cm] at (0,0) {}; 
\node (C1) at (s.corner 1)  [vertex] {};
\node (C2) at (s.corner 2) [vertex] {};
\node (C3) at (s.corner 3) [vertex] {};
\draw[->] (C2) -- node  [gap] {$\scriptstyle c_1$} (C1); 
\draw[->] (C1) -- node  [gap] {$\scriptstyle c_2$} (C3);
\draw[->] (C3) -- node  [gap] {$\scriptstyle c_3$} (C2); 
\draw [->,bend right] (C1) to node [left] {$\scriptstyle a_1$} (C2);
\draw [->,bend right] (C2) to node [below] {$\scriptstyle a_3$} (C3);
\draw [->,bend right] (C3) to node [right] {$\scriptstyle a_2$} (C1);
\end{tikzpicture}
\end{array}
&
\begin{array}{c}
c_1a_1=a_3c_3\\
c_2a_2=a_1c_1\\
c_3a_3=a_2c_2
\end{array}
\end{array}
\]
We consider the dimension vector $(1,1,1)$ (corresponding to the ranks of the CM modules that we have summed), and stability $(-2,1,1)$, where the $-2$ sits at the bottom left vertex.  As above, for a module to be $\theta$-stable corresponds to there being, for every vertex $i$, a non-zero path from the bottom left vertex  to $i$.  Thus we have three open sets, corresponding to the following pictures 
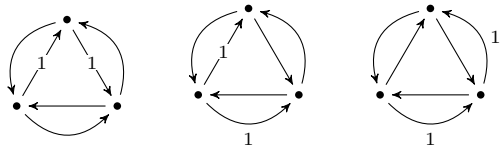
\[
\begin{array}{ccc}
\begin{array}{c}
\begin{tikzpicture} [bend angle=45, looseness=1.2]
\node[name=s,regular polygon, regular polygon sides=3, minimum size=1.5cm] at (0,0) {}; 
\node (C1) at (s.corner 1)  [vertex] {};
\node (C2) at (s.corner 2) [vertex] {};
\node (C3) at (s.corner 3) [vertex] {};
\draw[->] (C2) -- node  [gap] {$\scriptstyle 1$} (C1); 
\draw[->] (C1) -- node  [gap] {$\scriptstyle 1$} (C3);
\draw[->] (C3) --   (C2); 
\draw [->,bend right] (C1) to  (C2);
\draw [->,bend right] (C2) to  (C3);
\draw [->,bend right] (C3) to  (C1);
\end{tikzpicture}
\end{array}
&
\begin{array}{c}
\begin{tikzpicture} [bend angle=45, looseness=1.2]
\node[name=s,regular polygon, regular polygon sides=3, minimum size=1.5cm] at (0,0) {}; 
\node (C1) at (s.corner 1)  [vertex] {};
\node (C2) at (s.corner 2) [vertex] {};
\node (C3) at (s.corner 3) [vertex] {};
\draw[->] (C2) -- node  [gap] {$\scriptstyle 1$} (C1); 
\draw[->] (C1) --   (C3);
\draw[->] (C3) --   (C2); 
\draw [->,bend right] (C1) to  (C2);
\draw [->,bend right] (C2) to node [below] {$\scriptstyle 1$} (C3);
\draw [->,bend right] (C3) to  (C1);
\end{tikzpicture}
\end{array}
&
\begin{array}{c}
\begin{tikzpicture} [bend angle=45, looseness=1.2]
\node[name=s,regular polygon, regular polygon sides=3, minimum size=1.5cm] at (0,0) {}; 
\node (C1) at (s.corner 1)  [vertex] {};
\node (C2) at (s.corner 2) [vertex] {};
\node (C3) at (s.corner 3) [vertex] {};
\draw[->] (C2) --    (C1); 
\draw[->] (C1) --   (C3);
\draw[->] (C3) --    (C2); 
\draw [->,bend right] (C1) to  (C2);
\draw [->,bend right] (C2) to node [below] {$\scriptstyle 1$} (C3);
\draw [->,bend right] (C3) to node [right] {$\scriptstyle 1$} (C1);
\end{tikzpicture}
\end{array}
\end{array}
\]
Accounting for the relations, these open sets are parameterized by
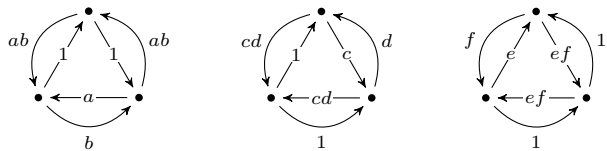
\[
\begin{array}{ccc}
\begin{array}{c}
\begin{tikzpicture} [bend angle=45, looseness=1.2]
\node[name=s,regular polygon, regular polygon sides=3, minimum size=1.5cm] at (0,0) {}; 
\node (C1) at (s.corner 1)  [vertex] {};
\node (C2) at (s.corner 2) [vertex] {};
\node (C3) at (s.corner 3) [vertex] {};
\draw[->] (C2) -- node  [gap] {$\scriptstyle 1$} (C1); 
\draw[->] (C1) -- node  [gap] {$\scriptstyle 1$} (C3);
\draw[->] (C3) -- node  [gap] {$\scriptstyle a$} (C2); 
\draw [->,bend right] (C1) to node [left] {$\scriptstyle ab$} (C2);
\draw [->,bend right] (C2) to node [below] {$\scriptstyle b$} (C3);
\draw [->,bend right] (C3) to node [right] {$\scriptstyle ab$} (C1);
\end{tikzpicture}
\end{array}
&
\begin{array}{c}
\begin{tikzpicture} [bend angle=45, looseness=1.2]
\node[name=s,regular polygon, regular polygon sides=3, minimum size=1.5cm] at (0,0) {}; 
\node (C1) at (s.corner 1)  [vertex] {};
\node (C2) at (s.corner 2) [vertex] {};
\node (C3) at (s.corner 3) [vertex] {};
\draw[->] (C2) -- node  [gap] {$\scriptstyle 1$} (C1); 
\draw[->] (C1) -- node  [gap] {$\scriptstyle c$} (C3);
\draw[->] (C3) -- node  [gap] {$\scriptstyle cd$} (C2); 
\draw [->,bend right] (C1) to node [left] {$\scriptstyle cd$} (C2);
\draw [->,bend right] (C2) to node [below] {$\scriptstyle 1$} (C3);
\draw [->,bend right] (C3) to node [right] {$\scriptstyle d$} (C1);
\end{tikzpicture}
\end{array}
&
\begin{array}{c}
\begin{tikzpicture} [bend angle=45, looseness=1.2]
\node[name=s,regular polygon, regular polygon sides=3, minimum size=1.5cm] at (0,0) {}; 
\node (C1) at (s.corner 1)  [vertex] {};
\node (C2) at (s.corner 2) [vertex] {};
\node (C3) at (s.corner 3) [vertex] {};
\draw[->] (C2) -- node  [gap] {$\scriptstyle e$} (C1); 
\draw[->] (C1) -- node  [gap] {$\scriptstyle ef$} (C3);
\draw[->] (C3) -- node  [gap] {$\scriptstyle ef$} (C2); 
\draw [->,bend right] (C1) to node [left] {$\scriptstyle f$} (C2);
\draw [->,bend right] (C2) to node [below] {$\scriptstyle 1$} (C3);
\draw [->,bend right] (C3) to node [right] {$\scriptstyle 1$} (C1);
\end{tikzpicture}
\end{array}
\end{array}
\]
The first, $U_1$, is just affine space $\mathbb{A}^2$ with coordinates $a,b$ (written $\mathbb{A}_{a,b}^2$), whilst the second and third are  $U_2=\mathbb{A}_{c,d}^2$ and $U_3=\mathbb{A}_{e,f}^2$ respectively.  We immediately see that our moduli space is smooth, since it is covered by three affine opens, each of which is smooth.  

Now we ask how these open sets glue.  Visually, it is clear that $U_1$ and $U_2$ glue if and only if the arrow $b$ in $U_1$ is not equal to zero (in which case we can base change to make it the identity, and hence land in $U_2$).  Similarly $U_2$ and $U_3$ glue if and only if the arrow $d$ in $U_2$ is nonzero.  In principle there could also be a glue between $U_1$ and $U_3$, but for these to glue certainly the arrows $b$ and $ab$ in $U_1$ must be nonzero, hence the glue between $U_1$ and $U_3$ is already covered by the previous two glues.  Explicitly, the two glues are
\[
\mathbb{A}^2_{a,b}\ni(a,b)=
\begin{array}{c}
\begin{tikzpicture} [bend angle=45, looseness=1.2]
\node[name=s,regular polygon, regular polygon sides=3, minimum size=1.5cm] at (0,0) {}; 
\node (C1) at (s.corner 1)  [vertex] {};
\node (C2) at (s.corner 2) [vertex] {};
\node (C3) at (s.corner 3) [vertex] {};
\draw[->] (C2) -- node  [gap] {$\scriptstyle 1$} (C1); 
\draw[->] (C1) -- node  [gap] {$\scriptstyle 1$} (C3);
\draw[->] (C3) -- node  [gap] {$\scriptstyle a$} (C2); 
\draw [->,bend right] (C1) to node [left] {$\scriptstyle ab$} (C2);
\draw [->,bend right] (C2) to node [below] {$\scriptstyle b$} (C3);
\draw [->,bend right] (C3) to node [right] {$\scriptstyle ab$} (C1);
\end{tikzpicture}
\end{array}
\cong
\begin{array}{c}
\begin{tikzpicture} [bend angle=45, looseness=1.2]
\node[name=s,regular polygon, regular polygon sides=3, minimum size=1.5cm] at (0,0) {}; 
\node (C1) at (s.corner 1)  [vertex] {};
\node (C2) at (s.corner 2) [vertex] {};
\node (C3) at (s.corner 3) [vertex] {};
\draw[->] (C2) -- node  [gap] {$\scriptstyle 1$} (C1); 
\draw[->] (C1) -- node  [gap] {$\scriptstyle b^{-1}$} (C3);
\draw[->] (C3) -- node  [gap] {$\scriptstyle ab$} (C2); 
\draw [->,bend right] (C1) to node [left] {$\scriptstyle ab$} (C2);
\draw [->,bend right] (C2) to node [below] {$\scriptstyle 1$} (C3);
\draw [->,bend right] (C3) to node [right] {$\scriptstyle ab^2$} (C1);
\end{tikzpicture}
\end{array}
=(b^{-1},ab^2)
\in \mathbb{A}^2_{c,d}
\]
and
\[
\mathbb{A}^2_{c,d}\ni(c,d)=
\begin{array}{c}
\begin{tikzpicture} [bend angle=45, looseness=1.2]
\node[name=s,regular polygon, regular polygon sides=3, minimum size=1.5cm] at (0,0) {}; 
\node (C1) at (s.corner 1)  [vertex] {};
\node (C2) at (s.corner 2) [vertex] {};
\node (C3) at (s.corner 3) [vertex] {};
\draw[->] (C2) -- node  [gap] {$\scriptstyle 1$} (C1); 
\draw[->] (C1) -- node  [gap] {$\scriptstyle c$} (C3);
\draw[->] (C3) -- node  [gap] {$\scriptstyle cd$} (C2); 
\draw [->,bend right] (C1) to node [left] {$\scriptstyle cd$} (C2);
\draw [->,bend right] (C2) to node [below] {$\scriptstyle 1$} (C3);
\draw [->,bend right] (C3) to node [right] {$\scriptstyle d$} (C1);
\end{tikzpicture}
\end{array}
\cong
\begin{array}{c}
\begin{tikzpicture} [bend angle=45, looseness=1.2]
\node[name=s,regular polygon, regular polygon sides=3, minimum size=1.5cm] at (0,0) {}; 
\node (C1) at (s.corner 1)  [vertex] {};
\node (C2) at (s.corner 2) [vertex] {};
\node (C3) at (s.corner 3) [vertex] {};
\draw[->] (C2) -- node  [gap] {$\scriptstyle d^{-1}$} (C1); 
\draw[->] (C1) -- node  [gap] {$\scriptstyle cd$} (C3);
\draw[->] (C3) -- node  [gap] {$\scriptstyle cd$} (C2); 
\draw [->,bend right] (C1) to node [left] {$\scriptstyle cd^2$} (C2);
\draw [->,bend right] (C2) to node [below] {$\scriptstyle 1$} (C3);
\draw [->,bend right] (C3) to node [right] {$\scriptstyle 1$} (C1);
\end{tikzpicture}
\end{array}
=(d^{-1},cd^2)
\in \mathbb{A}^2_{e,f}
\]
The upshot is that the moduli space looks (very roughly) like the following:
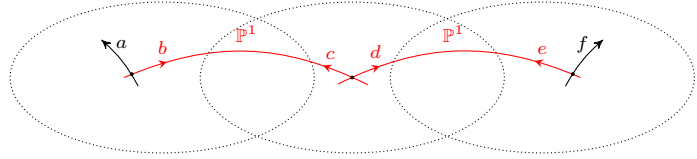
\[
\begin{array}{c}
\\
\begin{tikzpicture}[decoration={markings,mark=at position 0.6cm with {\arrow[red]{stealth};}, mark=at position -0.5cm with {\arrowreversed[red]{stealth};}} ]   
\draw[red,postaction={decorate}] (-0.1,0,0.225) to [bend left=25] node[above, pos=0.15] {$\scriptstyle d$} node[above, pos=0.85] {$\scriptstyle e$} node[pos=0.48, above] {$\scriptstyle \mathbb{P}^1$} (3,0,0);
\draw[red,postaction={decorate}] (-3,0,0) to [bend left=25] node[above, pos=0.15] {$\scriptstyle b$} node[above, pos=0.85] {$\scriptstyle c$} node[pos=0.5, above] {$\scriptstyle \mathbb{P}^1$} (0.3,0,0.25);
%origins
\filldraw [black] (0,0,0) circle (0.5pt);
\filldraw [black] (2.9,0.045,0) circle (0.5pt);
\filldraw [black] (-2.9,0.045,0) circle (0.5pt);
% opens
\draw[densely dotted,black] (-2.5,0) ellipse (2 and 1);
\draw[densely dotted,black] (0,0) ellipse (2 and 1);
\draw[densely dotted,black] (2.5,0) ellipse (2 and 1);
%outside axes
\draw[->] (-2.7,0,0.3) to [bend right=10] node[right, pos=0.9] {$\scriptstyle a$} (-3.3,0.5,0);
\draw[->] (2.92,0,0.3) to [bend left=10] node[left, pos=0.9] {$\scriptstyle f$} (3.3,0.5,0);
\end{tikzpicture} 
\end{array}
\]
where the black dots correspond to the origins of the co-ordinate charts, and the red lines glue to give two $\mathbb{P}^1$s.  

It is instructive to see explicitly why the two red $\mathbb{P}^1$s form the exceptional divisor.  By the nature of the Proj construction, there is a natural map from the moduli space to $\Spec \mathbb{C}[\mathcal{R}]^G$.  What is this map?  First notice that  $\mathbb{C}[\mathcal{R}]$ is, by definition, the commutative ring $\mathbb{C}[c_1,c_2,c_3,a_1,a_2,a_3]/(c_1a_1=c_2a_2=c_3a_3)$ where we have just taken the arrows and relations and made everything commute.  Now since the dimension vector is $\alpha=(1,1,1)$, the group $G$ is precisely $\GL(1,\mathbb{C})\times \GL(1,\mathbb{C})\times \GL(1,\mathbb{C})=\mathbb{C}^*\times\mathbb{C}^*\times\mathbb{C}^*$.  The action is by base change $g\cdot a=g_{t(a)}^{-1}ag_{h(a)}$, which means
\begin{eqnarray*}
(\lambda_1,\lambda_2,\lambda_3)\cdot c_1:=\lambda_1^{-1}c_1\lambda_2\\
(\lambda_1,\lambda_2,\lambda_3)\cdot c_2:=\lambda_2^{-1}c_2\lambda_3\\
(\lambda_1,\lambda_2,\lambda_3)\cdot c_3:=\lambda_3^{-1}c_3\lambda_1\\
(\lambda_1,\lambda_2,\lambda_3)\cdot a_1:=\lambda_2^{-1}a_1\lambda_1\\
(\lambda_1,\lambda_2,\lambda_3)\cdot a_2:=\lambda_3^{-1}a_2\lambda_2\\
(\lambda_1,\lambda_2,\lambda_3)\cdot a_3:=\lambda_1^{-1}a_3\lambda_3
\end{eqnarray*}
For a product to be invariant under this action the $\lambda$'s must cancel, and visually these correspond to cycles in the quiver.  Here, the invariants are generated by $A:=c_1c_2c_3$, $B:=a_1a_2a_3$ and $C:=c_1a_1=c_2a_2=c_3a_3$.  Note that $AB=C^3$, so $\mathbb{C}[\mathcal{R}]^G\cong\mathbb{C}[A,B,C]/(AB-C^3)$, which is our base singularity.

Using this information, the natural map takes a stable representation to the point $(c_1c_2c_3,a_1a_2a_3,c_1a_1)$ of $\Spec\mathbb{C}[\mathcal{R}]^G$, which in the case of the three open sets gives
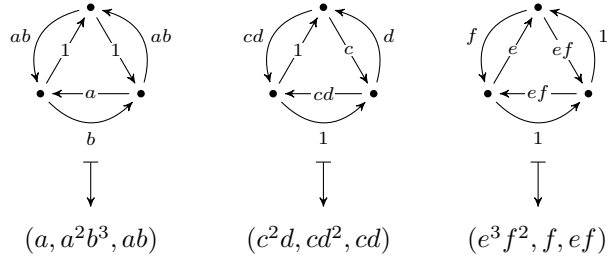
\[
\begin{array}{ccc}
\begin{array}{c}
\begin{tikzpicture} [bend angle=45, looseness=1.2]
\node[name=s,regular polygon, regular polygon sides=3, minimum size=1.5cm] at (0,0) {}; 
\node (C1) at (s.corner 1)  [vertex] {};
\node (C2) at (s.corner 2) [vertex] {};
\node (C3) at (s.corner 3) [vertex] {};
\draw[->] (C2) -- node  [gap] {$\scriptstyle 1$} (C1); 
\draw[->] (C1) -- node  [gap] {$\scriptstyle 1$} (C3);
\draw[->] (C3) -- node  [gap] {$\scriptstyle a$} (C2); 
\draw [->,bend right] (C1) to node [left] {$\scriptstyle ab$} (C2);
\draw [->,bend right] (C2) to node [below] {$\scriptstyle b$} (C3);
\draw [->,bend right] (C3) to node [right] {$\scriptstyle ab$} (C1);
\draw[|->] ($(s.side 2)+(0,-0.9)$) -- ($(s.side 2)+(0,-1.5)$);
\node (a) at ($(s.side 2)+(0,-1.5)$) {};
\end{tikzpicture}
\end{array}
&
\begin{array}{c}
\begin{tikzpicture} [bend angle=45, looseness=1.2]
\node[name=s,regular polygon, regular polygon sides=3, minimum size=1.5cm] at (0,0) {}; 
\node (C1) at (s.corner 1)  [vertex] {};
\node (C2) at (s.corner 2) [vertex] {};
\node (C3) at (s.corner 3) [vertex] {};
\draw[->] (C2) -- node  [gap] {$\scriptstyle 1$} (C1); 
\draw[->] (C1) -- node  [gap] {$\scriptstyle c$} (C3);
\draw[->] (C3) -- node  [gap] {$\scriptstyle cd$} (C2); 
\draw [->,bend right] (C1) to node [left] {$\scriptstyle cd$} (C2);
\draw [->,bend right] (C2) to node [below] {$\scriptstyle 1$} (C3);
\draw [->,bend right] (C3) to node [right] {$\scriptstyle d$} (C1);
\draw[|->] ($(s.side 2)+(0,-0.9)$) -- ($(s.side 2)+(0,-1.5)$);
\node (a) at ($(s.side 2)+(0,-1.5)$) {};
\end{tikzpicture}
\end{array}
&
\begin{array}{c}
\begin{tikzpicture} [bend angle=45, looseness=1.2]
\node[name=s,regular polygon, regular polygon sides=3, minimum size=1.5cm] at (0,0) {}; 
\node (C1) at (s.corner 1)  [vertex] {};
\node (C2) at (s.corner 2) [vertex] {};
\node (C3) at (s.corner 3) [vertex] {};
\draw[->] (C2) -- node  [gap] {$\scriptstyle e$} (C1); 
\draw[->] (C1) -- node  [gap] {$\scriptstyle ef$} (C3);
\draw[->] (C3) -- node  [gap] {$\scriptstyle ef$} (C2); 
\draw [->,bend right] (C1) to node [left] {$\scriptstyle f$} (C2);
\draw [->,bend right] (C2) to node [below] {$\scriptstyle 1$} (C3);
\draw [->,bend right] (C3) to node [right] {$\scriptstyle 1$} (C1);
\draw[|->] ($(s.side 2)+(0,-0.9)$) -- ($(s.side 2)+(0,-1.5)$);
\node (a) at ($(s.side 2)+(0,-1.5)$) {};
\end{tikzpicture}
\end{array}\\
(a,a^2b^3,ab) &(c^2d,cd^2,cd)&(e^3f^2,f,ef)
\end{array}
\]
Now we look above the singular point $(0,0,0)$, and in $U_1$ we see $\{ (0,b)\mid b\in\mathbb{C}\}$, in $U_2$ we see $\{ (c,d)\mid cd=0\}$, and in $U_3$ we see $\{ (e,0)\mid e\in\mathbb{C}\}$.   Thus, the red lines in the rough picture constitute the exceptional locus, and they are the union of two $\mathbb{P}^1$s.  The space is smooth, and is in fact the minimal resolution of our original motivating singularity.
\end{example}

\begin{remark}\label{star gen}
There is a pattern evident in the previous examples.  If we consider dimension vector $(1,1,\hdots,1)$ and stability condition $(-n,1,\hdots,1)$, where the $-n$ corresponds to a vertex $\star$, then a module $M$ of dimension vector $(1,1,\hdots,1)$ is $\theta$-stable if and only if for every vertex in the quiver representation of $M$, there is a non-zero path from $\star$ to that vertex. This is quite a pleasant combinatorial exercise (see Exercise~\ref{Ex3.16}).
\end{remark}

\begin{example}\label{SSP text} (The suspended pinch point). 
This example illustrates the phenomenon of non-isomorphic crepant resolutions in dimension three.  The example to consider is $R:=\mathbb{C}[u,v,x,y]/(uv-x^2y)$.  This has a 1-dimensional singular locus, namely the $y$-axis $u=v=x=0$.  

Here $\End_R(R\oplus (u,x)\oplus (u,x^2))$ is a NCCR, and in fact it can be presented as 
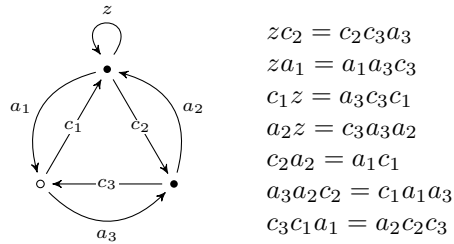
\[
\begin{array}{cc}
\begin{array}{c}
\begin{tikzpicture} [bend angle=45, looseness=1.2]
\node[name=s,regular polygon, regular polygon sides=3, minimum size=2cm] at (0,0) {}; 
\node (C1) at (s.corner 1) [vertex] {};
\node (C2) at (s.corner 2) [cvertex] {};
\node (C3) at (s.corner 3) [vertex] {};
\draw[->] (C2) -- node  [gap] {$\scriptstyle c_1$} (C1); 
\draw[->] (C1) -- node  [gap] {$\scriptstyle c_2$} (C3);
\draw[->] (C3) -- node  [gap] {$\scriptstyle c_3$} (C2); 
\draw [->,bend right] (C1) to node [left] {$\scriptstyle a_1$} (C2);
\draw [->,bend right] (C2) to node [below] {$\scriptstyle a_3$} (C3);
\draw [->,bend right] (C3) to node [right] {$\scriptstyle a_2$} (C1);
\node (C1a) at ($(s.corner 1)+(90:2pt)$) {};
\draw[<-]  (C1a) edge [in=55,out=125,loop,looseness=10] node[above] {$\scriptstyle z$} (C1a);
\end{tikzpicture}
\end{array} &
\begin{array}{l}
zc_2=c_2c_3a_3\\
za_1=a_1a_3c_3\\
c_1z=a_3c_3c_1\\
a_2z=c_3a_3a_2\\
c_2a_2=a_1c_1\\
a_3a_2c_2=c_1a_1a_3\\
c_3c_1a_1=a_2c_2c_3
\end{array}
\end{array}
\]
Since this is a NCCR, we pick the dimension vector corresponding to the ranks of the CM modules, which in this case is $(1,1,1)$.  By a similar calculation as in \ref{Z3 GIT} (but bearing in mind \ref{star gen}), computing the moduli space for the stability $(-2,1,1)$ then looking above the $y$-axis (=the singular locus) we see 
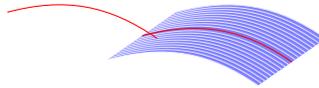
\[
\begin{array}{c}
\\
\begin{tikzpicture} [bend angle=25, looseness=1,transform shape, rotate=-10]
\fill[fill=blue!50] (0,0,1) -- (0,0,-1) to [bend left=25] (2,0,-1) -- (2,0,1) to [bend right=25] (0,0,1);
\foreach \y in {0.1,0.2,...,1}{ 
\draw[very thin,blue!10] (0,0,\y) to [bend left=25] (2,0,\y);
\draw[very thin,blue!10] (0,0,-\y) to [bend left=25] (2,0,-\y);}
\draw[red] (0,0,0) to [bend left=25] (2,0,0);
%\draw[red] (1.8,0,0) to [bend left=25] (3.8,0,0);
\draw[red] (-1.8,0,0) to [bend left=25] (0.2,0,0);
\end{tikzpicture}
\end{array}
\]
where the red lines indicate the two $\mathbb{P}^1$s that are above the origin.  However, if taking the stability $(-1,2,-1)$ (where the $2$ is on the top vertex) and looking above the $y$-axis gives the picture
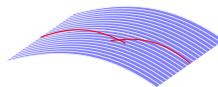
\[
\begin{array}{c}
\\
\begin{array}{c}
\begin{tikzpicture} [bend angle=25, looseness=1,transform shape, rotate=-10]
\fill[fill=blue!50] (0,0,1) -- (0,0,-1) to [bend left=25] (2,0,-1) -- (2,0,1) to [bend right=25] (0,0,1);
\foreach \y in {0,0.1,0.2,...,1}{ 
\draw[very thin,blue!10] (0,0,\y) to [bend left=25] (2,0,\y);
\draw[very thin,blue!10] (0,0,-\y) to [bend left=25] (2,0,-\y);}
\draw[red] (0,0,0) to [bend left=25] (1,0,-0.3);
\draw[red] (0.8,0,-0.3) to [bend left=25] (2,0,0);
\end{tikzpicture}
\end{array}
\end{array}
\]
This is recommended as an instructive exercise in quiver GIT (Exercise~\ref{Ex3.14} is similar). The left hand curve in our original picture has been flopped into the exceptional blue surface.  These two moduli spaces are both smooth (in fact, they are crepant resolutions of $\Spec R$), but they are not isomorphic.
\end{example}

In the language of toric geometry, the two crepant resolutions above correspond to the following pictures
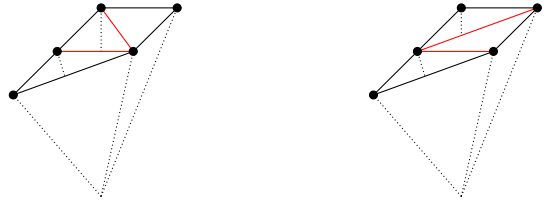
\[
\begin{array}{ccc}
\begin{array}{c}
\begin{tikzpicture} 
\draw[densely dotted] (0,-2.5,0) -- (0,0,0);
\fill[color=white] (0,0,1.5) -- (1,0,1.5) -- (0,-2.5,0) -- (0,0,1.5);
\draw[densely dotted] (0,-2.5,0) -- (0,0,1.5);
\fill[color=white] (0,0,3) -- (1,0,1.5) -- (0,-2.5,0) -- (0,0,3);
%triangulation
\draw[red] (0,0,1.5) -- (1,0,1.5);
\draw[red] (0,0,0) -- (1,0,1.5);
\filldraw [black] (0,0,0) circle (1.5pt);
\filldraw [black] (0,0,1.5) circle (1.5pt);
\filldraw [black] (0,0,3) circle (1.5pt);
\filldraw [black] (1,0,0) circle (1.5pt);
\filldraw [black] (1,0,1.5) circle (1.5pt);
\draw (0,0,0) -- (0,0,3);
\draw (0,0,3) -- (1,0,1.5);
\draw (1,0,1.5) -- (1,0,0);
\draw (1,0,0) -- (0,0,0);
\draw[densely dotted] (0,-2.5,0) -- (1,0,0);
\draw[densely dotted] (0,-2.5,0) -- (0,0,3);
\draw[densely dotted] (0,-2.5,0) -- (1,0,1.5);
\end{tikzpicture} 
\end{array}
&\qquad\qquad&
\begin{array}{c}
\begin{tikzpicture} 
\draw[densely dotted] (0,-2.5,0) -- (0,0,0);
\fill[color=white] (0,0,1.5) -- (1,0,0) -- (0,-2.5,0) -- (0,0,1.5);
\draw[densely dotted] (0,-2.5,0) -- (0,0,1.5);
\fill[color=white] (0,0,3) -- (1,0,1.5) -- (0,-2.5,0) -- (0,0,3);
%triangulation
\draw[red] (0,0,1.5) -- (1,0,1.5);
\draw[red] (0,0,1.5) -- (1,0,0);
\filldraw [black] (0,0,0) circle (1.5pt);
\filldraw [black] (0,0,1.5) circle (1.5pt);
\filldraw [black] (0,0,3) circle (1.5pt);
\filldraw [black] (1,0,0) circle (1.5pt);
\filldraw [black] (1,0,1.5) circle (1.5pt);
\draw (0,0,0) -- (0,0,3);
\draw (0,0,3) -- (1,0,1.5);
\draw (1,0,1.5) -- (1,0,0);
\draw (1,0,0) -- (0,0,0);
\draw[densely dotted] (0,-2.5,0) -- (1,0,0);
\draw[densely dotted] (0,-2.5,0) -- (0,0,3);
\draw[densely dotted] (0,-2.5,0) -- (1,0,1.5);

\end{tikzpicture} 
\end{array}\\
&&
\end{array}
\]
Since in toric geometry crepant resolutions correspond to subdividing the cone, there is one further crepant resolution.  This too can be obtained using quiver GIT (in a similar way to Exercise~\ref{Ex3.14}).

\medskip
\medskip
\noindent
{\bf Credits:} All the material on quiver GIT is based on the original paper of King \cite{King}, and also on the lecture course he gave in Bath in 2006.  Examples~\ref{P1 GIT} and \ref{blowup GIT} were two of the motivating examples for the theory. The example of the $\mathbb{Z}_3$ singularity has been studied by so many people that it is hard to properly credit each; certainly it appears in the work of Kronheimer \cite{Kronheimer} and Cassens--Slodowy \cite{CS}, and is used (and expanded) in the work of Craw--Ishii \cite{CI} and many others.  The fact that $\mathbb{C}[\c{R}]^G=\mathbb{C}[a,b,c]/(ab-c^3)$ in \ref{Z3 GIT} is a general theorem for Kleinian singularities, but also appears in this special case in the lectures of Le Bruyn \cite{leB}.  The suspended pinch point is the name given by physicists to the singularity $uv=x^2y$.  The geometry is toric and so well--known.   The NCCRs in this case can be found in either Van den Bergh \cite[\S 8]{VdBNCCR}, Nagao \cite[\S1.4]{Nag}, or \cite{IW5}, but also in many papers by various physicists.

\subsection{Exercises}

\begin{ex}\label{Ex3.12}
Consider the examples in Exercise~\ref{Ex2.16}. 
For each, consider the dimension vector $(1,1)$. There are essentially only two generic stability conditions, namely $(-1,1)$ and $(1,-1)$.  For each of the above examples, compute the spaces given by these two stabilities.  Each of the examples should illustrate a different phenomenon.
\end{ex}

\begin{ex}\label{Ex3.13}
Consider some of the examples in Exercise~\ref{Ex2.19}, namely
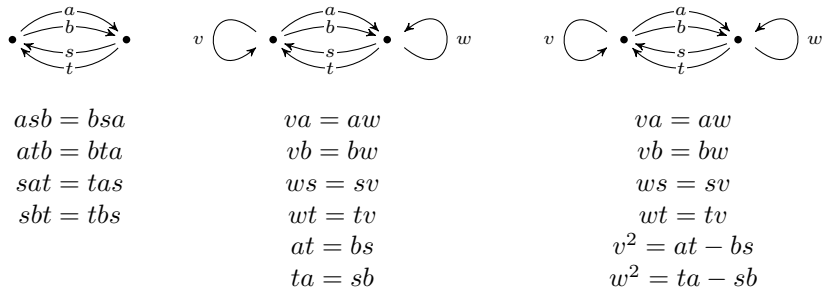
\[
\begin{array}{ccc}
&\\
\begin{array}{c}
\begin{tikzpicture} [bend angle=45, looseness=1]
\node (C1) at (0,0) [vertex] {};
\node (C2) at (1.5,0)  [vertex] {};
\draw [->,bend left] (C1) to node[gap]  {$\scriptstyle a$} (C2);
\draw [->,bend left=20,looseness=1] (C1) to node[gap]  {$\scriptstyle b$} (C2);
\draw [->,bend left] (C2) to node[gap]  {$\scriptstyle t$} (C1);
\draw [->,bend left=20,looseness=1] (C2) to node[gap]  {$\scriptstyle s$} (C1);
\end{tikzpicture}
\end{array}
& 
\begin{array}{c}
\begin{tikzpicture} [bend angle=45, looseness=1]
\node (C1) at (0,0) [vertex] {};
\node (C2) at (1.5,0)  [vertex] {};
\node (C1a) at (-0.1,0)  {};
\node (C2a) at (1.6,0) {};
\draw [->,bend left] (C1) to node[gap]  {$\scriptstyle a$} (C2);
\draw [->,bend left=20,looseness=1] (C1) to node[gap]  {$\scriptstyle b$} (C2);
\draw [->,bend left] (C2) to node[gap]  {$\scriptstyle t$} (C1);
\draw [->,bend left=20,looseness=1] (C2) to node[gap]  {$\scriptstyle s$} (C1);
\draw[<-]  (C1a) edge [in=140,out=-140,loop,looseness=12] node[left] {$\scriptstyle v$} (C1a);
\draw[->]  (C2a) edge [in=40,out=-40,loop,looseness=12] node[right] {$\scriptstyle w$} (C2a);
\end{tikzpicture}
\end{array}
&
\begin{array}{c}
\begin{tikzpicture} [bend angle=45, looseness=1]
\node (C1) at (0,0) [vertex] {};
\node (C2) at (1.5,0)  [vertex] {};
\node (C1a) at (-0.1,0)  {};
\node (C2a) at (1.6,0) {};
\draw [->,bend left] (C1) to node[gap]  {$\scriptstyle a$} (C2);
\draw [->,bend left=20,looseness=1] (C1) to node[gap]  {$\scriptstyle b$} (C2);
\draw [->,bend left] (C2) to node[gap]  {$\scriptstyle t$} (C1);
\draw [->,bend left=20,looseness=1] (C2) to node[gap]  {$\scriptstyle s$} (C1);
\draw[<-]  (C1a) edge [in=140,out=-140,loop,looseness=12] node[left] {$\scriptstyle v$} (C1a);
\draw[->]  (C2a) edge [in=40,out=-40,loop,looseness=12] node[right] {$\scriptstyle w$} (C2a);
\end{tikzpicture}
\end{array}
\\
asb=bsa
&
va=aw
&
va=aw
\\
atb=bta
&
vb=bw
&
vb=bw
\\
sat=tas
&
ws=sv
&
ws=sv
\\
sbt=tbs
&
wt=tv
&
wt=tv
\\
&
at=bs
&
v^2=at-bs
\\
&
ta=sb
&
w^2=ta-sb\\
&
\end{array}
\]
In each of the first three examples, compute the spaces given by the two generic stabilities for the dimension vector $(1,1)$.  Are the spaces isomorphic?  Are you sure? 
\end{ex}

\begin{ex}\label{Ex3.14}
Consider the ring $R:=\mathbb{C}[u,v,x,y]/uv=x^2y$.  In this example $\End_R(R\oplus (u,x)\oplus (u,xy))$ is a NCCR, and in fact it can be presented as 
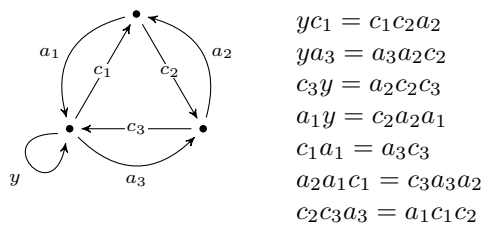
\[
\begin{array}{cc}
&\\
\begin{array}{c}
\begin{tikzpicture} [bend angle=45, looseness=1.2]
\node[name=s,regular polygon, regular polygon sides=3, minimum size=2cm] at (0,0) {}; 
\node (C1) at (s.corner 1) [vertex] {};
\node (C2) at (s.corner 2) [vertex] {};
\node (C3) at (s.corner 3) [vertex] {};
\draw[->] (C2) -- node  [gap] {$\scriptstyle c_1$} (C1); 
\draw[->] (C1) -- node  [gap] {$\scriptstyle c_2$} (C3);
\draw[->] (C3) -- node  [gap] {$\scriptstyle c_3$} (C2); 
\draw [->,bend right] (C1) to node [left] {$\scriptstyle a_1$} (C2);
\draw [->,bend right] (C2) to node [below] {$\scriptstyle a_3$} (C3);
\draw [->,bend right] (C3) to node [right] {$\scriptstyle a_2$} (C1);
\node (C2a) at ($(s.corner 2)+(-115:2pt)$) {};
\draw[<-]  (C2a) edge [in=180,out=260,loop,looseness=15] node[below left] {$\scriptstyle y$} (C2a);
\end{tikzpicture}
\end{array} &
\begin{array}{l}
yc_1=c_1c_2a_2\\
ya_3=a_3a_2c_2\\
c_3y=a_2c_2c_3\\
a_1y=c_2a_2a_1\\
c_1a_1=a_3c_3\\
a_2a_1c_1=c_3a_3a_2\\
c_2c_3a_3=a_1c_1c_2
\end{array}
\end{array}
\]
(the relations can in fact be packaged as a superpotential).
Consider the dimension vector $(1,1,1)$.  For this example there are essentially six generic stability conditions.  Compute each.  How many non-isomorphic crepant resolutions are obtained?
\end{ex}

\begin{ex}\label{Ex3.15}
(Shows that $\Rep{/\!\!/}\GL$ might not be what you want) Consider $(Q,R)$ given by
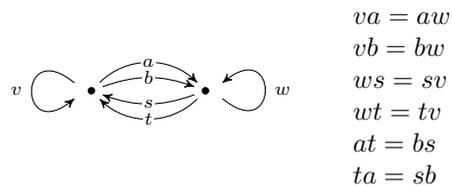
\[
\begin{array}{cc} 
\begin{array}{c}
\begin{tikzpicture} [bend angle=45, looseness=1]
\node (C1) at (0,0) [vertex] {};
\node (C2) at (1.5,0)  [vertex] {};
\node (C1a) at (-0.1,0)  {};
\node (C2a) at (1.6,0) {};
\draw [->,bend left] (C1) to node[gap]  {$\scriptstyle a$} (C2);
\draw [->,bend left=20,looseness=1] (C1) to node[gap]  {$\scriptstyle b$} (C2);
\draw [->,bend left] (C2) to node[gap]  {$\scriptstyle t$} (C1);
\draw [->,bend left=20,looseness=1] (C2) to node[gap]  {$\scriptstyle s$} (C1);
\draw[<-]  (C1a) edge [in=140,out=-140,loop,looseness=12] node[left] {$\scriptstyle v$} (C1a);
\draw[->]  (C2a) edge [in=40,out=-40,loop,looseness=12] node[right] {$\scriptstyle w$} (C2a);
\end{tikzpicture}
\end{array}
&
{\scriptstyle\begin{array}{l}
va=aw\\
vb=bw\\
ws=sv\\
wt=tv\\
at=bs\\
ta=sb
\end{array}}
\end{array}
\]
and set $\Lambda:=kQ/R$.  
We know from Exercise~\ref{Ex2.19} that this is a NCCR over $Z(\Lambda)$.  Show that $\Rep(kQ,(1,1)){/\!\!/}\GL$ is not isomorphic to $Z(\Lambda)$.
\end{ex}

\begin{ex}\label{Ex3.16} Prove \ref{star gen}.
\end{ex}

\newpage

\section{Into Derived Categories}\label{L4}

We retain the setup that $R$ is a commutative noetherian Gorenstein normal domain, and assume that $\Lambda=\End_R(M)$ is a NCCR.  In this section, we begin to relate this homologically to geometric crepant resolutions $Y\to\Spec R$.

\subsection{Derived Categories: Motivation and Definition} Let $\mathcal{A}$ and $\mathcal{B}$ denote abelian categories.  In our case $\mathcal{A}=\mod \Lambda$ and $\mathcal{B}=\coh Y$.  It is very unlikely that $\mathcal{A}\simeq\mathcal{B}$ since usually $\coh Y$ does not have enough projectives, whereas $\mod A$ always does.  But we still want to homologically relate $Y$ and $\Lambda$.  The derived category $\D(\mathcal{A})$ solves this issue since it carries many of the invariants that we care about, whilst at the same time allowing the flexibility of $\D(\mathcal{A})\simeq \D(\mathcal{B})$ even when $\mathcal{A}\ncong\mathcal{B}$.

Now to create the derived category, we observe that we can take a projective resolution of $M\in\mod\Lambda$ and view it as a commutative diagram 
\[
\begin{tikzpicture}
\node (Z) at (-1.5,0) {$\hdots$};
\node (A) at (0,0) {$P_2$};
\node (B) at (1.5,0) {$P_1$};
\node (C) at (3,0) {$P_0$};
\node (D) at (4.5,0) {$0$};
\node (E) at (6,0) {$\hdots$};
\node (Z1) at (-1.5,-1.5) {$\hdots$};
\node (A1) at (0,-1.5) {$0$};
\node (B1) at (1.5,-1.5) {$0$};
\node (C1) at (3,-1.5) {$M$};
\node (D1) at (4.5,-1.5) {$0$};
\node (E1) at (6,-1.5) {$\hdots$};
\draw[->] (Z) -- (A);
\draw[->] (A) -- (B);
\draw[->] (B) -- (C);
\draw[->] (C) -- (D);
\draw[->] (D) -- (E);
\draw[->] (Z1) -- (A1);
\draw[->] (A1) -- (B1);
\draw[->] (B1) -- (C1);
\draw[->] (C1) -- (D1);
\draw[->] (D1) -- (E1);
\draw[->] (A) -- (A1);
\draw[->] (B) -- (B1);
\draw[->] (C) -- (C1);
\draw[->] (D) -- (D1);
\end{tikzpicture}
\]
We write this $P_\bullet\stackrel{f}{\to}M$.  This map has the property that homology $\H^i(f):\H^i(P_\bullet)\to\H^i(M)$ is an isomorphism for all $i\in\mathbb{Z}$.   

\begin{defin}
For any abelian category $\mathcal{A}$, we define the category of chain complexes, denoted $\C(\mathcal{A})$, as follows.  Objects are chain complexes, i.e.\
\[
\begin{tikzpicture}
\node (Z) at (-1.5,0) {$\hdots$};
\node (A) at (0,0) {$C_{-1}$};
\node (B) at (1.5,0) {$C_0$};
\node (C) at (3,0) {$C_1$};
\node (D) at (4.5,0) {$\hdots$};
\draw[->] (Z) -- node[above] {$\scriptstyle d_{-1}$} (A);
\draw[->] (A) -- node[above] {$\scriptstyle d_{0}$} (B);
\draw[->] (B) -- node[above] {$\scriptstyle d_{1}$} (C);
\draw[->] (C) -- node[above] {$\scriptstyle d_{2}$} (D);
\end{tikzpicture}
\]
with each $C_i\in\mathcal{A}$, such that $d_id_{i+1}=0$ for all $i\in\mathbb{Z}$, and the morphisms $C_\bullet\to D_\bullet$ are collections of morphisms in $\mathcal{A}$ such that 
\[
\begin{tikzpicture}
\node (Z) at (-1.5,0) {$\hdots$};
\node (A) at (0,0) {$C_{-1}$};
\node (B) at (1.5,0) {$C_0$};
\node (C) at (3,0) {$C_1$};
\node (D) at (4.5,0) {$\hdots$};
\node (Z1) at (-1.5,-1.5) {$\hdots$};
\node (A1) at (0,-1.5) {$D_{-1}$};
\node (B1) at (1.5,-1.5) {$D_0$};
\node (C1) at (3,-1.5) {$D_1$};
\node (D1) at (4.5,-1.5) {$\hdots$};
\draw[->] (Z) -- (A);
\draw[->] (A) -- (B);
\draw[->] (B) -- (C);
\draw[->] (C) -- (D);
\draw[->] (Z1) -- (A1);
\draw[->] (A1) -- (B1);
\draw[->] (B1) -- (C1);
\draw[->] (C1) -- (D1);
\draw[->] (A) -- (A1);
\draw[->] (B) -- (B1);
\draw[->] (C) -- (C1);
\end{tikzpicture}
\] 
commutes.  A map of chain complexes $f:C_\bullet\to D_\bullet$ is called a quasi-isomorphism (=qis) if homology $\H^i(f):\H^i(C_\bullet)\to\H^i(D_\bullet)$ is an isomorphism for all $i\in\mathbb{Z}$.  The derived category $\D(\mathcal{A})$, is defined to be $\C(\mathcal{A})[\{qis\}^{-1}]$, where we just formally invert all quasi-isomorphisms. The bounded derived category $\Db(\mathcal{A})$ is defined to be the full subcategory of $\D(\mathcal{A})$ consisting of complexes isomorphic (in the derived category) to bounded complexes $\hdots\to 0\to C_i\to C_{i+1}\to\hdots\to C_j\to 0\to\hdots$.
\end{defin}

Thus in the derived category we just formally identify $M$ and its projective resolution.  Now much of what we do on the abelian category level is very formal --- the building blocks of homological algebra are short exact sequences, and we have constructions like kernels and cokernels.  Often for many proofs (e.g.\ in \S\ref{L1} and \S\ref{L2}) we just need the fact that these constructions exist, rather than precise knowledge of the form they take. 

When passing from abelian categories to derived categories, the building blocks are no longer short exact sequences, instead these are replaced by a weaker notion of \emph{triangles}.  As in the abelian setting, many constructions and proofs follow formally from the properties of triangles. The derived category is an example of a triangulated category, which is defined as follows.
\begin{defin} A triangulated category is an additive
category $\c{C}$ together with an additive autoequivalence 
$[1]:\c{C}\to \c{C}$ and a class of sequences
\[
X\to Y\to Z\to X[1]
\]
called triangles, satisfying the following:\\
T1(a). Every sequence $X^\prime\to Y^\prime\to Z^\prime\to X^\prime[1]$ isomorphic to a triangle is itself a triangle.\\
T1(b). For every object $X\in\c{C}$, $0\to X\xrightarrow{id} X\to 0[1]$ is a triangle.\\
T1(c). Every map $f\colon X\to Y$ can completed to a triangle 
\[
X\xrightarrow{f}Y\to Z\to X[1].
\]
T2. (Rotation) We have 
\[
X\xrightarrow{f}Y\xrightarrow{g} Z\xrightarrow{h} X[1]
\textnormal{ is a triangle} \iff Y\xrightarrow{g} Z\xrightarrow{h} X[1]\xrightarrow{-f[1]}Y[1]\textnormal{ is a triangle}.
\]
T3. Given a commutative diagram
\[
\begin{tikzpicture}
\node (A) at (0,0) {$X$};
\node (B) at (1.5,0) {$Y$};
\node (C) at (3,0) {$Z$};
\node (D) at (4.5,0) {$X[1]$};
\node (A1) at (0,-1.5) {$X^\prime$};
\node (B1) at (1.5,-1.5) {$Y^\prime$};
\node (C1) at (3,-1.5) {$Z^\prime$};
\node (D1) at (4.5,-1.5) {$X^\prime[1]$};
\draw[->] (A) -- (B);
\draw[->] (B) -- (C);
\draw[->] (C) -- (D);
\draw[->] (A1) -- (B1);
\draw[->] (B1) -- (C1);
\draw[->] (C1) -- (D1);
\draw[->] (A) -- (A1);
\draw[->] (B) -- (B1);
\draw[->] (D) -- (D1);
\end{tikzpicture}
\]
where the two rows are triangles, then there exists $Z\to Z^\prime$ such that the whole diagram commutes.\\
T4. (Octahedral Axiom). Given
\[
\begin{tikzpicture}
\node (A) at (0,0) {$A$};
\node (B) at (1.5,0) {$B$};
\node (C) at (4,0) {$C$};
\node (A1) at (6,0) {$A[1]$};
\node (A1a) at (5,-2) {$A[1]$};
\node (B1) at (3.75,-3.5) {$B[1]$};
\node (C1) at (2.25,-4.5) {$C[1]$};
\node (D) at (intersection of B--B1 and A--A1a) {$D$};
\node (E) at (intersection of C--C1 and A--A1a) {$E$};
\node (F) at (intersection of C--C1 and B--B1) {$F$};
\draw[->] (A) -- node[above] {$\scriptstyle a$} (B);
\draw[->] (B) -- node[above] {$\scriptstyle b$} (C);
\draw[->] (C) -- node[above] {$\scriptstyle c$} (A1);
\draw[->] (B) -- node[right,pos=0.2] {$\scriptstyle d$} (D);
\draw[->] (D) -- node[above] {$\scriptstyle e$} (E);
\draw[->] (E) -- node[above] {$\scriptstyle f$} (A1a);
\draw[->] (A) -- node[below] {$\scriptstyle ad$} (D);
\draw[->] (D) -- node[right] {$\scriptstyle g$} (F);
\draw[->] (F) -- node[right, pos=0.2] {$\scriptstyle h$} (B1);
\draw[double distance=1.5pt] (A1) -- (A1a);
\draw[->] (A1a) -- node[right] {$\scriptstyle a[1]$} (B1);
\draw[->] (B1) -- node[below,pos=0.2] {$\scriptstyle b[1]$} (C1);
\end{tikzpicture}
\]
where $(a,b,c)$, $(d,g,h)$ and $(ad,e,f)$ are triangles, there exists morphisms such that $C\to E\to F\to C[1]$ is a triangle, and the whole diagram commutes.
\end{defin}

The only fact needed for now is that short exact sequences of complexes give triangles in the derived category.

\subsection{Tilting}
We now return to our setup.  We are interested in possible equivalences between $\Db(\coh Y)$ and $\Db(\mod \Lambda)$.  How to achieve this?    We first note that there are two nice subcategories of $\Db(\coh Y)$ and $\Db(\mod \Lambda)$.

\begin{defin}
We define $\Perf(Y)\subseteq \Db(\coh Y)$ to be all those complexes that are (locally) quasi--isomorphic to bounded complexes consisting of vector bundles of finite rank.  We denote $\Kb(\proj\Lambda)\subseteq \Db(\mod\Lambda)$ to be all those complexes isomorphic to bounded complexes of finitely generated projective $\Lambda$-modules.
\end{defin}

From now on, to simplify matters we will always assume that our schemes are quasi--projective over a commutative noetherian ring of finite type over $\mathbb{C}$, since in our NCCR quiver GIT setup, this will always be true.  We could get by with less, but the details become more technical. 

Under these assumptions, $\Perf(Y)$ can be described as all those complexes that are  isomorphic (in the derived category) to bounded complexes consisting of vector bundles of finite rank \cite[1.6, 1.7]{Orlov1}.  Furthermore, any equivalence between $\Db(\coh Y)$ and $\Db(\mod\Lambda)$ must restrict to an equivalence between $\Perf(Y)$ and $\Kb(\proj\Lambda)$, since both can be characterized intrinsically as the homologically finite complexes.

Now the point is that $\Kb(\proj \Lambda)$ has a very special object ${}_\Lambda\Lambda$, considered as a complex in degree zero.  For $\Db(\coh Y)\simeq\Db(\mod \Lambda)$ we need $\Perf(Y)\simeq\Kb(\proj \Lambda)$, so we need $\Perf(Y)$ to contain an object that behaves in the same way as ${}_\Lambda\Lambda$ does.  But what properties does ${}_\Lambda\Lambda$ have?

The first property is Hom-vanishing in the derived category.
\begin{fact}\label{A Ext}
If $M$ and $N$ are $\Lambda$-modules, thought of as complexes in degree zero, we have
\[
\Hom_{\Db(\mod \Lambda)}(M,N[i])\cong \Ext^i_\Lambda(M,N)
\]
for all $i\in\mathbb{Z}$.  In particular $\Hom_{\Db(\mod \Lambda)}({}_\Lambda\Lambda,{}_\Lambda\Lambda[i])=0$ for all $i\neq 0$.
\end{fact}

Next, we have to develop some language to say that $\Kb(\proj \Lambda)$ is `built' from ${}_\Lambda\Lambda$. 

\begin{defin}
Let $\c{C}$ be a triangulated category.  A full subcategory $\c{D}$ is called a triangulated subcategory if (a) $0\in\c{D}$ (b) $\c{D}$ is closed under finite sums (c) $\c{D}$ is closed under shifts (d) (2 out of 3 property) If $X\to Y\to Z\to X[1]$ is a triangle in $\c{C}$, then if any two of $\{ X,Y,Z\}$ is in $\c{D}$, then so is the third.  If further $\c{D}$ is closed under direct summands (i.e.\ $X\oplus Y\in\c{D}$ implies that $X,Y\in\c{D}$), then we say that $\c{D}$ is thick.
\end{defin}

\begin{notation}
Let $\c{C}$ be a triangulated category, $M\in\c{C}$.  We denote by $\thick(M)$ the smallest full thick triangulated subcategory containing $M$.
\end{notation}

Using this, the second property that ${}_\Lambda\Lambda$ possesses is \emph{generation}.
\begin{example}\label{A gens}
Consider ${}_\Lambda\Lambda\in\Db(\mod \Lambda)$, considered as a complex in degree zero.  We claim that $\thick({}_\Lambda\Lambda)=\Kb(\proj\Lambda)$.  Since $\thick({}_\Lambda\Lambda)$ is closed under finite sums, it contains all finitely generated free $\Lambda$-modules, and further since it is closed under summands it contains all projective $\Lambda$-modules.  It is closed under shifts, so it contains $P[i]$ for all finitely generated projectives $P$ and all $i\in\mathbb{Z}$.  Now consider a 2-term complex 
\[
\begin{tikzpicture}
\node (Z) at (-1.5,0) {$\hdots$};
\node (A) at (0,0) {$0$};
\node (B) at (1.5,0) {$P_1$};
\node (C) at (3,0) {$P_0$};
\node (D) at (4.5,0) {$0$};
\node (E) at (6,0) {$\hdots$};
\draw[->] (Z) -- (A);
\draw[->] (A) -- (B);
\draw[->] (B) -- (C);
\draw[->] (C) -- (D);
\draw[->] (D) -- (E);
\end{tikzpicture}
\]
with $P_0, P_1\in\proj \Lambda$.  We have a commutative diagram
\[
\begin{tikzpicture}
\node (Z) at (-1.5,0) {$\hdots$};
\node (A) at (0,0) {$0$};
\node (B) at (1.5,0) {$0$};
\node (C) at (3,0) {$P_0$};
\node (D) at (4.5,0) {$0$};
\node (E) at (6,0) {$\hdots$};
\node (Z1) at (-1.5,-1.5) {$\hdots$};
\node (A1) at (0,-1.5) {$0$};
\node (B1) at (1.5,-1.5) {$P_1$};
\node (C1) at (3,-1.5) {$P_0$};
\node (D1) at (4.5,-1.5) {$0$};
\node (E1) at (6,-1.5) {$\hdots$};
\node (Z2) at (-1.5,-3) {$\hdots$};
\node (A2) at (0,-3) {$0$};
\node (B2) at (1.5,-3) {$P_1$};
\node (C2) at (3,-3) {$0$};
\node (D2) at (4.5,-3) {$0$};
\node (E2) at (6,-3) {$\hdots$};
\draw[->] (Z) -- (A);
\draw[->] (A) -- (B);
\draw[->] (B) -- (C);
\draw[->] (C) -- (D);
\draw[->] (D) -- (E);
\draw[->] (Z1) -- (A1);
\draw[->] (A1) -- (B1);
\draw[->] (B1) -- (C1);
\draw[->] (C1) -- (D1);
\draw[->] (D1) -- (E1);
\draw[->] (Z2) -- (A2);
\draw[->] (A2) -- (B2);
\draw[->] (B2) -- (C2);
\draw[->] (C2) -- (D2);
\draw[->] (D2) -- (E2);
\draw[->] (A) -- (A1);
\draw[->] (B) -- (B1);
\draw[->] (C) -- (C1);
\draw[->] (D) -- (D1);
\draw[->] (A1) -- (A2);
\draw[->] (B1) -- (B2);
\draw[->] (C1) -- (C2);
\draw[->] (D1) -- (D2);
\end{tikzpicture}
\]
which is a short exact sequence of complexes.  But short exact sequences of complexes give triangles in the derived category, so since the outer two terms belong to $\thick({}_\Lambda\Lambda)$, so does the middle (using the 2 out of 3 property).  This shows that all 2-term complexes of finitely generated projectives belong to $\thick({}_\Lambda\Lambda)$.  By induction, we have that all bounded complexes of finitely generated projectives belong to $\thick({}_\Lambda\Lambda)$, i.e.\ $\Kb(\proj\Lambda)\subseteq\thick({}_\Lambda\Lambda)$.  But $\Kb(\proj\Lambda)$ is a full thick triangulated subcategory containing ${}_\Lambda\Lambda$, so since $\thick({}_\Lambda\Lambda)$ is the smallest such, we conclude that $\Kb(\proj\Lambda)=\thick({}_\Lambda\Lambda)$.
\end{example}

Thus, combining \ref{A Ext} and \ref{A gens}, a necessary condition for $\Db(\coh Y)\simeq \Db(\mod \Lambda)$ is that there exists a complex $\c{V}\in\Perf(Y)$ for which $\Hom_{\Db(\coh Y)}(\c{V},\c{V}[i])=0$ for all $i\neq 0$, such that $\thick(\c{V})=\Perf(Y)$.  Tilting theory tells us that these properties are in fact sufficient. 

\begin{defin}
We say that $\c{V}\in\Perf(Y)$ is a tilting complex if $\Hom_{\Db(\coh Y)}(\c{V},\c{V}[i])=0$ for all $i\neq 0$, and further $\thick(\c{V})=\Perf(Y)$.  If further $\c{V}$ is a vector bundle (not just a complex), then we say that $\c{V}$ is a tilting bundle.
\end{defin}

The following is stated for the case when $\c{V}$ is a vector bundle (not a complex), since in these notes this is all that is needed.

\begin{thm}\label{VdB Hille}
With our running hypothesis on $Y$ (namely it is quasi--projective over a commutative noetherian ring of finite type over $\mathbb{C}$), assume that $\c{V}$ is a tilting bundle.  Then\\
(1) $\RHom_Y(\c{V},-)$ induces an equivalence between $\Db(\coh Y)$ and $\Db(\mod \End_Y(\c{V}))$.\\
(2) $Y$ is smooth if and only if $\gl\End_Y(\c{V})<\infty$.
\end{thm}

In practice, to check the Ext vanishing in the definition of a tilting bundle can be quite mechanical, whereas establishing generation is more of an art.  Below, we will often use the following trick to simplify calculations.

\begin{prop}\label{Neeman trick}
(Neeman's Generation Trick).  Say $Y$ has an ample line bundle $\c{L}$.  Pick $\c{V}\in\Perf(Y)$.  If $(\c{L}^{-1})^{\otimes n}\in\thick(\c{V})$ for all $n\geq 1$, then $\thick(\c{V})=\Perf(Y)$.
\end{prop}

\subsection{Tilting Examples}  We now illustrate tilting in the three examples from the previous section on quiver GIT, namely $\mathbb{P}^1$, the blowup of $\mathbb{A}^2$ at the origin, then our running $\mathbb{Z}_3$ example.  This will explain where the algebras used in \S\ref{L3} arose.

\begin{example}\label{P1 Db}
Consider $\mathbb{P}^1$.  We claim that $\c{V}:=\c{O}_{\mathbb{P}^1}\oplus\c{O}_{\mathbb{P}^1}(1)$ is a tilting bundle.  First, we have $\Ext^i_{\mathbb{P}^1}(\c{V},\c{V})\cong \H^i(\c{V}^{-1}\otimes\c{V})=\H^i(\c{O}_{\mathbb{P}^1})\oplus\H^i(\c{O}_{\mathbb{P}^1}(1))\oplus\H^i(\c{O}_{\mathbb{P}^1}(-1))\oplus\H^i(\c{O}_{\mathbb{P}^1})$, which is zero for all $i>0$ by a \v{C}ech cohomology calculation in Hartshorne \cite[III.5]{H}.  Thus $\Ext^i_{\mathbb{P}^1}(\c{V},\c{V})=0$ for all $i>0$. 

Now we use Neeman's generation trick (\ref{Neeman trick}).  We know that $\c{O}_{\mathbb{P}^1}(1)$ is an ample line bundle on $\mathbb{P}^1$.  Further, we have the Euler short exact sequence
\begin{eqnarray}
0\to\c{O}_{\mathbb{P}^1}(-1)\to\c{O}_{\mathbb{P}^1}^{\oplus 2}\to \c{O}_{\mathbb{P}^1}(1)\to 0, \label{E1}
\end{eqnarray}
which gives a triangle in the derived category.  Since the rightmost two terms both belong to $\thick(\c{V})$, by the 2 out of 3 property we deduce that $\c{O}_{\mathbb{P}^1}(-1)\in\thick(\c{V})$.  Now twisting (\ref{E1}) we obtain another short exact sequence
\begin{eqnarray}
0\to\c{O}_{\mathbb{P}^1}(-2)\to\c{O}_{\mathbb{P}^1}(-1)^{\oplus 2}\to \c{O}_{\mathbb{P}^1}\to 0.
\end{eqnarray}
Again this gives a triangle in the derived category, and since the rightmost two terms both belong to $\thick(\c{V})$, by the 2 out of 3 property we deduce that $\c{O}_{\mathbb{P}^1}(-2)\in\thick(\c{V})$.  Continuing like this we deduce that $\c{O}_{\mathbb{P}^1}(-n)\in\thick(\c{V})$ for all $n\geq 1$, and so $\thick(\c{V})=\Perf(\mathbb{P}^1)$ by \ref{Neeman trick}. 

Thus $\c{V}$ is a tilting bundle, so by \ref{VdB Hille} we deduce that $\Db(\coh \mathbb{P}^1)\simeq \Db(\mod\End_{\mathbb{P}^1}(\c{V}))$. We now identify the endomorphism ring with an algebra with which we are more familiar.  We have
\[
\End_{\mathbb{P}^1}(\c{V})=\End_{\mathbb{P}^1}(\c{O}\oplus\c{O}(1))\cong \begin{pmatrix} \Hom_{\mathbb{P}^1}(\c{O},\c{O})&\Hom_{\mathbb{P}^1}(\c{O},\c{O}(1))\\ \Hom_{\mathbb{P}^1}(\c{O}(1),\c{O})&\Hom_{\mathbb{P}^1}(\c{O}(1),\c{O}(1))\end{pmatrix}
\]
which, again by the \v{C}ech cohomology calculation in Hartshorne, is isomorphic to
\[
\begin{pmatrix} \H^0(\c{O})&\H^0(\c{O}(1))\\ \H^0(\c{O}(-1))&\H^0(\c{O})\end{pmatrix} 
\cong
\begin{pmatrix} \mathbb{C}&\mathbb{C}^2\\ 0&\mathbb{C}\end{pmatrix}\cong 
\begin{array}{c}
\begin{tikzpicture} 
\node (C1) at (0,0) [vertex] {};
\node (C2) at (1,0) [vertex] {};
\draw [->] ($(C1)+(30:4pt)$) -- ($(C2)+(150:4pt)$);
\draw [->] ($(C1)+(-30:4pt)$) -- ($(C2)+(210:4pt)$);
\end{tikzpicture}
\end{array}
\]
Thus  $\Db(\coh \mathbb{P}^1)\simeq \Db(\mod\begin{array}{c}
\begin{tikzpicture} 
\node (C1) at (0,0) [vertex] {};
\node (C2) at (1,0) [vertex] {};
\draw [->] ($(C1)+(30:4pt)$) -- ($(C2)+(150:4pt)$);
\draw [->] ($(C1)+(-30:4pt)$) -- ($(C2)+(210:4pt)$);
\end{tikzpicture}
\end{array})$.
\end{example}

\begin{example}\label{Z3 Db}
Consider now the blowup of $\mathbb{A}^2$ at the origin.  
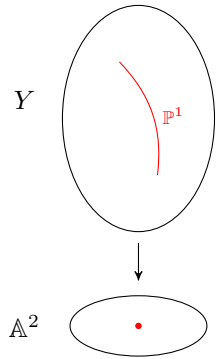
\[
\begin{array}{c}
\\
\begin{tikzpicture}   
\draw[red] (0,0) to [bend right=25] node[pos=0.48, right] {$\scriptstyle \mathbb{P}^1$} (-0.5,1.5);
\draw[black] (-0.25,0.75) ellipse (1 and 1.5);
\draw[->] (-0.25,-0.9) -- (-0.25,-1.4);
\draw[black] (-0.25,-2) ellipse (0.9 and 0.4);
\filldraw [red] (-0.25,-2) circle (1pt);
\node at (-1.75,-2) {$\mathbb{A}^2$};
\node at (-1.75,1) {$Y$};
\end{tikzpicture} 
\end{array}
\]
We constructed $Y$ explicitly in \ref{blowup GIT}, where we remarked that $Y=\c{O}_{\mathbb{P}^1}(-1)$.  Being the total space of a line bundle over $\mathbb{P}^1$, in this example we have extra information in the form of the diagram
\[
\begin{tikzpicture}
\node (A) at (0,0) {$Y=\c{O}_{\mathbb{P}^1}(-1)$};
\node (B) at (2,0) {$\mathbb{P}^1$};
\node (A1) at (0,-1.5) {$\mathbb{A}^2$};
\draw[->] (A) -- node[above] {$\scriptstyle \pi$} (B);
\draw[->] (A) -- node[left] {$\scriptstyle f$} (A1);
\end{tikzpicture}
\]
Let $\c{V}:=\c{O}_{\mathbb{P}^1}\oplus\c{O}_{\P}(1)$, as in the previous example, and set $\c{W}:=\pi^*(\c{V})$.  We claim that $\c{W}$ is a tilting bundle on $Y$.  To visualize this (we will need to in the next example), we denote $\c{W}=\c{O}_Y\oplus\c{L}_1$ and draw 
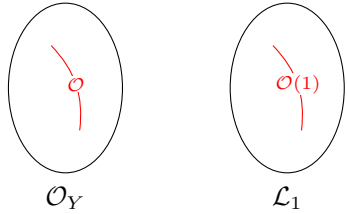
\[
\begin{array}{ccc}
\begin{array}{c}
\\
\begin{tikzpicture}[xscale=0.75,yscale=0.75]   
\draw[red] (0,0) to [bend right=25] node[gap] {$\scriptstyle \c{O}$} (-0.5,1.5);
\draw[black] (-0.25,0.75) ellipse (1 and 1.5);
\end{tikzpicture} 
\end{array}
&\quad&
\begin{array}{c}
\\
\begin{tikzpicture}[xscale=0.75,yscale=0.75]
\draw[red] (0,0) to [bend right=25] node[gap] {$\scriptstyle \c{O}(1)$} (-0.5,1.5);
\draw[black] (-0.25,0.75) ellipse (1 and 1.5);
\end{tikzpicture} 
\end{array}\\
\c{O}_Y&&\c{L}_1
\end{array}
\]
Since $\pi_*(\c{O}_Y)\cong \bigoplus_{p\leq 0}\c{O}_{\P}(-1)^{\otimes p}=\bigoplus_{k\geq 0}\c{O}_{\P}(k)$, by the projection formula $\pi_*\pi^*(\c{V})\cong \bigoplus_{k\geq 0}\c{V}\otimes_{\P}\c{O}_{\P}(k)$.  Thus the Ext vanishing condition on $\c{W}$ follows from properties of adjoint functors, namely
\begin{eqnarray*}
\Ext^i_Y(\c{W},\c{W})&=&\Ext^i_Y(\pi^*(\c{V}),\pi^*(\c{V}))\\
&\cong&\Ext^i_{\mathbb{P}^1}(\c{V},\pi_*\pi^*(\c{V}))\\
&\cong&\bigoplus_{k\geq 0}\Ext^i_{\P}(\c{V},\c{V}\otimes_{\P}\c{O}_{\P}(k))
\end{eqnarray*}
which is zero for all $i>0$ again by the \v{C}ech cohomology calculation in Hartshorne \cite[III.5]{H}.  Generation also follows immediately from our previous example, since $\c{L}_1$ is ample and $\pi^*$ is exact on short exact sequences of vector bundles.  Thus $\c{W}=\c{O}_Y\oplus\c{L}_1$ is a tilting bundle, so by \ref{VdB Hille} $Y$ and $\End_Y(\c{O}_Y\oplus \c{L}_1)$ are derived equivalent.   It is an instructive exercise (see Exercise~\ref{Ex4.34}) to show that
\[
\End_Y(\c{O}_Y\oplus \c{L}_1)\cong \begin{array}{cc}
\begin{array}{c}
\begin{tikzpicture} [bend angle=45, looseness=1]
\node (C1) at (0,0) [vertex] {};
\node (C2) at (1.5,0)  [vertex] {};
\draw [->,bend left] (C1) to node[gap]  {$\scriptstyle a$} (C2);
\draw [->,bend left=20,looseness=1] (C1) to node[gap]  {$\scriptstyle b$} (C2);
\draw [->,bend left=20,looseness=1] (C2) to node[gap]  {$\scriptstyle t$} (C1);
\end{tikzpicture}
\end{array} &
atb=bta.
\end{array}
\]
\end{example}

\begin{example}
Consider the running example $R:=\mathbb{C}[a,b,c]/(ab-c^3)$.  We have, by \ref{Z3 GIT}, the following picture
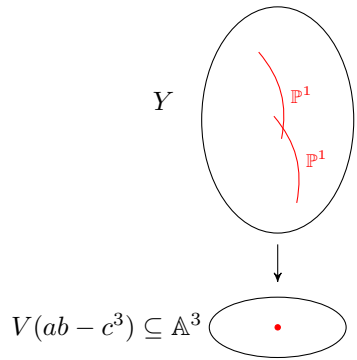
\[
\begin{array}{c}
\\
\begin{tikzpicture}   
\draw[red] (-0.2,0.5) to [bend right=25] node[pos=0.48, right] {$\scriptstyle \mathbb{P}^1$} (-0.5,1.65);
\draw[red] (0,-0.35) to [bend right=25] node[pos=0.48, right] {$\scriptstyle \mathbb{P}^1$} (-0.3,0.8);
\draw[black] (-0.25,0.75) ellipse (1 and 1.5);
\draw[->] (-0.25,-0.9) -- (-0.25,-1.4);
\draw[black] (-0.25,-2) ellipse (0.9 and 0.4);
\filldraw [red] (-0.25,-2) circle (1pt);
\node at (-2.5,-2) {$V(ab-c^3)\subseteq\mathbb{A}^3$};
\node at (-1.75,1) {$Y$};
\end{tikzpicture} 
\end{array}
\]
where the dot downstairs represents the singular point.
We would like $Y$ to be derived equivalent to our original algebra $\End_R(R\oplus (a,c)\oplus (a,c^2))$, so we need to find a tilting bundle on $Y$ with three summands.  Which to choose? We have to construct bundles, and the natural candidates are
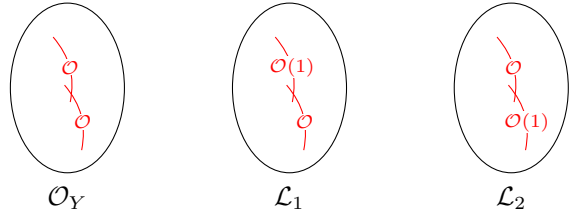
\[
\begin{array}{ccccc}
\begin{array}{c}
\\
\begin{tikzpicture}[xscale=0.75,yscale=0.75]   
\draw[red] (-0.2,0.5) to [bend right=25] node[gap] {$\scriptstyle \c{O}$} (-0.5,1.65);
\draw[red] (0,-0.35) to [bend right=25] node[gap, pos=0.4] {$\scriptstyle \c{O}$} (-0.3,0.8);
\draw[black] (-0.25,0.75) ellipse (1 and 1.5);
\end{tikzpicture} 
\end{array}
&\quad&
\begin{array}{c}
\\
\begin{tikzpicture}[xscale=0.75,yscale=0.75]
\draw[red] (-0.2,0.5) to [bend right=25] node[gap] {$\scriptstyle \c{O}(1)$} (-0.5,1.65);
\draw[red] (0,-0.35) to [bend right=25] node[gap,pos=0.4] {$\scriptstyle \c{O}$} (-0.3,0.8);
\draw[black] (-0.25,0.75) ellipse (1 and 1.5);
\end{tikzpicture} 
\end{array}
&\quad&
\begin{array}{c}
\\
\begin{tikzpicture}[xscale=0.75,yscale=0.75]
\draw[red] (-0.2,0.5) to [bend right=25] node[gap] {$\scriptstyle \c{O}$} (-0.5,1.65);
\draw[red] (0,-0.35) to [bend right=25] node[gap,pos=0.4] {$\scriptstyle \c{O}(1)$} (-0.3,0.8);
\draw[black] (-0.25,0.75) ellipse (1 and 1.5);
\end{tikzpicture} 
\end{array}
\\
\c{O}_Y&&\c{L}_1&&\c{L}_2
\end{array}
\]
It turns out, but more technology is needed to prove it, that $\c{W}:=\c{O}_Y\oplus\c{L}_1\oplus\c{L}_2$ is a tilting bundle on $Y$ with $\End_Y(\c{O}_Y\oplus\c{L}_1\oplus\c{L}_2)\cong\End_R(R\oplus (a,c)\oplus (a,c^2))$.
\end{example}

\subsection{Derived Categories and Crepant Resolutions}
In this section we explain the geometric origin of the condition $\End_R(M)\in\CM R$, and also why we only consider rings of the form $\End_R(M)$, answering Q4 from \S\ref{L1}.

The key theorem is the following.

\begin{thm}\label{CMiffcrepant}
Suppose that $f\colon Y\to\Spec R$ is a projective birational map, where $Y$ and $R$ are both normal Gorenstein of dimension $d$.  If $Y$ is derived equivalent to a ring $\Lambda$, then the following are equivalent.\\
\t{(1)} $f$ is crepant (i.e.\ $f^*\omega_R=\omega_Y$).\\
\t{(2)} $\Lambda\in\CM R$.\\
In this case $\Lambda\cong\End_R(M)$ for some $M\in\refl R$. 
\end{thm}

It follows that the only possible algebras derived equivalent to crepant partial resolutions of $d$-dimensional Gorenstein singularities have the form $\End_R(M)$, and ultimately this is the reason why we restrict to studying rings of this form.  We outline the main ingredients of the proof below.

\begin{remark}
By \ref{CMiffcrepant} the condition $\End_R(M)\in\CM R$ corresponds precisely to the geometric notion of crepancy, provided that we can actually find some $Y$ which is derived equivalent to $\End_R(M)$.  This explains the first condition in the definition of a NCCR.
\end{remark}

\begin{remark} 
Suppose that $\End_R(M)$ is a NCCR.  We remark that when $d=\dim R\geq 4$ there may be no scheme $Y$ projective birational over $\Spec R$ for which $Y$ is derived equivalent to $\End_R(M)$.  This is because NCCRs can exist even when commutative crepant resolutions do not.  A concrete example is invariants by the group $\frac{1}{2}(1,1,1,1)$.  Also, when $d\geq 4$ it is possible that a crepant resolution $Y$ exists but there is no algebra that is derived equivalent to $Y$.  Thus the correspondence between NCCRs and crepant resolutions breaks down completely, even when $d=4$.  Thus \ref{CMiffcrepant} is usually only used for analogies (or very specific situations) in high dimension.  
\end{remark}

Theorem~\ref{CMiffcrepant} has the following important corollary.

\begin{cor}\label{NCCRdimd}
Suppose that $Y\to\Spec R$ is a projective birational map between $d$-dimensional Gorenstein normal varieties.  If $Y$ is derived equivalent to a ring $\Lambda$, then the following are equivalent.\\
(1) $\Lambda$ is an NCCR.\\
(2) $Y$ is a crepant resolution of $\Spec R$.
\end{cor}
\begin{proof}
(1)$\Rightarrow$(2)  If $\Lambda$ is an NCCR then $\Lambda$ has finite global dimension.  This means that $\Kb(\proj\Lambda)=\Db(\mod\Lambda)$ and so via the derived equivalence $\Perf(Y)\simeq\Db(\coh Y)$.  Hence $Y$ is smooth.   Further $f$ is crepant by \ref{CMiffcrepant}. \\
(2)$\Rightarrow$(1)  Since $f$ is crepant, $\Lambda\cong\End_R(M)\in\CM R$ by \ref{CMiffcrepant}.  Since $Y$ is smooth $\Perf(Y)\simeq\Db(\coh Y)$ induces $\Kb(\proj\Lambda)=\Db(\mod\Lambda)$ and so every $N\in\mod\Lambda$ has $\pd_\Lambda N<\infty$.  By Auslander--Buchsbaum, necessarily $\pd_\Lambda N\leq \dim R$ and so $\Lambda\cong\End_R(M)$ is a NCCR.
\end{proof}

\begin{example}
To make this a little more concrete, consider our running example $R:=\mathbb{C}[a,b,c]/(ab-c^3)$.  We know from \S\ref{L1} and \S\ref{L2} that $\Lambda:=\End_R(R\oplus (a,c)\oplus (a,c^3))\in\CM R$ is a NCCR, and we constructed, using quiver GIT (\ref{Z3 GIT}), a space $Y$ which we then showed (\ref{Z3 Db}) is derived equivalent to $\Lambda$.  Since $\Lambda\in\CM R$, we can deduce from \ref{CMiffcrepant} that $Y$ is a crepant resolution.  Alternatively, knowing a bit of geometry,  the two exceptional curves calculated in \ref{Z3 GIT} are both ($-2$)-curves, so it follows that $Y\to\Spec R$ is crepant.  Hence we could alternatively use \ref{NCCRdimd} to give another proof that $\Lambda$ is a NCCR.
\end{example}

The main ingredient in the proof of \ref{CMiffcrepant} is Grothendieck duality and relative Serre functors, which we now review.

\subsection{Relative Serre Functors}
 The following is based on \cite[7.2.6]{GinzCY}.

\begin{defin}
Suppose that $Z\to\Spec T$ is a morphism where $T$ is CM ring with a canonical module $C_T$.  We say that a functor $\mathbb{S}\colon \Perf(Z)\to\Perf(Z)$ is a \emph{Serre functor relative to $C_T$} if there are functorial isomorphisms
\[
\RHom_T(\RHom_Z(\c{F},\c{G}),C_T)\cong \RHom_Z(\c{G},\mathbb{S}(\c{F}))
\]
in $\D(\Mod T)$ for all $\c{F}\in\Perf(Z)$, $\c{G}\in\Db(\coh Z)$.  If $\Lambda$ is a module-finite $T$-algebra, we define a Serre functor $\mathbb{S}\colon\Kb(\proj\Lambda)\to\Kb(\proj\Lambda)$ relative to $C_T$ in a similar way.
\end{defin}

The ability to consider different canonicals in the above definition is convenient when comparing the geometry to the algebra.  For example, when $T$ is Gorenstein, there is a geometrically--defined canonical module $\omega_T$, but $T$ itself is also a canonical module.  It turns out that from the NCCR perspective $T$ is the most natural (see \ref{CY general}(1)), whereas the notion of crepancy is defined with respect to $\omega_T$.

\begin{lemma}\label{geometry Serre}
Suppose that $f\colon Z\to\Spec T$ is a projective morphism, where $Z$ and $T$ are both Gorenstein varieties.  Then\\
\t{(1)} $f^!(\c{O}_T)\cong \c{L}[\dim Z-\dim T]$ where $\c{L}$ is some line bundle on $Z$.\\
\t{(2)} $\mathbb{S}_Z:=-\otimes_Z f^!\c{O}_T\colon \Perf(Z)\to\Perf(Z)$ is a Serre functor relative to the canonical $T$.
\end{lemma}
\begin{proof}
(1) Since $T$ is Gorenstein, $\omega_T$ is a line bundle and thus is a compact object in $\D(\Mod T)$.  Hence by \cite[p227--228]{Neeman} we have $f^!\omega_T=\mathbf{L}f^* \omega_T\otimes^{\bf L}_Z f^!\c{O}_T=f^*\omega_T\otimes_Z f^!\c{O}_T$ and so
\[
\omega_Z\cong f^!\omega_T[-\dim Z+\dim T]\cong f^*\omega_T\otimes_Z f^!\c{O}_T[-\dim Z+\dim T].
\]
Since both $\omega_Y$ and $f^*\omega_T$ are line bundles, 
\[
f^!\c{O}_T=(f^*\omega_T)^{-1}\otimes_Z\omega_Z[\dim Z-\dim T].
\]
(2) Since $f^!\c{O}_T$ is a shift of a bundle by (1), it follows that tensoring gives a functor $-\otimes_Zf^!\c{O}_T\colon \Perf(Z)\to\Perf(Z)$.  The result then follows since
\begin{align*}
\RHom_Z(\c{G},\c{F}\otimes_Z f^!\c{O}_T)
&\cong\RHom_Z(\mathbf{R}\c{H}om_Z(\c{F},\c{G}),f^!\c{O}_T)\\
&\cong\RHom_T(\RHom_Z(\c{F},\c{G}),\c{O}_T)
\end{align*}
for all $\c{F}\in\Perf(Z)$, $\c{G}\in\Db(\coh Z)$, where the last isomorphism is sheafified Grothendieck duality.
\end{proof}

We consider the Serre functor for NCCRs in the next section.

\subsection{Calabi--Yau Categories}
Related to Serre functors are CY categories and algebras.
We retain our setup, namely $R$ denotes an (equicodimensional) Gorenstein normal domain of dimension $d$.  To keep the technicalities to a minimum, in this section we will assume that $R$ is of finite type over an algebraically closed field $k$, but we could get by with much less.  In this section we show that NCCRs are $d$-CY.

\begin{defin}
Suppose that $\c{C}$ is a triangulated category in which the Hom spaces are all $k$-vector spaces.  We say that $\c{C}$ is $d$-CY if there exists a functorial isomorphism
\begin{eqnarray}
\Hom_{\c{C}}(x,y[d])\simeq\Hom_{\c{C}}(y,x)^*\label{CY iso}
\end{eqnarray}
for all $x,y\in\c{C}$, where $(-)^*$ denotes the $k$-dual.
\end{defin}

For a $k$-algebra $\Lambda$, we naively ask whether $\Db(\mod \Lambda)$ is $d$-CY.  It is (almost) never in the strict sense above, since then
\[
\Lambda^*\cong\Hom_\Lambda(\Lambda,\Lambda)^*\cong \Hom_{\Db}(\Lambda,\Lambda)^*\cong \Hom_{\Db}(\Lambda,\Lambda[d])\cong\Ext_\Lambda^d(\Lambda,\Lambda)=0
\]
whenever $d\neq 0$.  Also, note that for $k$-duality to work well requires the Hom spaces to be finite dimensional.  Hence, for an algebra to be CY, we must ask for (\ref{CY iso}) to be true for only certain classes of objects $x$ and $y$.  This is done as follows.

\begin{defin}
Let $\Lambda$ be a module finite $R$-algebra, then for $d\in\mathbb{Z}$ we call $\Lambda$ $d$-Calabi--Yau (=$d$-CY) if there is a functorial isomorphism 
\begin{eqnarray}
\Hom_{\Db(\mod \Lambda)}(x,y[d])\cong \Hom_{\Db(\mod \Lambda)}(y,x)^*\label{CY alg iso}
\end{eqnarray}
for all $x\in\Db(\fl \Lambda)$, $y\in\Db(\mod\Lambda)$, where $\Db(\fl \Lambda)$ denotes all complexes $x$ for which $\dim_k\oplus_{i\in\mathbb{Z}}\H^i(x)<\infty$ .  Similarly we call $\Lambda$ singular $d$-Calabi--Yau (=$d$-sCY) if (\ref{CY alg iso}) holds for all $x\in\Db(\fl\Lambda)$ and $y\in\Kb(\proj \Lambda)$.
\end{defin}

Since $R$ is Gorenstein, and $\Lambda$ is a module-finite $R$-algebra, there is a functor
\[
\mathbb{S}_\Lambda:=\RHom_R(\Lambda,R)\otimes^{\mathbf{L}}_\Lambda-\colon \D^-(\mod\Lambda)\to\D^-(\mod\Lambda).
\]
By \cite[3.5(2)(3)]{IR}, there exists a functorial isomorphism
\begin{eqnarray}
\RHom_\Lambda(a,\mathbb{S}(b))\cong \RHom_R(\RHom_\Lambda(b,a),R)\label{3.5IR}
\end{eqnarray}
in $\D(R)$ for all $a\in\Db(\mod\Lambda)$ and all $b\in\Kb(\proj\Lambda)$.  This is not quite a Serre functor, since we don't yet know whether $\mathbb{S}_\Lambda$ preserves $\Kb(\proj\Lambda)$.

\begin{thm}\label{CY general}
(Iyama--Reiten) Let $R$ an be equicodimensional Gorenstein normal domain over an algebraically closed field $k$, and let $\Lambda$ be an NCCR.   Then\\
\t{(1)} $\mathbb{S}_\Lambda=\mathrm{Id}$, and so $\mathrm{Id}$ is a Serre functor on $\Lambda$ relative to the canonical $R$.\\
\t{(2)} $\Lambda$ is $d$-CY, that is
\[
\Hom_{\Db(\mod \Lambda)}(x,y[d])\cong \Hom_{\Db(\mod \Lambda)}(y,x)^*
\]
for all $x\in\Db(\fl\Lambda)$ and $y\in\Kb(\proj \Lambda)=\Db(\mod\Lambda)$.
\end{thm}
\begin{proof}
(1) By definition of $\mathbb{S}_\Lambda$, we just need to establish that $\RHom_R(\Lambda,R)\cong\Lambda$ as $\Lambda$-$\Lambda$ bimodules.  But by the definition of an NCCR, $\Lambda\in\CM R$ and so $\Ext^i_R(\Lambda,R)=0$ for all $i>0$.  This shows that $\RHom_R(\Lambda,R)\cong \Hom_R(\Lambda,R)$.  The fact that $\Hom_R(\Lambda,R)\cong \Lambda$ as $\Lambda$-$\Lambda$ bimodules follows from the fact that symmetric $R$-algebras are closed under reflexive equivalences (this uses the fact that $R$ is normal --- see fact \ref{facts3and4}(\ref{F3})).  We conclude that $\mathbb{S}_\Lambda=\mathrm{Id}$, and this clearly preserves $\Kb(\proj\Lambda)$.  Thus \eqref{3.5IR} shows that $\mathrm{Id}$ is a Serre functor relative to the canonical $R$.\\
(2) This is in fact a consequence of (1), using local duality.  See \cite[3.6, 3.7]{IR}.
\end{proof}

\subsection{Singular Derived Categories}  From before recall that $\Kb(\proj R)\subseteq\Db(\mod R)$.  

\begin{defin}
We define the singular derived category (sometimes called the triangulated category of singularities, or the singularity category) $\Dsg(R)$ to be the quotient category $\Db(\mod R)/\Kb(\proj R)$.  
\end{defin}

Since $\Kb(\proj R)$ is a full thick triangulated subcategory, for general abstract reasons the quotient $\Dsg(R)$ is also triangulated.  Also, being a localization, morphisms in $\Dsg(R)$ are equivalence classes of morphisms in $\Db(\mod R)$.  But the derived category is itself a localization, so the morphisms in $\Dsg(R)$ are equivalence classes of equivalence classes.  From this perspective, there is no reason to expect that this category should behave well.

If $R$ is a Gorenstein ring, it is thus remarkable that $\Dsg(R)$ can be described easily by using CM $R$-modules.  There is a natural functor
\[
\begin{tikzpicture}
\node (A) at (0,0) {$\CM R$};
\node (B) at (2.5,0) {$\Db(\mod R)$};
\node (C) at (7.25,0) {$\Dsg(R)=\Db(\mod R)/\Kb(\proj R).$};
\draw[->] (A) -- (B);
\draw[->] (B) -- (C);
\end{tikzpicture}
\]
This is not an equivalence, since projective modules $P$ are CM, and these get sent to complexes which by definition are zero in $\Dsg(R)$.  Hence we must `remove' the projectives in $\CM R$.

How to do this is standard, and is known as taking the stable category $\uCM R$.  By definition the objects in $\uCM R$ are just the same objects as in $\CM R$, but morphism spaces are defined as $\u{\Hom}_{R}(X, Y):=\Hom_{R}(X, Y)/\mathcal{P}(X, Y)$ where $\mathcal{P}(X, Y)$ is the subspace of morphisms factoring through $\proj R$.  If $P\in\proj R$ then $id_P$ clearly factors through a projective and so $id_P=0_P$ in $\uCM R$.  This shows that $P\simeq 0$ in $\uCM R$ for all $P\in\proj R$, and consequently the above functor induces a functor
\[
\begin{tikzpicture}
\node (A) at (0,0) {$\uCM R$};
\node (B) at (4,0) {$\Dsg(R)=\Db(\mod R)/\Kb(\proj R).$};
\draw[->] (A) -- node[above] {$\scriptstyle F$}(B);
\end{tikzpicture}
\]
When $R$ is Gorenstein, the category $\uCM R$ is actually triangulated, being the stable category of a Frobenius category.  The shift functor $[1]$ is given by the inverse of the syzygy functor.

\begin{thm} (Buchweitz) Let $R$ be a Gorenstein ring.  Then the natural functor $F$ above is a triangle equivalence, so  $\uCM R\simeq \Dsg(R)$ as triangulated categories.
\end{thm}

This shows, at least when $R$ is Gorenstein, why the CM modules encode much of the singular behaviour of $R$.  

\begin{remark}
When $R$ is Gorenstein with only isolated singularities, the triangulated category $\uCM R$ is in fact $(\dim R-1)$-CY.  This follows from Auslander--Reiten (=AR) duality in \ref{ARduality} below.  The existence of AR duality links to NCCRs via the CY property, it can be used to prove the existence of AR sequences (which in turn answers the motivating Q1 from \S\ref{L1}), and will appear again in the McKay Correspondence in \S\ref{L5}.
\end{remark}
\begin{remark}
If $R$ is a hypersurface, we have already seen in examples (and it is true generally via matrix factorizations) that $\Omega^2=\mathrm{Id}$, where $\Omega$ is the syzygy functor introduced in \S\ref{L1}.  Consequently $[1]^2=\mathrm{Id}$, so $\uCM R$ can be 2-CY, 4-CY, etc.  This shows that the precise value of the CY property is not unique. 
\end{remark}

\subsection{Auslander--Reiten Duality}\label{AR proof section}
\begin{thm}\label{ARduality} (AR duality)
Let $R$ be a $d$-dimensional equicodimensional Gorenstein ring, with only isolated singularities.  Then there exists a functorial isomorphism
\[
\Hom_{\uCM R}(X,Y)\cong D(\Hom_{\uCM R}(Y,X[d-1]))
\]
for all $X,Y\in\CM R$.
\end{thm}
The proof is actually quite straightforward, using only fairly standard homological constructions, given the following two commutative algebra facts.  For a finitely generated $R$-module $M$, we denote $E_{R}(M)$ to be the injective hull of $M$.
\begin{facts}\label{ARfacts} Let $R$ be a $d$-dimensional equi-codimensional Gorenstein ring.  Then \\
\t{(1)} The minimal $R$-injective resolution of $R$ has the form
\begin{equation}\label{min inj res}
0\to R\to I_0:=\bigoplus_{\p: \hgt\p=0} E(R/\p)\to \hdots\to I_d:=\bigoplus_{\p:\hgt\p=d} E(R/\p)\to 0.
\end{equation}
In particular the Matlis dual is $D=\Hom_R(-,I_d)$.  \\
\t{(2)} If $W\in\mod R$ with $\dim W=0$, then $\Hom_R(W,I_i)=0$ for all $0\leq i\leq d-1$.
\end{facts}
\begin{proof}
(1) is one of Bass's original equivalent characterizations of Gorenstein rings \cite{Bass}.  For the more general CM version, see \cite[3.2.9, 3.3.10(b)]{BH}.\\
(2) follows from (1), together with the knowledge that (i) $\Ass E_{R}(R/\p)=\{ \p\}$, and (ii) if $X\in \mod R$ and $Y\in \Mod R$ satisfies $\Supp X\cap\Ass Y=\emptyset$, then $\Hom_R(X,Y)=0$.
\end{proof}

Set $(-)^*:=\Hom_R(-,R)$. For any $X\in\mod R$, consider the start of a projective resolution 
\[
P_1\stackrel{f}{\to}P_0\to X\to 0.
\]
We define $\Tr X\in\mod R$ to be the cokernel of $f^*$, that is
\[
0\to X^*\to P_0^*\stackrel{f^*}{\to}P_1^*\to\Tr X\to0.
\]
This gives a duality
\[
\Tr\colon \umod R\xrightarrow{\sim}\umod R
\]
called the Auslander--Bridger transpose.  We denote $\Omega\colon
\umod R\to\umod R$ to be the syzygy functor.  Combining these we define the Auslander--Reiten translation $\tau$ to be
\[
\tau(-):=\Hom_R(\Omega^d\Tr(-) ,R):
\uCM R\to\uCM R.
\]
We let $D:=\Hom(-,I_d)$ denote Matlis duality.

\begin{lemma}\label{tausyz} Suppose that $R$ is Gorenstein with $d:=\dim R$.  Then $\tau\cong\Omega^{2-d}$.
\end{lemma}
\begin{proof}
We have $\Omega^2\Tr(-)\cong \Hom_R(- ,R)$.
Thus
\begin{eqnarray*}
\tau=\Hom_R(\Omega^d\Tr(-) ,R)
&\cong& \Hom_R(\Omega^{d-2}\Hom_R(-,R),R)\\
&\cong&\Omega^{2-d}\Hom_R(\Hom_R(-,R),R)\\
&\cong&\Omega^{2-d}.
\end{eqnarray*}
\end{proof}

With these facts, we can prove \ref{ARduality}:
\begin{proof}
Denote $T:=\Tr X$.  Since $Y\in\CM R$, $\Ext^{i}_{R}(Y,R)=0$ for all $i>0$ and so applying $\Hom_R(Y,-)$ to (\ref{min inj res}) gives an exact sequence
\[
0\to {}_{R}(Y,R)\to {}_{R}(Y,I_0)\to {}_{R}(Y,I_1)\to\hdots\to {}_{R}(Y,I_{d-1})\to {}_{R}(Y,I_d)\to 0
\]
of $R$-modules, which we split into short exact sequences as
\[
\SelectTips{cm}{10}
\xy0;/r.275pc/:
\POS(5,0)*+{0}="0",(15,0)*+{{}_{R}(Y,R)}="1",(30,0)*+{{}_{R}(Y,I_0)}="2",(45,0)*+{{}_{R}(Y,I_1)}="3",(60,0)*+{{}_{R}(Y,I_2)}="4",(69,0)*+{\,}="4a",(69.5,0)*+{\,}="4b",(80,0)*+{{}_{R}(Y,I_{d-2})}="5",(97,0)*+{{}_{R}(Y,I_{d-1})}="6",(112,0)*+{{}_{R}(Y,I_{d})}="7",(121,0)*+{0}="8",(37.5,-6.5)*+{C_1}="b1",(52.5,-6.5)*+{C_2}="b2",(87.5,-6.5)*+{C_{d-1}}="b3"
\ar"0";"1"
\ar"1";"2"
\ar"2";"3"
\ar"3";"4"
\ar"4";"4a"
\ar@{}"4a";"4b"_{\hdots}
\ar"4b";"5"
\ar"5";"6"
\ar"6";"7"
\ar"7";"8"
\ar@{->>}"2";"b1"
\ar@{^{(}->}"b1";"3"
\ar@{->>}"3";"b2"
\ar@{^{(}->}"b2";"4"
\ar@{->>}"5";"b3"
\ar@{^{(}->}"b3";"6"
\endxy .
\]
Applying $\Hom_{R}(T,-)$ gives exact sequences
\[
\begin{tikzpicture}
\node (a1) at (0,0) {\small$\Ext^1_{R}(T,{}_{R}(Y,I_{d-1}))$};
\node (a2) at (3.5,0) {\small$\Ext^{1}_{R}(T,{}_{R}(Y,I_d))$};
\node (a3) at (6.5,0) {\small$\Ext^{2}_{R}(T,C_{d-1})$};
\node (a4) at (10,0) {\small$\Ext^{2}_{R}(T,{}_{R}(Y,I_{d-1}))$};
%line 2
\node (b1) at (0,-0.5) {\small$\Ext^2_{R}(T,{}_{R}(Y,I_{d-2}))$};
\node (b2) at (3.5,-0.5) {\small$\Ext^{2}_{R}(T,C_{d-1})$};
\node (b3) at (6.5,-0.5) {\small$\Ext^{3}_{R}(T,C_{d-2})$};
\node (b4) at (10,-0.5) {\small$\Ext^{3}_{R}(T,{}_{R}(Y,I_{d-2}))$};
%gap
\node at (6.5,-0.9) {$\vdots$};
%line 3
\node (c1) at (0,-1.5) {\small$\Ext^{d-1}_{R}(T,{}_{R}(Y,I_1))$};
\node (c2) at (3.5,-1.5) {\small$\Ext^{d-1}_{R}(T,C_2)$};
\node (c3) at (6.5,-1.5) {\small$\Ext^{d}_{R}(T,C_1)$};
\node (c4) at (10,-1.5) {\small$\Ext^{d}_{R}(T,{}_{R}(Y,I_1))$};
%line 4
\node (d1) at (0,-2) {\small$\Ext^{d}_{R}(T,{}_{R}(Y,I_0))$};
\node (d2) at (3.5,-2) {\small$\Ext^{d}_{R}(T,C_1)$};
\node (d3) at (6.5,-2) {\small$\Ext^{d+1}_{R}(T,{}_{R}(Y,R))$};
\node (d4) at (10,-2) {\small$\Ext^{d+1}_{R}(T,{}_{R}(Y,I_0))$};
%arrows
\draw[->] (a1) -- (a2);
\draw[->] (a2) -- (a3);
\draw[->] (a3) -- (a4);
\draw[->] (b1) -- (b2);
\draw[->] (b2) -- (b3);
\draw[->] (b3) -- (b4);
\draw[->] (c1) -- (c2);
\draw[->] (c2) -- (c3);
\draw[->] (c3) -- (c4);
\draw[->] (d1) -- (d2);
\draw[->] (d2) -- (d3);
\draw[->] (d3) -- (d4);
\end{tikzpicture}
\]
But whenever $I$ is an injective $R$-module, we have a functorial isomorphism
\[
\Ext^j_{R}(A,{}_{R}(B,I))\cong\Hom_R(\Tor^{R}_j(A,B),I)
\]
and so we may re-write the above as
\[
\begin{tikzpicture}
\node (a1) at (0,0) {\small${}_R(\Tor^{R}_1(T,Y),I_{d-1})$};
\node (a2) at (3.5,0) {\small${}_R(\Tor^{R}_1(T,Y),I_d)$};
\node (a3) at (6.5,0) {\small$\Ext^{2}_{R}(T,C_{d-1})$};
\node (a4) at (10,0) {\small${}_R(\Tor_{2}^{R}(T,Y),I_{d-1})$};
%line 2
\node (b1) at (0,-0.5) {\small${}_R(\Tor_{2}^{R}(T,Y),I_{d-2})$};
\node (b2) at (3.5,-0.5) {\small$\Ext^{2}_{R}(T,C_{d-1})$};
\node (b3) at (6.5,-0.5) {\small$\Ext^{3}_{R}(T,C_{d-2})$};
\node (b4) at (10,-0.5) {\small${}_R(\Tor_{3}^{R}(T,Y),I_{d-2})$};
%gap
\node at (6.5,-0.9) {$\vdots$};
%line 3
\node (c1) at (0,-1.5) {\small${}_R(\Tor_{d-1}^{R}(T,Y),I_1)$};
\node (c2) at (3.5,-1.5) {\small$\Ext^{d-1}_{R}(T,C_2)$};
\node (c3) at (6.5,-1.5) {\small$\Ext^{d}_{R}(T,C_1)$};
\node (c4) at (10,-1.5) {\small${}_R(\Tor_{d}^{R}(T,Y),I_1)$};
%line 4
\node (d1) at (0,-2) {\small${}_R(\Tor_{d}^{R}(T,Y),I_0)$};
\node (d2) at (3.5,-2) {\small$\Ext^{d}_{R}(T,C_1)$};
\node (d3) at (6.5,-2) {\small$\Ext^{d+1}_{R}(T,{}_{R}(Y,R))$};
\node (d4) at (10,-2) {\small${}_R(\Tor_{d+1}^{R}(T,Y),I_0)$};
%arrows
\draw[->] (a1) -- (a2);
\draw[->] (a2) -- (a3);
\draw[->] (a3) -- (a4);
\draw[->] (b1) -- (b2);
\draw[->] (b2) -- (b3);
\draw[->] (b3) -- (b4);
\draw[->] (c1) -- (c2);
\draw[->] (c2) -- (c3);
\draw[->] (c3) -- (c4);
\draw[->] (d1) -- (d2);
\draw[->] (d2) -- (d3);
\draw[->] (d3) -- (d4);
 %draw dotted box.  Start top left, then bottom left, the bottom right, the top right
 \draw [dotted] (-1.5,.5) -- (-1.5,-2.5) -- (1.5,-2.5) -- (1.5,.5) -- cycle;
  \draw [dotted] (8.5,.5) -- (8.5,-2.5) -- (11.5,-2.5) -- (11.5,.5) -- cycle;
\end{tikzpicture}
\]
Since $R$ is an isolated singularity, and CM modules are free on regular local rings, $X_\p\in\proj R_\p$ and so $T_{\p}\in\proj R_{\p}$ for all non-maximals primes $\p$.  It follows that $\Tor_i^R(T,Y)$ all have finite length, so by \ref{ARfacts}(2) all the terms in the dotted boxes are zero.  Thus the above reduces to 
\begin{eqnarray}
D(\Tor_1^R(T,Y))\cong\Ext_{R}^{2}(T,C_{d-1})\cong\hdots\cong  \Ext_{R}^{d+1}(T,{}_{R}(Y,R))\label{red1}
\end{eqnarray}
But now 
\begin{equation}\label{red2}
\Ext_{R}^{d+1}(T,{}_{R}(Y,R))\cong \Ext_{R}^{1}(\Omega^d T,{}_{R}(Y,R))\cong \Ext^1_{R}(Y,{}_{R}(\Omega^d T,R))=\Ext^1_R(Y,\tau X),
\end{equation}
and so combining (\ref{red1}) and (\ref{red2}) gives
\[
D(\Tor_1^R(T,Y))\cong \Ext^1_R(Y,\tau X).
\]
Thus the standard functorial isomorphism $\Tor^{R}_1(\Tr X,Y)\cong \Hom_{\uCM R}(X,Y)$ (see e.g.\ \cite{Y}) yields 
\[
D(\Hom_{\uCM R}(X,Y))\cong \Ext^1_R(Y,\tau X)\cong\Hom_{\uCM R}(\Omega Y,\tau X)=\Hom_{\uCM R}(Y[-1],\tau X).
\]
But now by \ref{tausyz} we have $\tau X=\Omega^{2-d}{X}=X[d-2]$, so
\[
D(\Hom_{\uCM R}(X,Y))\cong\Hom_{\uCM R}(Y[-1],X[d-2])=\Hom_{\uCM R}(Y,X[d-1]),
\]
as required.
\end{proof}

\medskip
\medskip
\noindent
{\bf Credits:}  The derived category section is a very brief summary of a course given by Jeremy Rickard in Bristol in 2005.  Further information on derived categories can now be found in many places, for example the notes of Keller \cite{Keller} or Milicic \cite{Milicic}.  Perfect complexes originate from SGA6, but came to prominence first through the work of Thomason--Trobaugh \cite{TT}, then by viewing them as compact objects via Neeman \cite{Neeman}.  Many of the technical results presented here can be found in Orlov \cite{Orlov1, Orlov2} and Rouquier \cite{Rouq}.  

Tilting modules existed in the 1980s, but tilting complexes first appeared in Rickard's Morita theory for derived categories \cite{Rickard}.  Tilting bundles at first were required to have endomorphism rings with finite global dimension but this is not necessary; the precise statement of \ref{VdB Hille} is taken from Hille--Van ben Bergh \cite{HVdB}, based on the tricks and ideas of Neeman \cite{Neeman}.  

Example \ref{P1 Db} is originally due to Beilinson \cite{Beilinson}, who proved it by resolving the diagonal.   The example of the blowup of the origin and the $\mathbb{Z}_3$ singularity both follow from Artin--Verdier \cite{AV}, Esnault \cite{Esnault}, and Wunram \cite{Wunram}, but those papers do not contain the modern derived category language.  A much more general setting was provided by Van den Bergh \cite{VdB1d}, which subsumes all these results.

Crepancy and the CM property in the smooth setting is used in \cite{VdBNCCR}, being somewhat implicit in \cite{BKR}.  The proof here (in the singular setting) can be found in \cite{IW5}, based on the ideas of relative Serre functors due to Ginzburg \cite{GinzCY}.  The fact that all such algebras have the form $\End_R(M)$ is really just a result of Auslander--Goldman \cite{AG}.  All the CY algebra section can be found in Iyama--Reiten \cite[\S3]{IR}. 

The singular derived category section is based on the very influential preprint of Buchweitz \cite{Buch}.  The proof of Auslander--Reiten duality can be found in \cite{Aus3}, and the technical background on injective hulls can be found in commutative algebra textbooks, e.g.\ \cite{BH}.

\subsection{Exercises}

\begin{ex}\label{Ex4.30}
(Torsion-free modules are not needed) Suppose that $R$ is a normal domain, and that $M$ is a torsion-free $R$-module.
\begin{enumerate}
\item Is $\End_R(M)\cong\End_R(M^{**})$
true in general?
\item Suppose further $\End_R(M)\in\refl R$ (e.g.\ if $R$ is CM, this happens when $\End_R(M)\in\CM R$).  In this case, show that $\End_R(M)\cong\End_R(M^{**})$.  Since $M^{**}\in\refl R$, this is why in the definition of NCCR we don't consider torsion-free modules.
\end{enumerate}
\end{ex}

\begin{ex}\label{Ex4.31}
(Much harder, puts restrictions on the rings that can admit NCCRs) Suppose that $R$ is a $d$-dimensional CM domain admitting an NCCR $\Lambda:=\End_R(R\oplus M)$.  Show that
\begin{enumerate}
\item The Azumaya locus of $\Lambda$ is equal to the non-singular locus of $R$.
\item The singular locus $\Sing(R)$ must be closed (!).
\item If further $(R,\m)$ is local, show that the class group of $R$ must be finitely generated.
\end{enumerate}
\end{ex}

\begin{ex}\label{Ex4.32}
Give an example of a CM ring that is not equicodimensional.
\end{ex}

\begin{ex}\label{Ex4.33}
(NCCR characterization on the base, without finite global dimension mentioned)  Let $R$ be a $d$-dimensional local Gorenstein ring.  Then for any $M\in \CM R$ with $R\in \add M$, show that the following are equivalent\\
(1)  $\add M=\{ X\in\CM R \mid \Hom_R(M,X)\in\CM R \}$.\\
(2)  $\add M=\{ X\in\CM R \mid \Hom_R(X,M)\in\CM R \}$.\\
(3) $\End_R(M)$ is a NCCR.\\
(4) $\End_R(M)^{\op}\cong\End_R(M^*)$ is a NCCR.\\
(Hint: the proof should be similar to the Auslander $\gl=2$ proof in \S\ref{L1}.  The key, as always, is the Auslander--Buchsbaum formula and the depth lemma).
\end{ex}

\begin{ex}\label{Z3 Db ex}\label{Ex4.34}
With the setting as in \ref{Z3 Db}, prove that
\[
\End_Y(\c{O}_Y\oplus \c{L}_1)\cong \begin{array}{cc}
\begin{array}{c}
\begin{tikzpicture} [bend angle=45, looseness=1]
\node (C1) at (0,0) [vertex] {};
\node (C2) at (1.5,0)  [vertex] {};
\draw [->,bend left] (C1) to node[gap]  {$\scriptstyle a$} (C2);
\draw [->,bend left=20,looseness=1] (C1) to node[gap]  {$\scriptstyle b$} (C2);
\draw [->,bend left=20,looseness=1] (C2) to node[gap]  {$\scriptstyle t$} (C1);
\end{tikzpicture}
\end{array} &
atb=bta.
\end{array}
\] 
\end{ex}

\newpage
\section{McKay and Beyond}\label{L5}

This section gives an overview of the McKay Correspondence in dimension two, in some of its various different forms.  This then leads into dimension three, where we sketch some results of Bridgeland--King--Reid \cite{BKR} and Van den Bergh \cite{VdBNCCR}.  The underlying message is that there is a very tight relationship between crepant resolutions and NCCRs in low dimension.

\subsection{McKay Correspondence (Surfaces)}
Let $G\leq\SL(2,\mathbb{C})$ be a finite subgroup, and consider $R:=\mathbb{C}[x,y]^G$.  So far we know that $R$ has a NCCR, since $\mathbb{C}[x,y]\# G\cong \End_R
\left(\bigoplus_{\rho\in\Irr G}((\mathbb{C}[x,y]\otimes\rho)^{G})^{\oplus\dim_\mathbb{C}\rho}\right)$ by Auslander (\ref{Auslander skew}), and we have already observed that as a quiver algebra this NCCR has the form
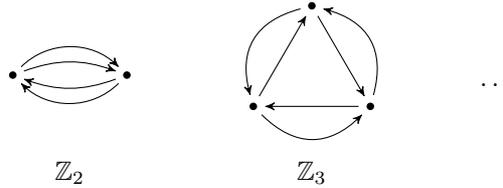
\[
\begin{array}{ccccc}
&&&\\
\begin{array}{c}
\begin{tikzpicture} [bend angle=45, looseness=1]
\node (C1) at (0,0) [vertex] {};
\node (C2) at (1.5,0)  [vertex] {};
\draw [->,bend left] (C1) to  (C2);
\draw [->,bend left=20,looseness=1] (C1) to  (C2);
\draw [->,bend left] (C2) to (C1);
\draw [->,bend left=20,looseness=1] (C2) to (C1);
\end{tikzpicture}
\end{array}
&&
\begin{array}{c}
\begin{tikzpicture} [bend angle=45, looseness=1.2]
\node[name=s,regular polygon, regular polygon sides=3, minimum size=1.75cm] at (0,0) {}; 
\node (C1) at (s.corner 1)  [vertex] {};
\node (C2) at (s.corner 2) [vertex] {};
\node (C3) at (s.corner 3) [vertex] {};
\draw[->] (C2) --  (C1); 
\draw[->] (C1) --  (C3);
\draw[->] (C3) --  (C2); 
\draw [->,bend right] (C1) to  (C2);
\draw [->,bend right] (C2) to  (C3);
\draw [->,bend right] (C3) to  (C1);
\end{tikzpicture}
\end{array}&\quad&\hdots\\
\mathbb{Z}_2&\quad& \mathbb{Z}_3
\end{array}
\]
The first is in the exercises (Exercise~\ref{Ex2.16}), the second is our running example.  There is a pattern, and to describe it requires the concept of the dual graph, which can be viewed as a simplified picture of the minimal resolution of $\Spec R$.
\begin{defin}
Denote by $\{ E_i\}$ the exceptional set of $\mathbb{P}^1$s in the minimal resolution $Y\to\Spec R$ (i.e.\  those $\mathbb{P}^1$s above the origin, as in \ref{Z3 GIT}).  Define the  dual graph as follows: for every $E_i$ draw a dot, and join two dots if the corresponding $\mathbb{P}^1$'s intersect.  
\end{defin}

\begin{example} For the $\mathbb{Z}_4$ singularity we obtain
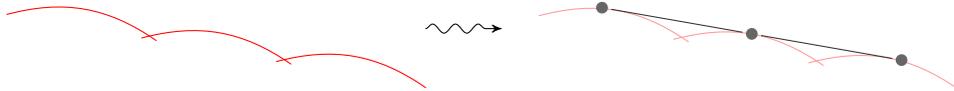
\[
\begin{array}{c}
\begin{tikzpicture}
\node at (0,0)
{\begin{tikzpicture} [bend angle=25, looseness=1,transform shape, rotate=-10]
\draw[red] (0,0,0) to [bend left=25] (2,0,0);
\draw[red] (1.8,0,0) to [bend left=25] (3.8,0,0);
\draw[red] (-1.8,0,0) to [bend left=25] (0.2,0,0);
%\node (0) at (1,-0.25,0) {};
\end{tikzpicture} };

\node at (7,0) {
\begin{tikzpicture} [bend angle=25, looseness=1,transform shape, rotate=-10]
\draw[red!40] (0,0,0) to [bend left=25] (2,0,0);
\draw[red!40] (1.8,0,0) to [bend left=25] (3.8,0,0);
\draw[red!40] (-1.8,0,0) to [bend left=25] (0.2,0,0);
%\filldraw [black] (0.1,0.045,0) circle (0.5pt);
%\filldraw [black] (1.9,0.045,0) circle (0.5pt);
%\filldraw [black] (-1.8,0,0) circle (0.5pt);
%\filldraw [black] (3.8,0,0) circle (0.5pt);
\filldraw [black!60] (1,0.25,0) circle (2pt);
\filldraw [black!60] (3,0.25,0) circle (2pt);
\filldraw [black!60] (-1,0.25,0) circle (2pt);
\node (1) at (-1,0.25,0) {};
\node (2) at (1,0.25,0) {};
\node (3) at (3,0.25,0) {};
\draw[black!80] (1) -- (2);
\draw[black!80] (2) -- (3);
%\draw[->,black!60,bend right] (2) to (1);
%\draw[->,black!60, bend right] (3) to (2);
\end{tikzpicture}};
\draw [->,decorate, 
decoration={snake,amplitude=.6mm,segment length=3mm,post length=1mm}] 
(2.75,0.2) -- (3.75,0.2);
\end{tikzpicture}
\end{array}
\]
and so the dual graph is simply $\begin{array}{c}
\begin{tikzpicture}[xscale=0.85,yscale=0.85]
 \node (0) at (0,0) [vertex] {};
 \node (1) at (1,0) [vertex] {};
 \node (2) at (2,0) [vertex] {};
 \draw [-] (0) -- (1);
\draw [-] (1) -- (2);
\end{tikzpicture}
\end{array}
$
\end{example}

The finite subgroups of $\SL(2,\mathbb{C})$ are classified.

\begin{thm}
Let $G$ be a finite subgroup of $\SL(2,\mathbb{C})$.  Then $G$ is isomorphic to one of the following groups
\[
\begin{array}{cc}
\textnormal{Type}&\textnormal{Definition}\\[1mm]
\mathbb{A}_n& \frac{1}{n}(1,n-1):=\left\langle \psi_n\right\rangle\\
\mathbb{D}_n& BD_{4n}:=\left\langle \psi_{2n},\tau\right\rangle\\
\mathbb{E}_6& \left\langle \psi_4,\tau, \eta\right\rangle\\
\mathbb{E}_7& \left\langle \psi_{8},\tau,\eta \right\rangle\\
\mathbb{E}_8& \langle \kappa,\omega,\iota\rangle\\
\end{array}
\]
where
\[
\begin{array}{c}
\begin{array}{cccc}
\psi_k= \begin{pmatrix}\e_k & 0\\ 0& \e_k^{-1}
\end{pmatrix} 
&\tau = \begin{pmatrix}0 & \e_4\\ \e_4& 0
\end{pmatrix}
&\eta=\frac{1}{\sqrt{2}} \begin{pmatrix}\e_8 & e_8^3\\ \e_8& \e_8^7 \end{pmatrix}
&\kappa= \begin{pmatrix}0 & -1\\ 1& 0 \end{pmatrix}
\end{array}\\
\begin{array}{cc}
\omega= \begin{pmatrix}\e_5^3 & 0\\ 0& \e_5^2
\end{pmatrix}&
\iota=\frac{1}{\sqrt{5}} \begin{pmatrix}\e_5^4-\e_5 & \e_5^2-\e_5^3\\ \e_5^2-\e_5^3& \e_5-\e_5^4
\end{pmatrix}
\end{array}
\end{array}
\]
and $\e_t$ denotes a primitive ${t}^{\rm th}$ root of unity.  
\end{thm}

Now $G\leq\SL(2,\mathbb{C})$ so $G$ acts on $V:=\mathbb{C}^2$.  In general, if $G\leq\GL(n,\mathbb{C})$ then the resulting geometry of $\mathbb{C}^n/G$ depends on two parameters, namely the group $G$ and the natural representation $V=\mathbb{C}^n$.  For example, if we fix the group $\mathbb{Z}_2$, then we could consider $\mathbb{Z}_2$ as a subgroup in many different ways, for example 
\[
\left\langle\begin{pmatrix} -1&0\\0&1\end{pmatrix}\right\rangle
\quad\textnormal{and}\quad
\left\langle\begin{pmatrix} -1&0\\0&-1\end{pmatrix}\right\rangle.
\]  
The first gives invariants $\mathbb{C}[x^2,y]$, which is smooth, whilst the latter gives invariants $\mathbb{C}[x^2,xy,y^2]$, which is singular.

Consequently the representation theory by itself will tell us nothing about the geometry (since it is only depends on one variable, namely the group $G$), so we must enrich the representation theory with the action of $G$ on $V$.   In  the following definition, this is why we tensor with $V$.

\begin{defin}
For given finite $G$ acting on $\mathbb{C}^2=V$, the McKay quiver is defined to be the quiver with vertices corresponding to the isomorphism classes of indecomposable representations of $G$, and the number of arrows from $\rho_1$ to $\rho_2$ is defined to be
\[
\dim_{\mathbb{C}}\Hom_{\mathbb{C}G}(\rho_1,\rho_2\otimes V),
\] 
i.e.\ the number of times $\rho_1$ appears in the decomposition of $\rho_2\otimes V$ into irreducibles.
\end{defin}

\begin{example}
We return to our running example \ref{Z3 running}  of the $\mathbb{Z}_3$ singularity.  Over $\mathbb{C}$, the group $\mathbb{Z}_3$ has only three irreducible representations $\rho_0, \rho_1, \rho_2$.   The natural representation is, by definition, $\left\langle\left( \begin{smallmatrix} \varepsilon_3 &0\\ 0&\varepsilon_3^2  \end{smallmatrix}\right)\right\rangle$, which splits as $\rho_1\oplus\rho_2$. Thus
\[
\rho_0\otimes V\cong \rho_1\oplus \rho_2,
\]
and so we conclude that there is precisely one arrow from $\rho_1$ to $\rho_0$, one arrow from $\rho_2$ to $\rho_0$, and no arrows from $\rho_0$ to $\rho_0$.  Thus, so far we have 
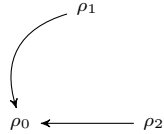
\[
\begin{tikzpicture} [bend angle=45, looseness=1.2]
\node[name=s,regular polygon, regular polygon sides=3, minimum size=2cm] at (0,0) {}; 
\node (C1) at (s.corner 1)  {$\scriptstyle \rho_1$};
\node (C2) at (s.corner 2)  {$\scriptstyle \rho_0$};
\node (C3) at (s.corner 3)  {$\scriptstyle \rho_2$};
\draw [->,bend right] (C1) to  (C2);
\draw [->,bend right] (C3) -- (C2);
\end{tikzpicture}
\]
Continuing the calculation, decomposing $\rho_1\otimes V$ and $\rho_2\otimes V$ gives the McKay quiver
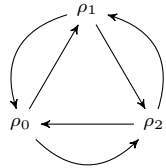
\[
\begin{tikzpicture} [bend angle=45, looseness=1.2]
\node[name=s,regular polygon, regular polygon sides=3, minimum size=2cm] at (0,0) {}; 
\node (C1) at (s.corner 1)  {$\scriptstyle \rho_1$};
\node (C2) at (s.corner 2)  {$\scriptstyle \rho_0$};
\node (C3) at (s.corner 3)  {$\scriptstyle \rho_2$};
\draw[->] (C2) --  (C1); 
\draw[->] (C1) --  (C3);
\draw[->] (C3) --  (C2); 
\draw [->,bend right] (C1) to  (C2);
\draw [->,bend right] (C2) to (C3);
\draw [->,bend right] (C3) to (C1);
\end{tikzpicture}
\]
which coincides with the quiver from \ref{quiver of Z3}.
\end{example}

\begin{example}\label{BD8 McKay}
As a second example, consider the group $BD_8$.  Being non-abelian of order 8, necessarily it must have four one-dimensional representations and one irreducible two-dimensional representation.  The natural representation $V$ is the irreducible two-dimensional representation.  A calculation, using only character theory (Exercise~\ref{Ex5.23}), shows that the McKay quiver is

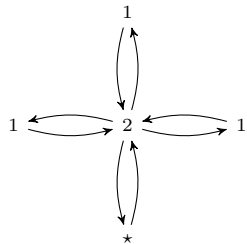
\[
\begin{tikzpicture}[xscale=1.5,yscale=1.5,bend angle=15, looseness=1]
 \node (1) at (1,0) {$\scriptstyle 1$};
 \node (1b) at (2,1) {$\scriptstyle 1$};
 \node (2) at (2,0) {$\scriptstyle 2$};
\node (3) at (3,0) {$\scriptstyle 1$};
\node (R) at (2,-1) {$\scriptstyle \star$};
\draw [bend right,->] (1) to (2);
\draw [bend right,->] (2) to (1);
\draw [bend right,->] (2) to (3);
\draw [bend right,->] (3) to (2);
\draw [bend right,->] (2) to (1b);
\draw [bend right,->] (1b) to (2);
\draw [bend right,->] (2) to (R);
\draw [bend right,->] (R) to (2);
\end{tikzpicture}
\]
where $\star$ is the trivial representation.  The numbers on the vertices correspond to the dimension of the corresponding irreducible representations.
\end{example}

Given the character table, the McKay quiver is a combinatorial object which is easy to construct.

\begin{thm} ($\SL(2,\mathbb{C})$ McKay Correspondence, combinatorial version).\\
Let $G\leq \SL(2,\mathbb{C})$ be a finite subgroup and let $Y\to \mathbb{C}^2/G=\Spec R$ denote the minimal resolution.  Then\\
(1) There exist 1-1 correspondences 
\[
\begin{array}{rcl}
\{ \t{exceptional curves}\}&\leftrightarrow& \{ \t{non-trivial irreducible representations} \}\\
&\leftrightarrow&\{ \t{non-free indecomposable CM $R$-modules} \}\\
\end{array}
\]
where the right hand sides are taken up to isomorphism.\\
(2) The dual graph is an ADE Dynkin diagram, of type corresponding to $\mathbb{A}$, $\mathbb{D}$ and $\mathbb{E}$ in the classification of the possible groups.\\
(3) (McKay)  There is a correspondence
\[
\begin{tikzpicture}[bend angle=20, looseness=1]
\node (A) at (-1.25,0) {$\{\t{dual graph} \}$};
\node (B) at (2.5,0) {$\{ \t{McKay quiver} \}$};
\draw[->,bend left] (0,0.1) to (1,0.1);
\draw[->,bend left] (1,-0.1) to (0,-0.1);
\end{tikzpicture}
\]
To go from the McKay quiver to the dual graph, simply kill the trivial representation, then merge every pair of arrows in opposite directions (to get an undirected graph).  To go from the dual graph to the McKay quiver, simply add a vertex to make the corresponding extended Dynkin diagram, then double the resulting graph to get a quiver.  
\end{thm}

The last part of the theorem is illustrated by the picture
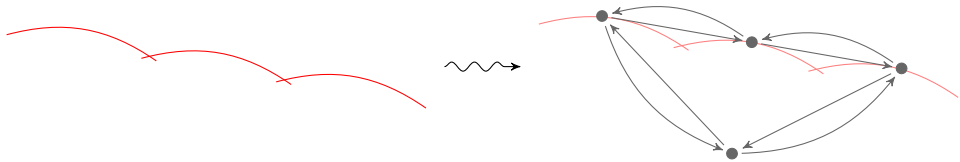
\[
\begin{array}{c}
\begin{tikzpicture}
\node at (0,0)
{\begin{tikzpicture} [bend angle=25, looseness=1,transform shape, rotate=-10]
\draw[red] (0,0,0) to [bend left=25] (2,0,0);
\draw[red] (1.8,0,0) to [bend left=25] (3.8,0,0);
\draw[red] (-1.8,0,0) to [bend left=25] (0.2,0,0);
\node (0) at (1,-1.25,0) {};
\end{tikzpicture} };

\node at (7,0.2) {
\begin{tikzpicture} [bend angle=25, looseness=1,transform shape, rotate=-10]
\draw[red!50] (0,0,0) to [bend left=25] (2,0,0);
\draw[red!50] (1.8,0,0) to [bend left=25] (3.8,0,0);
\draw[red!50] (-1.8,0,0) to [bend left=25] (0.2,0,0);
%\filldraw [black] (0.1,0.045,0) circle (0.5pt);
%\filldraw [black] (1.9,0.045,0) circle (0.5pt);
%\filldraw [black] (-1.8,0,0) circle (0.5pt);
%\filldraw [black] (3.8,0,0) circle (0.5pt);
\filldraw [black!60] (1,0.25,0) circle (2pt);
\filldraw [black!60] (3,0.25,0) circle (2pt);
\filldraw [black!60] (-1,0.25,0) circle (2pt);
\filldraw [black!60] (1,-1.25,0) circle (2pt);
\node (1) at (-1,0.25,0) {};
\node (2) at (1,0.25,0) {};
\node (3) at (3,0.25,0) {};
\node (0) at (1,-1.25,0) {};
\draw[->,black!60] (1) -- (2);
\draw[->,black!60] (2) -- (3);
\draw[->,black!60,bend right] (2) to (1);
\draw[->,black!60, bend right] (3) to (2);
\draw[->,black!60] (0) -- (1);
\draw[->,black!60] (3) -- (0);
\draw[->,bend right,black!60] (1) to (0);
\draw[->, bend right,black!60] (0) to (3);
\end{tikzpicture}};
\draw [->,decorate, 
decoration={snake,amplitude=.6mm,segment length=3mm,post length=1mm}] 
(3,0.4) -- (4,0.4);
\end{tikzpicture}
\end{array}
\]
\begin{example}
The correspondence for our running $\mathbb{Z}_3$ example is precisely 
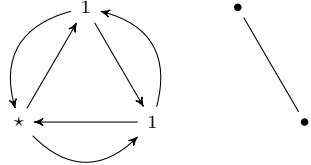
\[
\begin{tikzpicture} [bend angle=45, looseness=1.2]
\node[name=s,regular polygon, regular polygon sides=3, minimum size=2cm] at (0,0) {}; 
\node (C1) at (s.corner 1)  {$\scriptstyle 1$};
\node (C2) at (s.corner 2)  {$\scriptstyle \star$};
\node (C3) at (s.corner 3)  {$\scriptstyle 1$};
\draw[->] (C2) -- (C1); 
\draw[->] (C1) -- (C3);
\draw[->] (C3) -- (C2); 
\draw [->,bend right] (C1) to (C2);
\draw [->,bend right] (C2) to (C3);
\draw [->,bend right] (C3) to (C1);
\node (1b) at ($(s.corner 3)+(2,0)$) [vertex] {};
\node (2) at ($(s.corner 1)+(2,0)$) [vertex] {};
\draw [-] (2) -- (1b);
\end{tikzpicture}
\]
\end{example}

\begin{example}
The correspondence for the previous example \ref{BD8 McKay} is precisely 
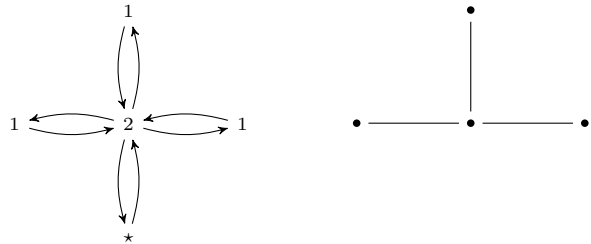
\[
\begin{tikzpicture}
\node at (0,2.5) {\begin{tikzpicture}[xscale=1.5,yscale=1.5,bend angle=15, looseness=1]
 \node (1) at (1,0) {$\scriptstyle 1$};
 \node (1b) at (2,1) {$\scriptstyle 1$};
 \node (2) at (2,0) {$\scriptstyle 2$};
\node (3) at (3,0) {$\scriptstyle 1$};
\node (R) at (2,-1) {$\scriptstyle \star$};
\draw [bend right,->] (1) to (2);
\draw [bend right,->] (2) to (1);
\draw [bend right,->] (2) to (3);
\draw [bend right,->] (3) to (2);
\draw [bend right,->] (2) to (1b);
\draw [bend right,->] (1b) to (2);
\draw [bend right,->] (2) to (R);
\draw [bend right,->] (R) to (2);
\end{tikzpicture}};
\node at (4.5,3.25) 
{\begin{tikzpicture}[xscale=1.5,yscale=1.5]
 \node (0) at (0,0) [vertex] {};
 \node (1) at (1,0) [vertex] {};
 \node (1b) at (1,1) [vertex] {};
 \node (2) at (2,0) [vertex] {};
% \node (0a) at (-0.1,-0.3) {$\scriptstyle -2$};
% \node (1a) at (0.9,-0.3) {$\scriptstyle -2$};
% \node (1ba) at (0.65,1) {$\scriptstyle -2$};
% \node (2a) at (1.9,-0.3) {$\scriptstyle -2$};
 \draw [-] (0) -- (1);
\draw [-] (1) -- (2);
\draw [-] (1) -- (1b);
\end{tikzpicture}};
\end{tikzpicture}
\]
\end{example}
This all takes place at a combinatorial level, to gain more we must add structure.

\subsection{Auslander's McKay Correspondence (Surfaces)}
In its simplest form, Auslander's version of the McKay Correspondence states that the geometry of the minimal resolution can be reconstructed using homological algebra on the category $\CM\mathbb{C}[[x,y]]^G$.  Reinterpreted, the CM modules on the singularity $\Spec\mathbb{C}[[x,y]]^G$ encode information about the resolution.

For what follows we must use the completion $\mathbb{C}[[x,y]]^G$ (not $\mathbb{C}[x,y]^G$) so that the relevant categories are Krull--Schmidt.  This is mainly just for technical simplification; with more work it is possible to phrase the results in the non-complete setting.  Throughout this subsection, to simplify notation we set $R:=\mathbb{C}[[x,y]]^G$.

\begin{defin}
We say that a short exact sequence 
\[
0\to A\to B\stackrel{f}{\to} C\to 0
\]
in the category $\CM R$ is an Auslander--Reiten (=AR) sequence if every $D\to C$ in $\CM R$ which is not a split epimorphism factors through $f$.
\end{defin}

Since $R$ is an isolated singularity (normal surface singularities are always isolated), it is a consequence of AR duality (\S\ref{AR proof section}) that AR sequences exist in the category $\CM R$, see for example \cite[\S3]{Y}.  Now given the existence of AR sequences, we attach to the category $\CM R$ the AR quiver as follows.

\begin{defin}
The AR quiver of the category $\CM R$ has as vertices the isomorphism classes of the indecomposable CM $R$-modules.  As arrows, for each indecomposable $M$ we consider the AR sequence ending at M 
\[
0\to \tau(M)\to E\stackrel{f}{\to} M\to 0,
\]
decompose $E$ into indecomposable modules as $E\cong \bigoplus_{i=1}^nM_i^{\oplus a_i}$ and for each $i$ draw $a_i$ arrows from $M_i$ to $M$.
\end{defin}

As a consequence of \S\ref{L1} (precisely \ref{thm Auslander} and \ref{Auslander skew}) we know that there are only finitely many CM $R$-modules up to isomorphism.  Further, by projectivization \ref{facts1and2}(\ref{F1}) there is an equivalence of categories
\[
\CM R\simeq \proj\mathbb{C}[[x,y]]\# G.
\]  
The quiver of the skew group ring is well-known.
\begin{lemma}\label{skew Morita}
Let $G\leq\GL(V)$ be a finite subgroup.  Then the skew group ring $\mathbb{C}[V]\# G$ is always Morita equivalent to the McKay quiver, modulo some relations. 
\end{lemma}
The relations in \ref{skew Morita} can be described, but are not needed below.  Some further work involving the Koszul complex gives the following.
\begin{thm} (Auslander McKay Correspondence) 
Let $G$ be a finite subgroup of $\SL(2,\mathbb{C})$.  Then the AR quiver of $\CM\mathbb{C}[[x,y]]^G$ equals the McKay quiver of $G$, and so there is a correspondence
\[
\begin{tikzpicture}[bend angle=20, looseness=1]
\node (A) at (-1.25,0) {$\{\t{dual graph} \}$};
\node (B) at (3.55,0) {$\{ \t{AR quiver of }\CM\mathbb{C}[[x,y]]^G$ \}};
\draw[->,bend left] (0,0.1) to (1,0.1);
\draw[->,bend left] (1,-0.1) to (0,-0.1);
\end{tikzpicture}
\]
\end{thm}

\subsection{Derived McKay Correspondence (Surfaces)}
The Auslander McKay Correspondence upgraded the combinatorial version by considering CM modules over the commutative ring $R$ (to give more structure), and applied methods in homological algebra to obtain the dual graph.  The derived version of the McKay Correspondence improves this further, since it deals with not just the dual graph, which is a simplified version of the minimal resolution, but with the minimal resolution itself.  The upshot is that by considering the NCCR given by the skew group ring (the Auslander algebra of \S\ref{L1}) seriously as an object in its own right, we are in fact able to obtain the minimal resolution as a quiver GIT, and thus obtain the geometry from the noncommutative resolution without assuming that the geometry exists. 

\begin{thm}\label{Db McKay}(Derived $\SL(2,\mathbb{C})$ McKay Correspondence). Let $G\leq \SL(2,\mathbb{C})$ be a finite subgroup and let $Y\to \mathbb{C}^2/G=\Spec R$ denote the minimal resolution.  Then
\begin{itemize}
\item[(1)] $\Db(\mod \mathbb{C}[x,y] \# G)\simeq \Db(\coh Y)$.
\item[(2)] Considering $\mathbb{C}[x,y] \# G$ as the McKay quiver subject to the preprojective relations, choose the dimension vector $\alpha$ corresponding to the dimensions of the irreducible representations.  Then for \emph{any} generic stability condition $\theta$, we have
\begin{itemize}
\item[(a)] $\Db(\coh \c{M}_{\theta}^{s})\simeq\Db(\mod \mathbb{C}[x,y] \# G)$.
\item[(b)] $\c{M}_{\theta}^{s}$ is the minimal resolution of $\Spec R$.
\end{itemize}
\end{itemize}
\end{thm}

There are now many different proofs of both (1) and (2), and some are mentioned in the credits below.  A sketch proof is given in the next section, which also covers the dimension three situation.   

\begin{remark}
With regards to the choice of stability, once \ref{Db McKay} has been established for one generic $\theta$, it actually follows that the theorem holds for all generic stability conditions simply by tracking $\theta$ through derived equivalences induced by spherical twists and their inverses. This observation also has applications in higher dimensions.
\end{remark}

\begin{example}
Returning to the running $\mathbb{Z}_3$ example, we computed in \ref{Z3 GIT} the moduli space for a very specific  generic parameter, remarked it was the minimal resolution and showed in \ref{Z3 Db} that the NCCR and the minimal resolution were derived equivalent.  We know that NCCRs are unique up to Morita equivalence in dimension two (\ref{NCCR dim 2}), so this establishes the above theorem in the special case $G=\mathbb{Z}_3$ and $\theta=(-2,1,1)$.  The theorem extends our result by saying we could have taken \emph{any} generic parameter, and in fact any group $G$.
\end{example}

\subsection{From NCCRs to Crepant Resolutions}\label{last section} The main theorem in the notes, which encapsulates both Gorenstein quotient singularities in dimension two and three, and also many other situations, is the following.  Everything in this section is finite type over $\mathbb{C}$, and by a point we mean a \emph{closed} point.

\begin{thm}\label{VdB main} (NCCRs give crepant resolutions in dimension $\leq 3$)\\
Let $R$ be a Gorenstein normal domain with $\dim R=n$, and suppose that there exists an NCCR $\Lambda:=\End_R(M)$.  Choose the dimension vector corresponding to the ranks of the indecomposable summands of $M$ (in some decomposition of $M$ into indecomposables), and for any generic stability $\theta$ consider the moduli space $\c{M}^s_\theta$.  Let $X_1$ be the locus of $\Spec R$ where $M$ is locally free, and define $\c{U}_\theta$ to be the unique irreducible component of $\c{M}^s_\theta$ that maps onto $X_1$.  If $\dim(\c{U}_\theta\times_R\c{U}_\theta)\leq n+1$, then
\begin{itemize}
\item[(1)] There is an equivalence of triangulated categories
\[
\Db(\coh\c{U}_\theta)\simeq \Db(\mod\Lambda).
\]
\item[(2)] $\c{U}_\theta$ is a crepant resolution of $\Spec R$.
\end{itemize}
\end{thm}

Part (2) is an immediate consequence of part (1) by \ref{NCCRdimd}, and note that the assumption on the dimension of the product is automatically satisfied if $n\leq 3$.  We sketch the proof of \ref{VdB main} below.   One of the key innovations in \cite{BKR}, adapted from \cite{BM}, is that the new intersection theorem from commutative algebra can be used to establish regularity.

Recall that if $E\in\Db(\coh Y)$, then the support of $E$, denoted $\Supp E$, is the closed subset of $Y$ obtained as the union of the support of its cohomology sheaves.  On the other hand, the homological dimension of $E$, denoted $\homdim E$, is the smallest integer $i$ such that $E$ is quasi-isomorphic to a complex of locally free sheaves of length $i$.  Note that by convention, a locally free sheaf considered as a complex in degree zero has length $0$.   We need the following lemma due to Bridgeland--Maciocia \cite[5.3, 5.4]{BM}.
\begin{lemma}\label{supp lemmas}
Suppose that $Z$ is a scheme of finite type over $\mathbb{C}$, and let $E\in\Db(\coh Z)$.\\
(1) Fix a point $z\in Z$.  Then 
\[
z\in\Supp E\iff \Hom_{\D(Z)}(E,\c{O}_z[i])\neq 0 \mbox{ for some }i.
\]
(2) Suppose further that $Z$ is quasi-projective.  If there exists $j\in\mathbb{Z}$ and $s\geq 0$ such that for all points $z\in Z$
\[
\Hom_{\D(Z)}(E,\c{O}_z[i])=0\mbox{ unless }j\leq i\leq j+s,
\]
then $\homdim E\leq s$.
\end{lemma}

The next result is much deeper, and follows from the New Intersection Theorem in commutative algebra \cite[1.2]{BI}.
 \begin{thm}\label{intersection theorem}
Let $Z$ be an irreducible scheme of dimension $n$ over $\mathbb{C}$, and let $E$ be a non-zero object of $\Db(\coh Z)$.  Then\\
(1) $\dim(\Supp E)\geq \dim Z-\homdim E$.\\
(2)  Suppose there is a point $z'\in Z$ such that the skyscraper $\c{O}_{z'}$ is a direct summand of $H^0(E)$, and further
\[
\Hom_{\D(Z)}(E,\c{O}_z[i])=0\quad\mbox{unless }z=z'\mbox{ and }0\leq i\leq n. 
\]
Then $Z$ is nonsingular at $z'$ and $E\cong H^0(E)$ in $\D(Z)$.
\end{thm}

Recall that a set $\Omega\subseteq\Db(\coh Y)$ is called a {\em spanning class} if (a) $\RHom_Y(a,c)=0$ for all $c\in\Omega$ implies $a=0$, and (b) $\RHom_Y(c,a)=0$ for all $c\in\Omega$ implies $a=0$.
\begin{lemma}\label{spanning lemma}
Suppose that $Z$ is a smooth variety of finite type over $\mathbb{C}$, projective over $\Spec T$ where $T$ is a Gorenstein ring.  Then $\Omega:=\{\c{O}_z\mid z\in Z\}$ is a spanning class. 
\end{lemma}
\begin{proof}
By \ref{supp lemmas}(1) $z\in\Supp E\iff \RHom_Z(E,\c{O}_z)\neq 0$, hence $\RHom_Z(E,\c{O}_z)=0$ for all $\c{O}_z\in\Omega$ implies that $\Supp E$ is empty, thus $E=0$.  This shows (a).  To check the condition (b), suppose that $\RHom_Z(\c{O}_z,E)=0$ for all $\c{O}_z\in\Omega$.  Then
\[
0=\RHom_T(\RHom_Z(\c{O}_z,E),T)
\stackrel{\scriptsize\mbox{\ref{geometry Serre}}}{\cong}\RHom_Z(E,\mathbb{S}_Z\c{O}_z).
\]
But since $Z$ and $T$ are Gorenstein, again by \ref{geometry Serre} $\mathbb{S}_Z$ is simply tensoring by a line bundle and (possibly) shifting.  Tensoring a skyscraper by a line bundle gives back the same skyscraper, so the above implies that $\RHom_Z(E,\c{O}_z)=0$ for all $\c{O}_z\in\Omega$.  As above, this gives $E=0$.
\end{proof}

As one final piece of notation, if $Y$ is a $\mathbb{C}$-scheme and $\Gamma$ is a $\mathbb{C}$-algebra, we define 
\[
\c{O}_Y^\Gamma:=\c{O}_Y\otimes_{\underline{\mathbb{C}}}\underline{\Gamma},
\] 
where $\underline{\mathbb{C}}$ and $\underline{\Gamma}$ are the constant sheaves associated to $\mathbb{C}$ and $\Gamma$. 

We can now sketch the proof of \ref{VdB main}.
\begin{proof}
For notational simplicity in the proof, denote $\c{U}_\theta$ simply as $Y$.  The tautological bundle from the GIT construction restricted to $Y$ gives us a bundle $\c{M}$ on $Y$, which is a sheaf of $\Lambda$-modules.   There is a projective birational morphism $f\colon Y\to\Spec R$, see \ref{map to R proof}.

Consider the commutative diagram
\begin{eqnarray}
\begin{array}{c}
\begin{tikzpicture}
\node (a1) at (0,1.5) {$Y\times Y$};
\node (a2) at (2,1.5) {$Y$};
\node (b1) at (0,0) {$Y$};
\node (b2) at (2,0) {$\Spec\mathbb{C}$};
\draw[->] (a1) -- node[above] {$\scriptstyle p_1$} (a2);
\draw[->] (a1) -- node[left] {$\scriptstyle p_2$} (b1);
\draw[->] (a2) -- node[right] {$\scriptstyle \pi$} (b2);
\draw[->] (b1) -- node[above] {$\scriptstyle \pi$} (b2);
\end{tikzpicture}
\end{array}\label{BC 1}
\end{eqnarray}
There is a natural functor
\[
\Phi:=\mathbf{R}\pi_*(-\otimes^{{\bf L}}_{\c{O}_Y}\c{M})\colon\D(\Qcoh Y)\to\D(\Mod \Lambda)
\]
which has left adjoint $\Psi:=\pi^*(-)\otimes^{\bf{L}}_{\c{O}_Y^\Lambda}\c{M}^\vee$.  There are natural isomorphisms
\begin{align*}
\Psi\Phi(-)&=\pi^*(\mathbf{R}\pi_*(-\otimes^{{\bf L}}_{\c{O}_Y}\c{M}))\otimes^{\bf{L}}_{\c{O}_Y^\Lambda}\c{M}^\vee\\
&\cong\mathbf{R}{p_2}_*p_1^*(-\otimes^{{\bf L}}_{\c{O}_Y}\c{M})\otimes^{\bf{L}}_{\c{O}_Y^\Lambda}\c{M}^\vee\tag{flat base change}\\
&\cong\mathbf{R}{p_2}_*(p_1^*(-\otimes^{{\bf L}}_{\c{O}_Y}\c{M})\otimes^{\bf{L}}_{\c{O}_{Y\times Y}^\Lambda}p_2^*\c{M}^\vee)\tag{projection formula}\\
&\cong\mathbf{R}{p_2}_*(p_1^*(-)\otimes^{{\bf L}}_{\c{O}_{Y\times Y}}p_1^*\c{M}\otimes^{\bf{L}}_{\c{O}_{Y\times Y}^\Lambda}p_2^*\c{M}^\vee).
\end{align*}
It follows that  $\Psi\Phi$ is the Fourier--Mukai functor with respect to 
\[
\c{Q}:=\c{M}\boxtimes_{\Lambda}^{\bf{L}}\c{M}^\vee:=p_1^*\c{M}\otimes^{\bf{L}}_{\c{O}_{Y\times Y}^\Lambda}p_2^*\c{M}^\vee.
\]
It is known that any Fourier--Mukai functor applied to a skyscraper $\c{O}_y$ is the derived restriction of the corresponding kernel via the morphism $i_y\colon\{y\}\times Y\to Y\times Y$, so since $\Psi\Phi$ is the Fourier--Mukai functor with kernel $\c{Q}$,
\begin{eqnarray}
\Psi\Phi \c{O}_y\cong \mathbf{L}i_y^*\c{Q}.\label{Q formula}
\end{eqnarray}
We claim that $\c{Q}$ is supported on the diagonal.  For this, we let $\c{M}_y$ be the fibre of $\c{M}$ at $y$, and first remark that
\begin{eqnarray}
\Ext^i_{\Lambda}(\c{M}_y,\c{M}_{y'})=0\mbox{ for all }y\neq y' \mbox{ if }i\notin[1,n-1].\label{Ext for Q 1}
\end{eqnarray}
The case $i=0$ is clear, the case $i=n$ follows from CY duality \ref{CY general}(2) on $\Lambda$, and the remaining cases $i<0$ and $i>n$ follow since $\c{M}_y$ and $\c{M}_{y'}$ are modules, and $\gl\Lambda=n$.  

Now the chain of isomorphisms
\begin{align*}
\Hom_{D(Y\times Y)}(\c{Q},\c{O}_{y,y'}[i])
&\cong \Hom_{D(Y\times Y)}(\c{Q},{i_y}_*\c{O}_{y'}[i])\\
&\cong \Hom_{D(Y)}({\mathbf L}i_y^*\c{Q},\c{O}_{y'}[i])\tag{adjunction}\\
&\cong \Hom_{D(Y)}(\Psi\Phi\c{O}_y,\c{O}_{y'}[i])\tag{by \eqref{Q formula}}\\
&\cong \Hom_{D(\Lambda)}(\Phi\c{O}_y,\Phi\c{O}_{y'}[i])\tag{adjunction}\\
&\cong \Hom_{D(\Lambda)}(\c{M}_y,\c{M}_{y'}[i])\\
&\cong \Ext^i_{\Lambda}(\c{M}_y,\c{M}_{y'})
\end{align*}
shows two things. 
\begin{enumerate}
\item First, if $(y,y')\notin Y\times_R Y$ then $f(y)\neq f(y')$ and so $\c{M}_y$ and $\c{M}_{y'}$ are finite length $\Lambda$-modules which, when viewed as $R$-modules, are supported at different points of $R$.  Since 
\[
\Ext^i_{\Lambda}(\c{M}_y,\c{M}_{y'})_\m\cong\Ext^i_{{\Lambda}_\m}({\c{M}_y}_\m,{\c{M}_{y'}}_\m)
\]
for all maximal ideals $\m$ of $R$, we deduce that there are no Homs or Exts between $\c{M}_y$ and $\c{M}_{y'}$, so the above chain of isomorphisms gives $\Hom_{Y\times Y}(\c{Q},\c{O}_{y,y'}[i])=0$ for all $i$.  Hence, by \ref{supp lemmas}(1), we see that $\c{Q}$ is supported on $Y\times_R Y$, and so 
\[
\dim(\Supp\c{Q}|_{(Y\times Y)\backslash \Delta})\leq \dim(Y\times_R Y)\leq n+1,
\]
where the last inequality holds by assumption.  
\item Second, if $(y,y')\in (Y\times Y)\backslash\Delta$ (i.e.\ $y\neq y'$), then by \eqref{Ext for Q 1}
\[
\Hom_{D(Y\times Y)}(\c{Q},\c{O}_{y,y'}[i])=0\mbox{ unless }1\leq i\leq n-1. 
\]
Applying \ref{supp lemmas}(2) with $j=1$ gives $\homdim \c{Q}|_{(Y\times Y)\backslash \Delta}\leq n-2$, thus by \ref{intersection theorem}(1)
\begin{align*}
\dim(\Supp\c{Q}|_{(Y\times Y)\backslash \Delta})&\geq \dim((Y\times Y)\backslash \Delta)-\homdim \c{Q}|_{(Y\times Y)\backslash \Delta}\\
&\geq 2n-(n-2)=n+2.
\end{align*}
\end{enumerate}

Combining (1) and (2) gives a contradiction unless $\c{Q}|_{(Y\times Y)\backslash \Delta}=0$, so indeed $\c{Q}$ is supported on the diagonal.  In particular, by \ref{supp lemmas}(1) it follows that $\Hom_{\D(Y\times Y)}(\c{Q},,\c{O}_{y,y'}[i])=0$ for all $i\in\mathbb{Z}$ (provided that $y\neq y'$), so the above chain of isomorphisms shows that
\begin{eqnarray}
\Ext^{i}_\Lambda(\c{M}_{y},\c{M}_{y'})=0\mbox{ for all }i\in\mathbb{Z},\mbox{ provided that }y\neq y'.\label{M Ext vanishing 2}
\end{eqnarray}

Pick a point $y\in Y$ and next consider $\Psi\Phi\c{O}_y$.  The counit of the adjunction gives us a natural map $\Psi\Phi\c{O}_y\to \c{O}_y$ and thus a triangle
\begin{eqnarray}
c_y\to\Psi\Phi\c{O}_y\to \c{O}_y\to.\label{defin c triangle}
\end{eqnarray}
Applying $\Hom_Y(-,\c{O}_y)$ and using the adjunction gives an exact sequence
\begin{align*}
0\to&\Hom_Y(c_y,\c{O}_y[-1])\to\Hom_Y(\c{O}_y,\c{O}_y)\xrightarrow{\alpha} \Hom_\Lambda(\Phi\c{O}_y,\Phi\c{O}_y)\to \Hom_Y(c_y,\c{O}_y)\to\\
& \Hom_Y(\c{O}_y,\c{O}_y[1])\xrightarrow{\beta} \Hom_\Lambda(\Phi\c{O}_y,\Phi\c{O}_y[1])\to \Hom_Y(c_y,\c{O}_y[1])\to \hdots
\end{align*}
The domain and the codomain of $\alpha$ are isomorphic to $\mathbb{C}$, and $\Phi$ (being a functor) takes the identity to the identity, thus $\alpha$ is an isomorphism.  By a result of Bridgeland \cite[4.4]{BridgelandLMS}, $\beta$ can be identified with the Kodaira--Spencer map, and is injective.  We conclude that $\Hom_Y(c_y,\c{O}_y[i])=0$ for all $i\leq 0$, from which the spectral sequence
\[
E_2^{p,q}=\Ext^p(H^{-q}(c_y),\c{O}_y)\Rightarrow \Hom(c_y,\c{O}_y[p+q])
\]
(see e.g. \cite[2.70(ii)]{HuybrechtsFM}) yields $H^i(c_y)=0$ for all $i\geq 0$.   Taking cohomology of \eqref{defin c triangle} then gives $H^0(\Psi\Phi\c{O}_y)\cong \c{O}_y$.  Set $E:=\Psi\Phi\c{O}_y$, then since further
\[
\Hom_{\D(Y)}(E,\c{O}_{y'}[i])\cong\Ext^i_{\Lambda}(\c{M}_y,\c{M}_{y'})=0
\]
for all $y\neq y'$ by \eqref{M Ext vanishing 2}, and also for all $i<0$ and all $i>n$ in the case $y=y'$ (since $\Lambda$ has global dimension $n$), it follows from the New Intersection Theorem  \ref{intersection theorem}(2) that $Y$ is nonsingular at $y$, and $E\cong H^0(E)$, thus also $c_y=0$.

Since the above holds for all $y\in Y$ we conclude that $Y$ is smooth.  Thus $\Perf(Y)=\Db(\coh Y)$, and further since $\Lambda$ has finite global dimension it is also true that $\Kb(\proj\Lambda)=\Db(\mod\Lambda)$. Consequently it is easy to see that 
\[
\Phi=\mathbf{R}\pi_*(-\otimes^{{\bf L}}_{Y}\c{M})\colon\Db(\coh Y)\to\Db(\mod \Lambda)
\]
and its left adjoint $\Psi=-\otimes^{{\bf L}}_{\Lambda}\c{M}^\vee$ also preserves boundedness and coherence.  Further, by \ref{geometry Serre}(2), $\mathbb{S}_Y=-\otimes f^!\c{O}_R$ is a Serre functor on $Y$ relative to the canonical $R$, and $\mathbb{S}_\Lambda=\mathrm{Id}$ is a Serre functor on $\Lambda$ relative to the canonical $R$ by \ref{CY general}(1).  Using
\begin{align*}
\RHom_\Lambda(\Phi(a),b)&\cong\RHom_R(\RHom_\Lambda(\mathbb{S}_\Lambda^{-1}b,\Phi(a)),R)\\
&\cong \RHom_R(\RHom_Y(\Psi\mathbb{S}_\Lambda^{-1}b,a),R)\tag{adjunction}\\
&\cong \RHom_Y(a,\mathbb{S}_Y\Psi\mathbb{S}_\Lambda^{-1}b),
\end{align*}
taking degree zero cohomology we see that $\mathbb{S}_Y\circ\Psi \circ\mathbb{S}_\Lambda^{-1}$ is right adjoint to $\Phi$.

Now,  since $Y$ is smooth, $\{ \c{O}_y \mid y\in Y\}$ is a spanning class by \ref{spanning lemma}.  Hence since $\Phi$ has both right and left adjoints, and further $\Phi\c{O}_y=\c{M}_y$, to show that $\Phi$ is fully faithful we just have to check that the natural maps 
\begin{eqnarray}
\Hom_{\Db(\coh Y)}(\c{O}_{y},\c{O}_{y'}[i])\to\Hom_{\Db(\mod\Lambda)}(\c{M}_{y},\c{M}_{y'}[i])\label{LRspanningcheck}
\end{eqnarray}
are bijections for all $y,y'\in Y$ and $i\in\mathbb{Z}$ \cite[1.49]{HuybrechtsFM}.  

But by \eqref{defin c triangle} (and the fact we now know that $c_y=0$), the counit $\Psi\Phi\to 1$ is an isomorphism on all the skyscrapers $\c{O}_y$, and hence also on their shifts.  Thus all the skyscrapers and all their shifts belong to the full subcategory on which $\Phi$ is fully faithful, and so in particular \eqref{LRspanningcheck} must be bijective for all $y,y'\in Y$ and all $i\in\mathbb{Z}$.   We deduce that $\Phi$ is fully faithful.

Finally, since $\Db(\coh Y)$ is non-trivial and $\Db(\mod\Lambda)$ is indecomposable, to show that $\Phi$ is an equivalence we just need that $\mathbb{S}_\Lambda\Phi\c{O}_y\cong\Phi \mathbb{S}_Y\c{O}_y$ for all $y\in Y$ (the proof of this fact follows in a similar way to \cite[1.56]{HuybrechtsFM}, with Serre functors replaced by the RHom versions).  But, as above $\mathbb{S}_\Lambda=\mathrm{Id}$ and $\mathbb{S}_Y$ is tensoring by a line bundle, so $\Phi\mathbb{S}_Y\c{O}_y=\Phi(\c{L}\otimes\c{O}_y)\cong\Phi\c{O}_y\cong\mathrm{Id}\Phi\c{O}_y\cong \mathbb{S}_\Lambda\Phi\c{O}_y$.  It follows that $\Phi$ is an equivalence.
\end{proof}

\begin{remark}
The irreducible component appears in the statement of the theorem since (as in Exercise~\ref{Ex3.12}) often components arise in quiver GIT.  However, with the hypotheses as in \ref{VdB main}, there is still no known example of when $\c{U}_\theta\neq\c{M}^{s}_\theta$.  In the case when $R$ is complete local and $\dim R=3$, $\c{U}_\theta=\c{M}^{s}_\theta$ by \cite[6.6.1]{VdBNCCR}.
\end{remark}

\begin{remark}\label{map to R proof}
The above proof of \ref{VdB main} skips over the existence of a projective birational map $Y\to\Spec R$.  Recall (from \ref{remark warning for later}) that the quiver GIT only gives a projective map $Y\to \Spec \mathbb{C}[\mathcal{R}]^G$, but $ \mathbb{C}[\mathcal{R}]^G$ need not equal $R$ by Exercise~\ref{Ex3.15}.  For details on how to overcome this problem, see \cite[\S6.2]{VdBNCCR} or \cite[2.2, 2.3]{CI}.
\end{remark}

\medskip
\noindent
{\bf Credits:} The geometry of the minimal resolutions of the quotient singularities arising from finite subgroups of $\SL(2,\mathbb{C})$ has a long history stretching back to at least du Val.  The relationship with the representation theory was discovered by McKay in \cite{McKay}, which in turned produced many geometric interpretations and generalizations, e.g.\ in \cite{AV}, \cite{Esnault} and \cite{Wunram}.  Auslander's version of the McKay Correspondence was proved in \cite{Aus3}, and is summarised in \cite[\S10]{Y}.

The fact that the skew group ring is always Morita equivalent to the McKay quiver was written down in \cite{BSW}, but it was well-known to Reiten--Van den Bergh \cite{RV89}, Crawley-Boevey and many others well before then.

It is possible to use \cite{AV,Esnault,Wunram} to establish the derived equivalence in \ref{Db McKay}, but there are many other proofs of this.  The first was Kapranov--Vasserot \cite{KV}, but it also follows from \cite{VdB1d}, \cite{BKR} or \cite{VdBNCCR}.  The main result, \ref{VdB main}, is Van den Bergh's \cite{VdBNCCR} interpretation of Bridgeland--King--Reid \cite{BKR}.   See also  \cite[\S7]{Crawnotes}.

\subsection{Exercises}

\begin{ex}\label{Ex5.22}
We can use weighted $\mathbb{C}^*$-actions on polynomial rings as an easy way to produce NCCRs.   Consider the polynomial ring $S=\mathbb{C}[x_1,\hdots, x_n,y_1,\hdots, y_m]$, with $n,m\geq 2$,  and non-negative integers $a_1,\hdots,a_n,b_1,\hdots, b_m\in\mathbb{N}$.  We define a $\mathbb{C}^*$-action on $S$ by
\[
\lambda\cdot x_i:=\lambda^{a_i}x_i\quad\textnormal{and}\quad \lambda\cdot y_i:=\lambda^{-b_i}y_i.
\]
As shorthand, we denote this action by $(a_1,\hdots,a_n,-b_1,\hdots,-b_m)$. We consider
\[
S_i:=\{ f\in S\mid \lambda\cdot f=\lambda^if \}
\]
for all $i\in\mathbb{Z}$.  It should be fairly clear that $S=\oplus_{i\in\mathbb{Z}} S_i$.  We will consider the invariant ring $S_0$, which is known to be CM, of dimension $n+m-1$.  It is Gorenstein if the sum of the weights is zero, i.e.\ $a_1+\cdots +a_n-b_1-\cdots b_m=0$.  Note that the $S_i$ are modules over $S_0$, so this gives a cheap supply of $S_0$-modules.
\begin{enumerate}
\item (The easiest case). Consider the weights $(1,1,-1,-1)$.  Here 
\[
S_0=\mathbb{C}[x_1y_1,x_1y_2,x_2y_1,x_2y_2]\cong \mathbb{C}[a,b,c,d]/ac=bd.
\]
We have seen this example before.  Which of the $S_i\in\CM S_0$?  (The result is combinatorial, but quite hard to prove).  Show that $\End_{S_0}(S_1\oplus S_2)\cong\End_{S_0}(S_0\oplus S_1)$.  This gives an example of an NCCR given by a reflexive module that is not CM.
\item Find $S_0$ for the action $(2,1,-2,-1)$.  We have also seen this before.  How to build an NCCR for this example?
\item Experiment with other $\mathbb{C}^*$ actions.  Is there any structure?
\end{enumerate} 
\end{ex}

\begin{ex}\label{Ex5.23}
Compute the character table for the groups $\mathbb{A}_n$ and $\mathbb{D}_n$ and hence determine their McKay quivers.
\end{ex}

\begin{ex}\label{Ex5.24}
For $R=\mathbb{C}[[x,y]]^G$ with $G$ a finite subgroup of $\SL(2,\mathbb{C})$, prove that for each non-free CM $R$-module, the minimal number of generators is precisely  twice its rank.
\end{ex}

\noindent{\bf Computer Exercises:}
\begin{ex}\label{Ex5.25}
(How to show infinite global dimension without actually computing it). By Exercise~\ref{Ex4.33}, when $R$ is local Gorenstein, $M\in\CM  R$ with $R\in\add M$, and $\End_R(M)\in\CM R$, then 
\[
\gl\End_R(M)<\infty\iff \add M=\{ X\in\CM R\mid \Hom_R(M,X)\in\CM R\}.
\]
Thus to show $\End_R(M)$ has infinite global dimension, we just need to find a CM module $X$, with $X\notin\add M$, such that $\Hom_R(M,X)\in\CM R$.  Once we have guessed an $X$, we can check the rest on Singular. Consider the ring $R:=\mathbb{C}[[u,v,x,y]]/(uv-x(x^2+y^7))$.
\begin{enumerate}
\item  Let $M:=R\oplus (u,x)$. Show that $\End_R(M)\in\CM R$.
\item Consider the module $X$ given as 
\[
R^4 \xrightarrow{\left(\begin{smallmatrix}-x&-y&-u&0\\ -y^6& x& 0&-u\\ v&0&x^2&xy\\ 0&v&xy^6&-x^2  \end{smallmatrix}\right)}
R^4 \to X \to 0
\]
Show that $\Hom_R(M,X)\in\CM R$, and so $\End_R(M)$ has infinite global dimension.   (aside:\ how we produce such a counterexample $X$ is partially explained in Exercise~\ref{Ex5.26})
\end{enumerate}
\end{ex}

\begin{ex}\label{Ex5.26}
(Kn\"orrer periodicity) Consider a hypersurface $R:=\mathbb{C}[[x,y,]]/(f)$ where $f\in\m:=(x,y)$.  We could take $f=x(x^2+y^7)$, as in the last example. Consider a CM $R$-module $X$ with given projective presentation 
\[
R^a\stackrel{\varphi}{\to}R^a\stackrel{\psi}\to R^a\to  X\to 0
\]
(the free modules all having the same rank is actually forced).  The Kn\"orrer functor takes $X$ to a module $K(X)$ for the ring $R^\prime:=\mathbb{C}[[u,v,x,y]]/(uv-f)$, defined as the cokernel
\[
(R^\prime)^{2a}\xrightarrow{\left(\begin{smallmatrix}-\varphi&-u\mathbb{I}\\ v\mathbb{I}& \psi\end{smallmatrix}\right)} (R^\prime)^{2a}\to  K(X)\to 0
\]
\begin{enumerate}
\item Experimenting with different $f$, show that $K(X)\in\CM R^\prime$.  For example, if we write $f$ into irreducibles $f=f_1\hdots f_n$, then
\[
R^a\xrightarrow{f_1}R^a\xrightarrow{f_2\hdots f_n} R^a\to  X\to 0
\]
are examples of CM $R$-modules on which to test the hypothesis.
\item When $X,Y\in\CM R$, compute both $\Ext^1_R(X,Y)$ and $\Ext^1_{R^\prime}(K(X), K(Y))$.  Is there a pattern?
\end{enumerate}
\end{ex}

\newpage
\section{Appendix: Quiver Representations}\label{QuiverAppendix}

Quivers provide a method to visualize modules (=representations) and are very useful tool to explicitly write down examples of modules.   They have many uses throughout mathematics.

\begin{defin}
A (finite) quiver $Q$ is a directed graph with finitely many vertices and finitely many arrows. 
\end{defin}
We often label the vertices with numbers.
\begin{example}\label{7.2}
For example 
\[
\begin{array}{ccc}
Q_{1}=
\begin{array}{c}
\begin{tikzpicture}
\node (a) at (-1,0) [vertex] {};
\node (b) at (0,0) [vertex] {};
\node (c) at (1,0) [vertex] {};
\node (a1) at (-1,-0.2) {$\scriptstyle 1$};
\node (b1) at (0,-0.2) {$\scriptstyle 2$};
\node (c1) at (1,-0.2) {$\scriptstyle 3$};
\draw[->] (a) to (b);
\draw[->] (c) to (b);
 \draw[->]  (c) edge [in=55,out=120,loop,looseness=8]  (c);
\end{tikzpicture}
\end{array}
&&
Q_{2}=
\begin{array}{c}
\begin{tikzpicture}[bend angle=15, looseness=1]
\node (a) at (-1,0) [vertex] {};
\node (b) at (0,0) [vertex] {};
\node (c) at (1,0) [vertex] {};
\node (d) at (2,0) [vertex] {};
\node (a1) at (-1,-0.2) {$\scriptstyle 1$};
\node (b1) at (0,-0.2) {$\scriptstyle 2$};
\node (c1) at (1,-0.2) {$\scriptstyle 3$};
\node (d1) at (2,-0.2) {$\scriptstyle 4$};
\draw[->,bend right] (a) to (b);
\draw[->,bend left] (a) to (b);
\draw[<-,bend right] (b) to (c);
\draw[->,bend left] (b) to (c);
 \draw[->]  (d) edge [in=55,out=120,loop,looseness=8]  (d);
\end{tikzpicture}
\end{array}
\end{array}
\]
are both quivers. 
\end{example}
Important technical point:  for every vertex $i$ in a quiver we should actually add in a loop at that vertex (called the trivial loop, or trivial path) and denote it by $e_i$, but we do not draw these loops.  Thus really in \ref{7.2} the quivers are 
\[
\begin{array}{ccc}
Q_{1}=
\begin{array}{c}
\begin{tikzpicture}
\node (a) at (-1,0) [vertex] {};
\node (b) at (0,0) [vertex] {};
\node (c) at (1,0) [vertex] {};
\draw[->] (a) to (b);
\draw[->] (c) to (b);
\draw[->]  (c) edge [in=55,out=120,loop,looseness=8]  (c);
 \draw[<-,densely dotted]  (a) edge [in=-120,out=-55,loop,looseness=8] node[below] {$\scriptstyle e_{1}$} (a);
  \draw[<-,densely dotted]  (b) edge [in=-120,out=-55,loop,looseness=8] node[below] {$\scriptstyle e_{2}$} (b);
   \draw[<-,densely dotted]  (c) edge [in=-120,out=-55,loop,looseness=8] node[below] {$\scriptstyle e_{3}$} (c);
 \end{tikzpicture}
\end{array}
&&
Q_{2}=
\begin{array}{c}
\begin{tikzpicture}[bend angle=15, looseness=1]
\node (a) at (-1,0) [vertex] {};
\node (b) at (0,0) [vertex] {};
\node (c) at (1,0) [vertex] {};
\node (d) at (2,0) [vertex] {};
\draw[->,bend right] (a) to (b);
\draw[->,bend left] (a) to (b);
\draw[<-,bend right] (b) to (c);
\draw[->,bend left] (b) to (c);
 \draw[->]  (d) edge [in=55,out=120,loop,looseness=8]  (d);
  \draw[<-,densely dotted]  (a) edge [in=-120,out=-55,loop,looseness=8] node[below] {$\scriptstyle e_{1}$} (a);
  \draw[<-,densely dotted]  (b) edge [in=-120,out=-55,loop,looseness=8] node[below] {$\scriptstyle e_{2}$} (b);
   \draw[<-,densely dotted]  (c) edge [in=-120,out=-55,loop,looseness=8] node[below] {$\scriptstyle e_{3}$} (c);
     \draw[<-,densely dotted]  (d) edge [in=-120,out=-55,loop,looseness=8] node[below] {$\scriptstyle e_{4}$} (d);
\end{tikzpicture}
\end{array}
\end{array}
\]
but we do not usually draw the dotted loops.

Given any quiver we can produce a $k$-algebra as follows:
\begin{defin}
\textnormal{(1)} For a quiver $Q$ denote the set of vertices by $Q_0$ and the set of arrows by $Q_1$.  For every arrow $a\in Q_1$ we define the head of $a$ (denoted $h(a)$) to be the vertex that $a$ points to, and we define the tail of $a$ to be the vertex that $a$ starts from.  For example if
\[
Q=
\begin{array}{c}
\begin{tikzpicture}
\node (a) at (-1,0) [vertex] {};
\node (b) at (0,0) [vertex] {};
\node (a1) at (-1,-0.2) {$\scriptstyle 1$};
\node (b1) at (0,-0.2) {$\scriptstyle 2$};
\draw[->] (a) to node[above] {$\scriptstyle a$} (b);
 \end{tikzpicture}
\end{array}
\]
then $h(a)=2$ and $t(a)=1$.\\[1mm]
\textnormal{(2)}  A non-trivial path in $Q$ is just a formal expression $a_1\cdot a_2\cdot \hdots\cdot  a_n$ where $a_1,\hdots, a_n$ are non-trivial arrows in $Q$ satisfying $h(a_i)=t(a_{i+1})$ for all $i$ such that $1\leq i\leq n-1$.  Pictorially this means we have a sequence of arrows
\[
\begin{array}{c}
\begin{tikzpicture}
\node (a) at (-1,0) [vertex] {};
\node (b) at (0,0) [vertex] {};
\node (c) at (1,0) [vertex] {};
\node (c2) at (1.5,0) {$\hdots$};
\node (d) at (2,0) [vertex] {};
\node (e) at (3,0) [vertex] {};
\draw[->] (a) to node[above] {$\scriptstyle a_{1}$} (b);
\draw[->] (b) to node[above] {$\scriptstyle a_{2}$} (c);
\draw[->] (d) to node[above] {$\scriptstyle a_{n}$} (e);
 \end{tikzpicture}
\end{array}
\]
in $Q$ and we just write down the formal expression $a_1\cdot a_2\cdot \hdots\cdot  a_n$.  We define the head and the tail of a path in the obvious way, namely $h(a_1\cdot a_2\cdot \hdots\cdot  a_n)=h(a_n)$ and $t(a_1\cdot a_2\cdot \hdots\cdot  a_n)=t(a_1)$.\\[1mm]
\textnormal{(3)} For a quiver $Q$ we define the path algebra $kQ$ as follows.  $kQ$ has a $k$-basis given by all non-trivial paths in $Q$ together with the trivial loops.  Multiplication is defined by 
\[
pq:=\left\{\begin{array}{cl}p\cdot q & \mbox{if }h(p)=t(q)\\ 0&\mbox{else}\end{array}\right.\qquad e_ip:=\left\{\begin{array}{cl}p & \mbox{if }t(p)=i\\ 0&\mbox{else}\end{array}\right.\qquad pe_i:=\left\{\begin{array}{cl}p & \mbox{if }h(p)=i\\ 0&\mbox{else}\end{array}\right.
\]
for all paths $p$ and $q$ and then extend by linearity.
\end{defin}

\begin{remark} (1) Although $kQ$ may be infinite dimensional, by definition every element of $kQ$ is a \emph{finite} sum $\sum\lambda_pp$ over some paths in $Q$.\\[1mm]
(2) Pictorially multiplication is like composition, i.e. if $p=\begin{array}{c}
\begin{tikzpicture}[xscale=0.75]
\node (a) at (-1,0) [vertex] {};
\node (b) at (0,0) [vertex] {};
\node (c) at (1,0) [vertex] {};
\node (c2) at (1.5,0) {$\hdots$};
\node (d) at (2,0) [vertex] {};
\node (e) at (3,0) [vertex] {};
\draw[->] (a) to node[above] {$\scriptstyle a_{1}$} (b);
\draw[->] (b) to node[above] {$\scriptstyle a_{2}$} (c);
\draw[->] (d) to node[above] {$\scriptstyle a_{n}$} (e);
 \end{tikzpicture}
\end{array}$ and $q=\begin{array}{c}
\begin{tikzpicture}[xscale=0.75]
\node (a) at (-1,0) [vertex] {};
\node (b) at (0,0) [vertex] {};
\node (c) at (1,0) [vertex] {};
\node (c2) at (1.5,0) {$\hdots$};
\node (d) at (2,0) [vertex] {};
\node (e) at (3,0) [vertex] {};
\draw[->] (a) to node[above] {$\scriptstyle b_{1}$} (b);
\draw[->] (b) to node[above] {$\scriptstyle b_{2}$} (c);
\draw[->] (d) to node[above] {$\scriptstyle b_{m}$} (e);
 \end{tikzpicture}
\end{array}$ then
\[
pq=\left\{\begin{array}{cc}\begin{array}{c}
\begin{tikzpicture}[xscale=0.65]
\node (1) at (-5,0) [vertex] {};
\node (2) at (-4,0) [vertex] {};
\node (3) at (-3,0) [vertex] {};
\node (c2) at (-2.5,0) {$\hdots$};
\node (4) at (-2,0) [vertex] {};
\node (a) at (-1,0) [vertex] {};
\node (b) at (0,0) [vertex] {};
\node (c) at (1,0) [vertex] {};
\node (c2) at (1.5,0) {$\hdots$};
\node (d) at (2,0) [vertex] {};
\node (e) at (3,0) [vertex] {};
\draw[->] (1) to node[above] {$\scriptstyle a_{1}$} (2);
\draw[->] (2) to node[above] {$\scriptstyle a_{2}$} (3);
\draw[->] (4) to node[above] {$\scriptstyle a_{n}$} (a);
\draw[->] (a) to node[above] {$\scriptstyle b_{1}$} (b);
\draw[->] (b) to node[above] {$\scriptstyle b_{2}$} (c);
\draw[->] (d) to node[above] {$\scriptstyle b_{m}$} (e);
 \end{tikzpicture}
\end{array}& \mbox{if } h(a_n)=t(b_1) \mbox{ i.e. }h(p)=t(q)\\ 0&\mbox{else}\end{array}\right.
\]
This means that $pq$ is equal to the formal expression $a_1\cdot\hdots a_n\cdot b_1\cdot \hdots b_m$ if $p$ and $q$ can be composed, and is equal to zero if $p$ and $q$ cannot be composed.  In practice this means that $kQ$ is often noncommutative.\\[1mm]
(3) $kQ$ is an algebra, with identity $1_{kQ}=\sum_{i\in Q_0}e_i$. 
\end{remark}

\begin{examples}\label{ex1} (1) Consider $Q=
\begin{tikzpicture}
\node (d) at (2,0) [vertex] {};
\end{tikzpicture}
$ (recall we never draw the trivial loop, but it is there).  Then the basis of $kQ$ is given by $e_1$, the only path.  Hence every element in $kQ$ looks like $\lambda e_1$ for some $\lambda\in k$.  Multiplication is given by $e_1e_1=e_1$ extended by linearity, which means
$(\lambda e_1)(\mu e_1)=(\lambda\mu)e_1e_1=(\lambda\mu)e_1$.  This implies that $kQ$ is just the field $k$.\\[1mm]
(2) Consider $Q=
\begin{array}{c}
\begin{tikzpicture}
\node (a) at (-1,0) [vertex] {};
\node (b) at (0,0) [vertex] {};
\node (a1) at (-1,-0.2) {$\scriptstyle 1$};
\node (b1) at (0,-0.2) {$\scriptstyle 2$};
\draw[->] (a) to node[above] {$\scriptstyle a$} (b);
 \end{tikzpicture}
\end{array}$.  The basis for $kQ$ is given by $e_1$, $e_2$, $a$.  An element of $kQ$ is by definition $\lambda_1 e_1+\lambda_2 e_2+\lambda_3 a$ for some $\lambda_1,\lambda_2,\lambda_3\in k$.  Multiplication is given by
\begin{equation*}\begin{split}
(\lambda_1 e_1+\lambda_2 e_2+\lambda_3 a)(\mu_1 e_1+\mu_2 e_2+\mu_3 a)&=\lambda_1\mu_1e_1e_1+\lambda_1\mu_2e_1e_2+\lambda_1\mu_3e_1a\\ 
&+\lambda_2\mu_1e_2e_1+\lambda_2\mu_2e_2e_2+\lambda_2\mu_3e_2a\\
&+\lambda_3\mu_1ae_1+\lambda_3\mu_2ae_2+\lambda_3\mu_3aa\\
\end{split}\end{equation*}
which is equal to $\lambda_1\mu_1 e_1+0+\lambda_1\mu_3a+0+\lambda_2\mu_2e_2+0+0+\lambda_3\mu_2a+0$.  That is
\[
(\lambda_1 e_1+\lambda_2 e_2+\lambda_3 a)(\mu_1 e_1+\mu_2 e_2+\mu_3 a)=\lambda_1\mu_1 e_1+\lambda_2\mu_2e_2+(\lambda_1\mu_3+\lambda_3\mu_2)a
\]
This should be familiar. If we write $\lambda_1 e_1+\lambda_2 e_2+\lambda_3 a$ as $\left(\begin{array}{cc} \lambda_1&\lambda_3\\ 0&\lambda_2\end{array}\right)$, then the above multiplication is simply
\[
\left(\begin{array}{cc} \lambda_1&\lambda_3\\ 0&\lambda_2\end{array}\right)\left(\begin{array}{cc} \mu_1&\mu_3\\ 0&\mu_2\end{array}\right)=\left(\begin{array}{cc} \lambda_1\mu_1&\lambda_1\mu_3+\lambda_3\mu_2\\ 0&\lambda_2\mu_2\end{array}\right)
\]
which shows that $kQ\cong U_2(k)$, upper triangular matrices.   More formally define $\psi:kQ\to U_2(k)$ by $e_1\mapsto E_{11}$, $e_2\mapsto E_{22}$ and $a\mapsto E_{12}$ and extend by linearity.  By above this is a $k$-algebra homomorphism which is clearly surjective.  Since both sides have dimension three, $\psi$ is also injective. \\[1mm]
(3) Consider  $Q=
\begin{tikzpicture}[bend angle=15, looseness=1]
\node (d) at (2,0) [vertex] {};
 \draw[->]  (d) edge [in=55,out=120,loop,looseness=8] node[above] {$\scriptstyle \alpha$}  (d);
\end{tikzpicture}
$.  The basis of $kQ$ is given by $e_1,\alpha, \alpha\cdot\alpha,\alpha\cdot\alpha\cdot\alpha,\hdots$ and so $kQ$ is infinite dimensional.  If we agree to write $\underbrace{\alpha\cdot\hdots\cdot\alpha}_{n}:=\alpha^n$ then every element of $kQ$ is by definition a \emph{finite} sum of paths in $Q$, i.e. a polynomial in $\alpha$.  Since all paths can be composed the multiplication in $kQ$ is 
\[
\alpha^i\alpha^j=(\underbrace{\alpha\cdot\hdots\cdot\alpha}_{i})(\underbrace{\alpha\cdot\hdots\cdot\alpha}_{j})=\underbrace{\alpha\cdot\hdots\cdot\alpha}_{i+j}=\alpha^{i+j}
\]
extended by linearity, i.e. polynomial multiplication.  This shows that $kQ\cong k[X]$.\\[1mm]
(4) Consider $Q=
\begin{array}{c}
\begin{tikzpicture}
\node (a) at (-1,0) [vertex] {};
\node (b) at (0,0) [vertex] {};
\node (c) at (1,0) [vertex] {};
\node (a1) at (-1,-0.2) {$\scriptstyle 1$};
\node (b1) at (0,-0.2) {$\scriptstyle 2$};
\node (c1) at (1,-0.2) {$\scriptstyle 3$};
\draw[->] (a) to node[above] {$\scriptstyle a$} (b);
\draw[->] (b) to node[above] {$\scriptstyle b$} (c);
 \end{tikzpicture}
\end{array}$.  The basis for $kQ$ is given by $e_1$, $e_2$, $e_3$, $a$, $b$ and $a\cdot b$. and so $kQ$ is six-dimensional.   In $kQ$ the product $ba$ equals zero whereas the product $ab$ equals the path $a\cdot b$.  Some other products:
\[
aa=0\quad e_1a=a\quad e_2a=0\quad e_1e_2=0\quad (a\cdot b)e_3=a\cdot b
\]
In fact, $kQ\cong  U_3(k)$ in a similar way to (2).  See Exercise~\ref{Ex6.22}.
\end{examples}

Thus by studying quivers we have recovered many of the algebras that we already know.  In fact if we now study \emph{quivers with relations} we can obtain even more:
\begin{defin}
For a given quiver $Q$ a relation is just a $k$-linear combination of paths in $Q$, each with the same head and tail.  Given a finite number of specified relations $R_1,\hdots, R_n$ we can form the two-sided ideal $R:=kQR_1kQ+\hdots+kQR_nkQ$ of $kQ$.  We call $(Q,R)$ a quiver with relations and we call $kQ/R$ the path algebra of a quiver with relations. 
\end{defin}

\begin{remark}
Informally think of a relation $p-q$ as saying `going along path $p$ is the same as going along path $q$' since $p=q$ in $kQ/R$.  Because of this we sometimes say `subject to the relation $p=q$' when we really mean `subject to the relation $p-q$'.
\end{remark}

\begin{examples}\label{ex2}
(1)  Consider $Q=
\begin{tikzpicture}[bend angle=15, looseness=1]
\node (d) at (2,0) [vertex] {};
 \draw[->]  (d) edge [in=55,out=120,loop,looseness=8] node[above] {$\scriptstyle \alpha$}  (d);
\end{tikzpicture}
$  subject to the relation $\alpha\cdot\alpha$.  Then $kQ\cong k[X]$ and under this isomorphism the two-sided ideal generated by $\alpha\cdot\alpha$ corresponds to the ideal generated by $X^2$ in $k[X]$.  Thus $kQ/R\cong k[X]/(X^2)$. \\
(2) Consider $Q=
\begin{array}{c}
\begin{tikzpicture}[bend angle=20, looseness=1]
\node (a) at (-1,0) [vertex] {};
\node (b) at (0,0) [vertex] {};
\node (a1) at (-1,-0.2) {$\scriptstyle 1$};
\node (b1) at (0,-0.2) {$\scriptstyle 2$};
\draw[->,bend left] (b) to node[below] {$\scriptstyle b$} (a);
\draw[->,bend left] (a) to node[above] {$\scriptstyle a$} (b);
\end{tikzpicture}
\end{array}
$ subject to the relations $a\cdot b-e_1$ and $b\cdot a-e_2$.  Then $kQ/R\cong M_2(k)$.  To see this notice (in a very similar way to \ref{ex1}(2)) that there is a $k$-algebra homomorphism $\psi:kQ\to M_2(k)$ by sending $e_1\mapsto E_{11}$, $e_2\mapsto E_{22}$, $a\mapsto E_{12}$, $b\mapsto E_{21}$.  Now
\[
\begin{array}{c}
\psi(a\cdot b-e_1)=E_{12}E_{21}-E_{11}=E_{11}-E_{11}=0\\
\psi(b\cdot a-e_2)=E_{21}E_{12}-E_{22}=E_{22}-E_{22}=0
\end{array}
\]
and so $\psi$ induces a well-defined algebra homomorphism $kQ/R\to M_2(k)$.  It is clearly surjective.  But the dimension of $kQ/R$ is four which is the same as the dimension of $M_2(k)$.  Hence the map is also injective, so $kQ/R\cong M_2(k)$.
\end{examples}

\begin{defin}
(1)  Let $kQ$ be the path algebra of a quiver $Q$.  A finite dimensional quiver representation of $Q$ is the assignment to every vertex $i\in Q_0$ a finite dimensional vector space $V_i$ and to every arrow $a$ a linear map $f_a:V_{t(a)}\to V_{h(a)}$. We sometimes denote this data by $(V_i,f_a)$.  Note that by convention we always assign to the trivial loops $e_i$ the identity linear map.\\
(2) If $(Q,R)$ is a quiver with relations, we define a finite dimensional quiver representation of $(Q,R)$  to be a finite dimensional quiver representation of $Q$ such that for all relations $R_i$, if $R_i=\sum\lambda_pp$ then $\sum\lambda_pf_p=0_{\rm map}$.
\end{defin}
Note that if there are no relations (i.e. $R=0$) then trivially a finite dimensional quiver representation of $Q$ is the same thing as a finite dimensional quiver representation of $(Q,R)$.
\begin{example}
 Consider $Q=
\begin{array}{c}
\begin{tikzpicture}[bend angle=20, looseness=1]
\node (a) at (-1,0) [vertex] {};
\node (b) at (0,0) [vertex] {};
\node (c) at (1,0) [vertex] {};
\draw[->,bend right] (a) to node[below] {$\scriptstyle b$} (b);
\draw[->,bend left] (a) to node[above] {$\scriptstyle a$} (b);
\draw[->] (b) -- node[above] {$\scriptstyle c$} (c);
\end{tikzpicture}
\end{array}$
 subject to the relation $a\cdot c-b\cdot c$.  If we denote
\[
M:=\begin{array}{c}
\begin{tikzpicture}[bend angle=20, looseness=1,xscale=1.5]
\node (a) at (-1,0)  {$k$};
\node (b) at (0,0) {$k$};
\node (c) at (1,0)  {$k$};
\draw[->,bend right] (a) to node[below] {$\scriptstyle f_{b}=id$} (b);
\draw[->,bend left] (a) to node[above] {$\scriptstyle f_{a}=id$} (b);
\draw[->] (b) -- node[above] {$\scriptstyle f_{c}=id$} (c);
\end{tikzpicture}
\end{array}\quad\mbox{and}\quad N:=\begin{array}{c}
\begin{tikzpicture}[bend angle=20, looseness=1,xscale=1.5]
\node (a) at (-1,0)  {$k$};
\node (b) at (0,0) {$k$};
\node (c) at (1,0)  {$k$};
\draw[->,bend right] (a) to node[below] {$\scriptstyle f_{b}=0$} (b);
\draw[->,bend left] (a) to node[above] {$\scriptstyle f_{a}=id$} (b);
\draw[->] (b) -- node[above] {$\scriptstyle f_{c}=id$} (c);
\end{tikzpicture}
\end{array}
\]
then $M$ is a quiver representation of $(Q,R)$ since the relation $f_a\cdot f_c-f_b\cdot f_c=0$ holds.  However $N$ is not a quiver representation of $(Q,R)$ since the relation does not hold.
\end{example}

To make quiver representations into a category, we must define morphisms:
\begin{defin}
Suppose $(Q,R)$ is a quiver with relations, and $V=(V_i,f_a)$ and $W=(W_i,g_a)$ are quiver representations for $(Q,R)$.  A morphism of quiver representations $\psi$ from $V$ to $W$ is given by specifying, for every vertex $i$, a linear map $\psi_i:V_i\rightarrow W_i$ such that for every arrow $a\in Q_1$
\[
\begin{array}{c}
\begin{tikzpicture}[xscale=1.2,yscale=1.2]
\node (a) at (-1.5,0)  {$V_{t(a)}$};
\node (b) at (0,0) {$V_{h(a)}$};
\node (a1) at (-1.5,-1.5)  {$W_{t(a)}$};
\node (b1) at (0,-1.5) {$W_{h(a)}$};
\draw[->] (a) -- node[above] {$\scriptstyle f_{a}$} (b);
\draw[->] (a1) -- node[above] {$\scriptstyle g_{a}$} (b1);
\draw[->] (a) -- node[left] {$\scriptstyle\psi_{t(a)}$} (a1);
\draw[->] (b) -- node[right] {$\scriptstyle\psi_{h(a)}$}(b1);
\end{tikzpicture}
\end{array}
\] 
we have $f_a\cdot \psi_{h(a)}=\psi_{t(a)}\cdot g_a$.
\end{defin}

\begin{examples}\label{ex4}
(1) Consider $Q=
\begin{tikzpicture}
\node (a) at (-1,0) [vertex] {};
\node (b) at (0,0) [vertex] {};
\draw[->] (a) to node[above] {$\scriptstyle a$} (b);
 \end{tikzpicture}$ (no relations).  Consider $M$ and $N$ defined in \ref{ex3}(2).  To specify a morphism of quiver representations from $M$ to $N$ we must find linear maps $\psi_1$ and $\psi_2$ such that 
\[
\begin{array}{c}
\begin{tikzpicture}[xscale=1,yscale=1]
\node (a) at (-1.5,0)  {$k$};
\node (b) at (0,0) {$k$};
\node (a1) at (-1.5,-1.5)  {$k$};
\node (b1) at (0,-1.5) {$0$};
\draw[->] (a) -- node[above] {$\scriptstyle\rm id$} (b);
\draw[->] (a1) -- node[above] {$\scriptstyle\rm 0$} (b1);
\draw[->] (a) -- node[left] {$\scriptstyle\psi_{1}$} (a1);
\draw[->] (b) -- node[right] {$\scriptstyle\psi_{2}$}(b1);
\end{tikzpicture}
\end{array}
\] 
commutes.  Note $\psi_2$ is the zero map, whereas $\psi_1$ can be an arbitrary scalar.  Now to specify  a morphism of quiver representations from $N$ to $M$ we must find linear maps $\phi_1$ and $\phi_2$ such that 
\[
\begin{array}{c}
\begin{tikzpicture}[xscale=1,yscale=1]
\node (a) at (-1.5,0)  {$k$};
\node (b) at (0,0) {$k$};
\node (a1) at (-1.5,-1.5)  {$k$};
\node (b1) at (0,-1.5) {$0$};
\draw[->] (a) -- node[above] {$\scriptstyle\rm id$} (b);
\draw[->] (a1) -- node[above] {$\scriptstyle\rm 0$} (b1);
\draw[<-] (a) -- node[left] {$\scriptstyle\phi_{1}$} (a1);
\draw[<-] (b) -- node[right] {$\scriptstyle\phi_{2}$}(b1);
\end{tikzpicture}
\end{array}
\] 
commutes.  Note $\phi_2$ is the zero map, and the fact that the diagram commutes forces $\phi_1$ to be the zero map too.  This shows that the only morphism of quiver representations from $N$ to $M$ is the zero morphism.\\[1mm]
(2) Consider $Q=
\begin{array}{c}
\begin{tikzpicture}[bend angle=20, looseness=1]
\node (a) at (-1,0) [vertex] {};
\node (b) at (0,0) [vertex] {};
\draw[->,bend left] (b) to node[below] {$\scriptstyle b$} (a);
\draw[->,bend left] (a) to node[above] {$\scriptstyle a$} (b);
\end{tikzpicture}
\end{array}
$ subject to the relations $a\cdot b-e_1$ and $b\cdot a-e_2$.  By \ref{ex2}(2) $kQ/R\cong M_2(k)$.  Suppose that 
$M:=
\begin{array}{c}
\begin{tikzpicture}[bend angle=20, looseness=1]
\node (a) at (-1.5,0)  {$k^{n}$};
\node (b) at (0,0) {$k^{n}$};
\draw[->,bend left] (b) to node[below] {$\scriptstyle f_{b}$} (a);
\draw[->,bend left] (a) to node[above] {$\scriptstyle f_{a}$} (b);
\end{tikzpicture}
\end{array}
$
is a quiver representation of $(Q,R)$.  Note that since $f_a\cdot f_b=id$ and $f_b\cdot f_a=id$, both $f_a$ and $f_b$ must be linear isomorphisms.  Now it is clear that 
$
\begin{array}{c}
\begin{tikzpicture}[bend angle=20, looseness=1]
\node (a) at (-1.5,0)  {$k^{n}$};
\node (b) at (0,0) {$k^{n}$};
\draw[->,bend left] (b) to node[below] {$\scriptstyle\rm id$} (a);
\draw[->,bend left] (a) to node[above] {$\scriptstyle\rm id$} (b);
\end{tikzpicture}
\end{array}
$ 
is a quiver representation of $(Q,R)$ and further
\[
\begin{array}{c}
\begin{tikzpicture}[bend angle=20, looseness=1]
\node (a) at (-1.5,0)  {$k^{n}$};
\node (b) at (0,0) {$k^{n}$};
\node (a1) at (-1.5,-1.5)  {$k^{n}$};
\node (b1) at (0,-1.5) {$k^{n}$};
\draw[->,bend left] (b) to node[below] {$\scriptstyle f_{b}$} (a);
\draw[->,bend left] (a) to node[above] {$\scriptstyle f_{a}$} (b);
\draw[->,bend left] (b1) to node[below] {$\scriptstyle\rm id$} (a1);
\draw[->,bend left] (a1) to node[above] {$\scriptstyle\rm id$} (b1);
\draw[->] (a1) -- node[left] {$\scriptstyle\rm id$} (a);
\draw[->] (b1) -- node[right] {$\scriptstyle f_{a}$}(b);
\end{tikzpicture}
\end{array}
\] 
is a morphism of quiver representations since both 

\[
\begin{array}{ccc}
\begin{array}{c}
\begin{tikzpicture}[bend angle=20, looseness=1]
\node (a) at (-1.5,0)  {$k^{n}$};
\node (b) at (0,0) {$k^{n}$};
\node (a1) at (-1.5,-1.5)  {$k^{n}$};
\node (b1) at (0,-1.5) {$k^{n}$};
\draw[->,bend left] (a) to node[above] {$\scriptstyle f_{a}$} (b);
\draw[->,bend left] (a1) to node[above] {$\scriptstyle\rm id$} (b1);
\draw[->] (a1) -- node[left] {$\scriptstyle\rm id$} (a);
\draw[->] (b1) -- node[right] {$\scriptstyle f_{a}$}(b);
\node at (1.5,-0.75) {and};
\node (a) at (3,0)  {$k^{n}$};
\node (b) at (4.5,0) {$k^{n}$};
\node (a1) at (3,-1.5)  {$k^{n}$};
\node (b1) at (4.5,-1.5) {$k^{n}$};
\draw[->,bend left] (b) to node[below] {$\scriptstyle f_{b}$} (a);
\draw[->,bend left] (b1) to node[below] {$\scriptstyle\rm id$} (a1);
\draw[->] (a1) -- node[left] {$\scriptstyle\rm id$} (a);
\draw[->] (b1) -- node[right] {$\scriptstyle f_{a}$}(b);
\end{tikzpicture}
\end{array}
\end{array}
\]
commute.
\end{examples}

\begin{defin}\label{defin:catquivermod}
For a quiver with relations $(Q,R)$ we define $\tt{fRep}(kQ,R)$ to be the category of all finite dimensional quiver representations, where the morphisms are defined to be all morphisms of quiver representations. We denote by $\fdmod kQ/R$ the category of finite dimensional right $kQ/R$-modules
\end{defin}

The category $\tt{fRep}(kQ,R)$ is visual and computable, whereas the category $\fdmod kQ/R$ is more abstract. The following result is one of the main motivations for studying quivers.

\begin{thm}\label{module}
Suppose $(Q,R)$ is a quiver with relations.  Then finite dimensional quiver representations of $(Q,R)$ are the same as finite dimensional right $kQ/R$-modules.  More specifically there is an equivalence of categories
\[
\tt{fRep}(kQ,R)\simeq \fdmod kQ/R.
\]
\end{thm}
\begin{proof}
(Sketch proof)  If $M$ is a finite dimensional right $kQ/R$-module define a finite dimensional quiver representation of $(Q,R)$ by setting $V_i=Me_i$ and $f_a:Me_{t(a)}\to Me_{h(a)}$ by $f_a(x)=xa$.  Conversely given a quiver representation $(V_i,f_a)$ of $(Q,R)$ define $M=\oplus_{i\in Q_0}V_i$.  Denote $V_i\stackrel{\iota_i}{\to}M\stackrel{\pi_i}{\to}V_i$ to be inclusion and projection, then $M$ is a right $kQ/R$-module via
\begin{eqnarray*}
x\cdot(a_1\cdot\hdots\cdot a_m) &:=&\iota_{h(a_m)}f_{a_m}\circ\hdots\circ f_{a_1}\pi_{t(a_1)}(x)\\
x\cdot e_i &:=&\iota_i\circ\pi_i(x)
\end{eqnarray*}
It is fairly straightforward to show that these are inverses.
\end{proof}

\begin{remark}(1) Suppose we want to understand the modules of some algebra $A$.   Theorem~\ref{module} says that, provided we can find a quiver $Q$ with relations $R$ such that $A\cong kQ/R$, then $A$-modules are precisely the same as quiver representations of $(Q,R)$.  This means that $A$-modules are very easy to write down!  Hence \ref{module} gives us a method to visualize modules.\\[1mm]
(2) Note that in the proof of \ref{module} if $(V_i,f_a)$ is a quiver representation of $(Q,R)$ then the corresponding $kQ/R$ module has dimension $\sum_{i\in Q_0}{\rm dim}_kV_i$. 
\end{remark}
\begin{examples}\label{ex3}
(1) Consider $Q=
\begin{tikzpicture}
\node (d) at (2,0) [vertex] {};
\end{tikzpicture}
$ (no relations).  Then by \ref{ex1}(1) $kQ\cong k$ and so $k$-modules are the same as quiver representations for $Q$.  But here to specify a quiver representation we just need to assign a vector space to the only vertex, and so quiver representations are precisely the same as vector spaces.  This just says that $k$-modules are the same as vector spaces.\\[1mm]
(2) Consider  $Q=
\begin{tikzpicture}
\node (a) at (-1,0) [vertex] {};
\node (b) at (0,0) [vertex] {};
\draw[->] (a) to node[above] {$\scriptstyle a$} (b);
 \end{tikzpicture}$ (no relations).  By \ref{ex1}(2) $kQ\cong U_2(k)$ and so $U_2(k)$-modules are the same as quiver representations of $Q$.  Hence examples of $U_2(k)$-modules include
\[
M:=\left(
\begin{array}{c}\begin{tikzpicture}
\node (a) at (-1,0) {$k$};
\node (b) at (0,0) {$k$};
\draw[->] (a) to node[above] {$\scriptstyle\rm id$} (b);
 \end{tikzpicture}
 \end{array}\right)\quad\mbox{and}\quad N:=\left(
 \begin{array}{c}\begin{tikzpicture}
\node (a) at (-1,0) {$k$};
\node (b) at (0,0) {$0$};
\draw[->] (a) to node[above] {$\scriptstyle 0$} (b);
 \end{tikzpicture}
 \end{array}
 \right).
\]
(3) Consider $Q=
\begin{tikzpicture}[bend angle=15, looseness=1]
\node (d) at (2,0) [vertex] {};
 \draw[->]  (d) edge [in=55,out=120,loop,looseness=8] node[above] {$\scriptstyle \alpha$}  (d);
\end{tikzpicture}
$ (no relations).  By \ref{ex1}(3) $kQ\cong k[X]$ and so $k[X]$-modules are the same as quiver representations of $Q$.  For example
\[
\begin{tikzpicture}[bend angle=15, looseness=1]
\node (d) at (2,0)  {$k$};
 \draw[->]  (d) edge [in=55,out=120,loop,looseness=6] node[above] {$\scriptstyle\rm id$}  (d);
\end{tikzpicture}
\]
is a quiver representation of $Q$ and so is a $k[X]$-module.   Now here a quiver representation of $Q$ is given by specifying a vector space $V$ together with a linear map from $V$ to itself and so (since $kQ\cong k[X]$)  $k[X]$-modules are given by specifying a vector space $V$ and a linear map $V\stackrel{\alpha}{\to}V$.\\[1mm]
(4) Consider $Q=
\begin{tikzpicture}[bend angle=15, looseness=1]
\node (d) at (2,0) [vertex] {};
 \draw[->]  (d) edge [in=55,out=120,loop,looseness=8] node[above] {$\scriptstyle \alpha$}  (d);
\end{tikzpicture}
$ subject to the relation $\alpha\cdot\alpha$.  By \ref{ex2}(1) $kQ/R\cong k[X]/(X^2)$.  Now 
\[
\begin{tikzpicture}[bend angle=15, looseness=1]
\node (d) at (2,0) {$k$};
 \draw[->]  (d) edge [in=55,out=120,loop,looseness=6] node[above] {$\scriptstyle f_{\alpha}=\rm id$}  (d);
\end{tikzpicture}
\] 
is not a quiver representation for $(Q,R)$ since the relation $f_\alpha\cdot f_\alpha=0$ does not hold ($id\cdot id\neq 0$), hence it is not  a module for $k[X]/(X^2)$.
\end{examples}

\begin{remark}
If $(Q,R)$ is a quiver with relations, then under \ref{module}  morphisms of quiver representations correspond to $kQ/R$-module homomorphisms.  Further\\
$\psi=(\psi_i)$ corresponds to a $\left\{ \begin{array}{l} \mbox{monomorphism}\\\mbox{epimorphism}\\\mbox{isomorphism} \end{array} \right.$ $\iff$ each $\psi_i$ is an $\left\{ \begin{array}{l} \mbox{injective}\\\mbox{surjective}\\\mbox{bijective} \end{array} \right.$ linear map.  Thus if we have a morphism of quiver representations $\psi=(\psi_i):N\to M$ in which each $\psi_i$ is an injective linear map, we call $N$ a subrepresentation of $M$ (since under the above correspondence $N$ embeds into the module of $M$, so we can view $N$ as a submodule of $M$).
\end{remark}

\begin{example}\label{ex5}
In Example~\ref{ex4}(2) the morphism of quiver representations
\[
\begin{array}{c}
\begin{tikzpicture}[bend angle=20, looseness=1]
\node (a) at (-1.5,0)  {$k^{n}$};
\node (b) at (0,0) {$k^{n}$};
\node (a1) at (-1.5,-1.5)  {$k^{n}$};
\node (b1) at (0,-1.5) {$k^{n}$};
\draw[->,bend left] (b) to node[below] {$\scriptstyle f_{b}$} (a);
\draw[->,bend left] (a) to node[above] {$\scriptstyle f_{a}$} (b);
\draw[->,bend left] (b1) to node[below] {$\scriptstyle\rm id$} (a1);
\draw[->,bend left] (a1) to node[above] {$\scriptstyle\rm id$} (b1);
\draw[->] (a1) -- node[left] {$\scriptstyle\rm id$} (a);
\draw[->] (b1) -- node[right] {$\scriptstyle f_{a}$}(b);
\end{tikzpicture}
\end{array}
\] 
is an isomorphism since both the connector maps $\rm id$ and $f_a$ are linear isomorphisms.  This shows that as $kQ/R\cong M_2(k)$-modules,
\[
M=\left(\begin{array}{c}
\begin{tikzpicture}[bend angle=20, looseness=1]
\node (a) at (-1.5,0)  {$k^{n}$};
\node (b) at (0,0) {$k^{n}$};
\draw[->,bend left] (b) to node[below] {$\scriptstyle f_{b}$} (a);
\draw[->,bend left] (a) to node[above] {$\scriptstyle f_{a}$} (b);
\end{tikzpicture}
\end{array}\right)\cong
\left(\begin{array}{c}
\begin{tikzpicture}[bend angle=20, looseness=1]
\node (a) at (-1.5,0)  {$k^{n}$};
\node (b) at (0,0) {$k^{n}$};
\draw[->,bend left] (b) to node[below] {$\scriptstyle\rm id$} (a);
\draw[->,bend left] (a) to node[above] {$\scriptstyle\rm id$} (b);
\end{tikzpicture}
\end{array}\right).
\]
\end{example}

Taking the direct sum of modules can also be visualized easily in the language of quivers:
\begin{defin}
Suppose $(Q,R)$ is a quiver with relations, and $V=(V_i,f_a), W=(W_i,g_a)$ are quiver representations of $(Q,R)$.  Then we define the direct sum $V\oplus W$ to be the quiver representation of $(Q,R)$ given by $(V_i\oplus W_i,\left(\begin{smallmatrix}f_a&0\\ 0&g_a\end{smallmatrix}\right))$
\end{defin}

\begin{example}\label{quiver example decompose}
(1)  Consider $Q=
\begin{tikzpicture}
\node (a) at (-1,0) [vertex] {};
\node (b) at (0,0) [vertex] {};
\draw[->] (a) to node[above] {$\scriptstyle a$} (b);
 \end{tikzpicture}$ (no relations) as in \ref{ex3}(2). Then 
\[
\left(
\begin{array}{c}
\begin{tikzpicture}
\node (a) at (-1,0) {$k$};
\node (b) at (0,0) {$k$};
\draw[->] (a) to node[above] {$\scriptstyle id$} (b);
 \end{tikzpicture}
 \end{array}\right)\bigoplus\left(\begin{array}{c}
\begin{tikzpicture}
\node (a) at (-1,0) {$k$};
\node (b) at (0,0) {$0$};
\draw[->] (a) to node[above] {$\scriptstyle 0$} (b);
 \end{tikzpicture}
 \end{array}\right)=
 \begin{array}{c}
\begin{tikzpicture}[xscale=2]
\node (a) at (-1,0) {$k\oplus k$};
\node (b) at (0,0) {$k\oplus 0$};
\draw[->] (a) to node[above] {$\scriptstyle  \left(\begin{smallmatrix}{\rm id}&0\\ 0&0\end{smallmatrix}\right)$} (b);
 \end{tikzpicture}
 \end{array}
 =\begin{array}{c}
\begin{tikzpicture}[xscale=1.5]
\node (a) at (-1,0) {$k^{2}$};
\node (b) at (0,0) {$k$};
\draw[->] (a) to node[above] {$\scriptstyle  \left(\begin{smallmatrix}{\rm id}\\ 0\end{smallmatrix}\right)$} (b);
 \end{tikzpicture}
 \end{array}
\]
(2)  In Example~\ref{ex5}
\[
M\cong \left(
\begin{array}{c}
\begin{tikzpicture}[bend angle=20, looseness=1]
\node (a) at (-1.5,0)  {$k^{n}$};
\node (b) at (0,0) {$k^{n}$};
\draw[->,bend left] (b) to node[below] {$\scriptstyle\rm id$} (a);
\draw[->,bend left] (a) to node[above] {$\scriptstyle\rm id$} (b);
\end{tikzpicture}
\end{array}
\right)=\bigoplus_{i=1}^{n}\left(
\begin{array}{c}
\begin{tikzpicture}[bend angle=20, looseness=1]
\node (a) at (-1.5,0)  {$k$};
\node (b) at (0,0) {$k$};
\draw[->,bend left] (b) to node[below] {$\scriptstyle\rm id$} (a);
\draw[->,bend left] (a) to node[above] {$\scriptstyle\rm id$} (b);
\end{tikzpicture}
\end{array}
\right).
\]
\end{example}

\subsection{Exercises}

\begin{ex}\label{Ex6.21}
Write down the dimension (if it is finite) of the following quiver algebras, where there are no relations.
\[
\begin{array}{ccc}
\begin{array}{cl}
\rm{(1)} &
\begin{array}{c}
\begin{tikzpicture}[bend angle=15, looseness=1]
\node (a) at (0,0) [vertex] {};
\node (b) at (1,0) [vertex] {};
\node (c) at (0,1) [vertex] {};
%\draw[->]  (a) edge [in=55,out=120,loop,looseness=8] node[above] {$\scriptstyle \alpha$}  (a);
\draw[->] (a) -- (b);
\draw[->] (c) -- (b);
\draw[->,bend left=30,looseness=1] (c) to (b);
\draw[->] (c) -- (a);
\end{tikzpicture} 
\end{array}
\end{array}&
\begin{array}{cl}
\rm{(2)} &
\begin{array}{c}
\begin{tikzpicture}[bend angle=15, looseness=1]
\node (a) at (0,0) [vertex] {};
\node (b) at (1,0) [vertex] {};
\node (c) at (0,1) [vertex] {};
\node (d) at (1,1) [vertex] {};
\draw[->] (a) -- (b);
\draw[->] (a) -- (c);
\draw[->] (b) -- (d);
\draw[->] (c) -- (d);
\draw[->] (a) -- (d);
\end{tikzpicture} 
\end{array}
\end{array}&
\begin{array}{cl}
\rm{(3)} &
\begin{array}{c}
\begin{tikzpicture}[bend angle=15, looseness=1]
\node (a) at (0,0) [vertex] {};
\node (b) at (1,0) [vertex] {};
\node (c) at (0,1) [vertex] {};
\node (d) at (1,1) [vertex] {};
\draw[->] (a) -- (b);
\draw[->] (a) -- (c);
\draw[->] (b) -- (d);
\draw[->] (c) -- (d);
\draw[<-] (a) -- (d);
\end{tikzpicture} 
\end{array}
\end{array}
\end{array}
\]
What is the general result? 
\end{ex}

\begin{ex} \label{Ex6.22}
\begin{enumerate}
\item Show that the algebra $U_n(k)$ of upper triangular matrices is algebra-isomorphic to the path algebra of the quiver
\[
\begin{array}{c}
\begin{tikzpicture}[xscale=1]
\node (a) at (-1,0) [vertex] {};
\node (b) at (0,0) [vertex] {};
\node (c) at (1,0) [vertex] {};
\node (c2) at (1.5,0) {$\hdots$};
\node (d) at (2,0) [vertex] {};
\node (e) at (3,0) [vertex] {};
\node at (-1,-0.25) {$\scriptstyle 1$};
\node at (0,-0.25)  {$\scriptstyle 2$};
\node at (1,-0.25)  {$\scriptstyle 3$};
\node at (2,-0.25)  {$\scriptstyle n-1$};
\node at (3,-0.25)  {$\scriptstyle n$};
\draw[->] (a) to (b);
\draw[->] (b) to (c);
\draw[->] (d) to (e);
 \end{tikzpicture}
\end{array}
\]
subject to no relations.  How do we view the algebra $D_n(k)$ of diagonal matrices in the above picture?
\item {} (This question shows that $U_2(k)$ is not a semisimple algebra).  Consider the case $n=2$ (i.e. $U_2(k)\cong\begin{array}{c}
\begin{tikzpicture}[xscale=1]
\node (a) at (-1,0) [vertex] {};
\node (b) at (0,0) [vertex] {};
\draw[->] (a) to (b);
 \end{tikzpicture}
\end{array}$) and let $M$ be the quiver representation (= $U_2(k)$-module) 
\[
\begin{array}{c}
\begin{tikzpicture}[xscale=1]
\node (a) at (-1,0) {$k$};
\node (b) at (0,0)  {$k$};
\draw[->] (a) to node[above] {$\scriptstyle\rm id$} (b);
 \end{tikzpicture}
\end{array}
\]
Show that $M$ has a subrepresentation (= $U_2(k)$ submodule) $ N:=\begin{array}{c}
\begin{tikzpicture}[xscale=1]
\node (a) at (-1,0) {0};
\node (b) at (0,0) {$k$};
\draw[->] (a) to node[above] {$\scriptstyle 0$} (b);
 \end{tikzpicture}
\end{array}$ and that there does not exist a submodule $N^\prime$ such that $M=N\oplus N^\prime$.
\end{enumerate}
\end{ex}

\begin{ex}\label{Ex6.23}
(From quivers to algebras). Consider the following quivers with relations
\[
\begin{array}{ccc}
\begin{array}{cl}
\rm{(1)} &
\begin{array}{c}
\begin{tikzpicture}[bend angle=15, looseness=1]
\node (d) at (2,0) [vertex] {};
 \draw[->]  (d) edge [in=55,out=120,loop,looseness=8] node[above] {$\scriptstyle \alpha$}  (d);
\end{tikzpicture} \\ \\
\alpha^3=e_1
\end{array}
\end{array}&
\begin{array}{cl}
\rm{(2)} &
\begin{array}{c}
\begin{tikzpicture}[bend angle=15, looseness=1]
\node (d) at (2,0) [vertex] {};
 \draw[->]  (d) edge [in=145,out=-145,loop,looseness=8] node[left] {$\scriptstyle \alpha$}  (d);
  \draw[<-]  (d) edge [in=-35,out=35,loop,looseness=8] node[right] {$\scriptstyle \beta$}  (d);
\end{tikzpicture} \\ \\
\alpha\beta=\beta\alpha
\end{array}
\end{array}&
\begin{array}{cl}
\rm{(3)} &
\begin{array}{c}
\begin{tikzpicture}[bend angle=15, looseness=1]
\node at (1,0) [vertex] {};
\node at (2,0) [vertex] {};
 \draw[->]  (d) edge [in=55,out=120,loop,looseness=8] node[above] {$\scriptstyle \alpha$}  (d);
\end{tikzpicture} \\ \\
\rm{(no\,\, relations)}
\end{array}
\end{array}
\end{array}
\]
Identify each with an algebra you are already familiar with.  If $k=\mathbb{C}$ is there a quiver with no relations which is isomorphic to the quiver with relations in (1)?
\end{ex}

\begin{ex}\label{Ex6.24}
(From algebras to quivers). Write the following algebras as quivers with relations (there is not a unique way of answering these --- try to solve them in as many ways as possible):
\begin{enumerate}
\item The free algebra in $n$ variables.
\item The polynomial ring in $n$ variables.
\item $R\times S$, given knowledge of $R$ and $S$ as quivers with relations.
\item The group algebra $\mathbb{C}G$, where $G$ is any finite group.
\item Any $k$-algebra given by a finite number of generators and a finite number of relations.
\end{enumerate}
\end{ex}

\begin{ex}\label{Ex6.25}
\begin{enumerate}
\item Show that $M_n(k)$ is algebra-isomorphic to the quiver
\[
\begin{array}{c}
\begin{tikzpicture}[xscale=1.25,bend angle=20,looseness=1]
\node (a) at (-1,0) [vertex] {};
\node (b) at (0,0) [vertex] {};
\node (c) at (1,0) [vertex] {};
\node (c2) at (1.5,0) {$\hdots$};
\node (d) at (2,0) [vertex] {};
\node (e) at (3,0) [vertex] {};
\node at (-1,-0.2) {$\scriptstyle 1$};
\node at (0,-0.2)  {$\scriptstyle 2$};
\node at (1,-0.2)  {$\scriptstyle 3$};
\node at (2,-0.2)  {$\scriptstyle n-1$};
\node at (3,-0.2)  {$\scriptstyle n$};
\draw[->,bend left] (a) to node[above] {$\scriptstyle f_{1}$} (b);
\draw[->,bend left] (b) to  node[above] {$\scriptstyle f_{2}$}(c);
\draw[->,bend left] (d) to  node[above] {$\scriptstyle f_{n-1}$}(e);
\draw[->,bend left] (b) to node[below] {$\scriptstyle g_{1}$} (a);
\draw[->,bend left] (c) to  node[below] {$\scriptstyle g_{2}$}(b);
\draw[->,bend left] (e) to  node[below] {$\scriptstyle g_{n-1}$}(d);
 \end{tikzpicture}
\end{array}
\]
subject to the relations $f_i g_i=e_{i}$ and $g_i f_i=e_{i+1}$ for all $i$ such that $1\leq i\leq n-1$.
\item (The quiver proof that $M_n(k)$ is semisimple).  By (1) representations of the above quiver with relations are the same thing as finite dimensional $M_n(k)$-modules.  Using the quiver with relations, show that there is precisely one simple $M_n(k)$-module, and it has dimension $n$.  Further show directly that every finite dimensional $M_n(k)$-module is the finite direct sum of this simple module.
\item (Direct proof of Morita equivalence). Using (1) and (2), show that there is an equivalence of categories between $\mod k$ and $\mod M_{n}(k)$.
\end{enumerate}
\end{ex}

\begin{ex}\label{Ex6.26}
Show, using quivers, that $k[x]/x^{n}$ has precisely one simple module (up to isomorphism), and it has dimension one.
\end{ex}

\begin{ex}\label{Ex6.27}
Let $k$ be a field of characteristic zero. Let $A$ be the first Weyl algebra, that is the path algebra of the quiver
\[
\begin{tikzpicture}[bend angle=15, looseness=1]
\node (d) at (2,0) [vertex] {};
 \draw[->]  (d) edge [in=145,out=-145,loop,looseness=8] node[left] {$\scriptstyle X$}  (d);
  \draw[<-]  (d) edge [in=-35,out=35,loop,looseness=8] node[right] {$\scriptstyle Y$}  (d);
\end{tikzpicture}
\]
subject to the relation $XY-YX=1$.  Show that $\{ 0\}$ is the only finite dimensional $A$-module.
\end{ex}

\newpage
\section{Solution to Exercises}

\begin{proof}[Solution to Exercise \ref{Ex1.13}]
(2) $S_0$ is generated by $x^r$, $y^r$ and $xy$, so it is isomorphic as a ring to $\mathbb{C}[a,b,c]/(ab-c^r)$.  As an $S_0$-module, for $i>0$ we have $S_i$ is generated by $x^i$ and $y^{r-i}$ (depending on conventions). This leads to the quiver 
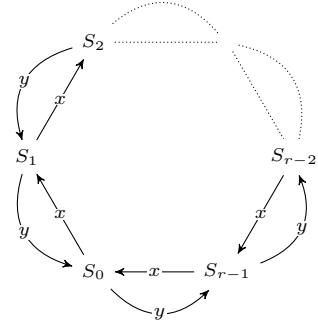
\[
\begin{tikzpicture} [bend angle=45, looseness=1.2]
\node[name=s,regular polygon, regular polygon sides=6, minimum size=3.5cm] at (0,0) {}; 
\node (C1) at (s.corner 1)  {};
\node (C2) at (s.corner 2)  {$\scriptstyle S_2$};
\node (C3) at (s.corner 3)  {$\scriptstyle S_1$};
\node (C4) at (s.corner 4)  {$\scriptstyle S_0$};
\node (C5) at (s.corner 5)  {$\scriptstyle S_{r-1}$};
\node (C6) at (s.corner 6)  {$\scriptstyle S_{r-2}$};
%standard x arrows 
\draw[->] (C4) -- node[gap] {$\scriptstyle x$} (C3); 
\draw[->] (C3) -- node[gap] {$\scriptstyle x$} (C2);
\draw[densely dotted] (C2) -- (C1);
\draw[densely dotted] (C1) -- (C6);
\draw[->] (C6) -- node[gap] {$\scriptstyle x$} (C5);
\draw[->]  (C5) -- node[gap] {$\scriptstyle x$} (C4);
%standard y arrows
\draw [densely dotted,bend right] (C1) to  (C2);
\draw [->,bend right] (C2) to node[gap] {$\scriptstyle y$} (C3);
\draw [->,bend right] (C3) to node[gap] {$\scriptstyle y$} (C4);
\draw [->,bend right] (C4) to node[gap] {$\scriptstyle y$} (C5);
\draw [->,bend right] (C5) to node[gap] {$\scriptstyle y$} (C6);
\draw [densely dotted,bend right] (C6) to (C1);
\end{tikzpicture} 
\]
(3) For $\frac{1}{3}(1,1)$ $S_0$ is generated as a ring by $x^3, x^2y, xy^2, y^3$, whereas  $S_1$ is generated as an $S_0$-module by $x$ and $y$, and $S_2$ is generated as an $S_0$-module by $x^2, xy, y^2$.  The quivers are
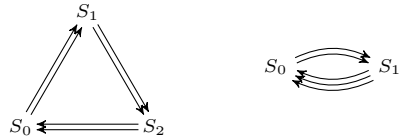
\[
\begin{array}{ccc}
\begin{array}{c}
\begin{tikzpicture}[bend angle=5,looseness=1.2] 
\node[name=s,regular polygon, regular polygon sides=3, minimum size=2cm] at (0,0) {}; 
\node (C1) at (s.corner 1)  {$\scriptstyle S_1$};
\node (C2) at (s.corner 2) {$\scriptstyle S_0$};
\node (C3) at (s.corner 3) {$\scriptstyle S_2$};
%\draw[->] (C2) -- (C1); 
%\draw[->] (C1) -- (C3); 
%\draw[->] (C3) -- (C2); 
\draw[->] (C2)+(50:6pt) -- ($(C1)+(-110:6pt)$); 
\draw[->] (C2)+(70:6pt) -- ($(C1)+(-130:6pt)$); 
\draw[->] (C1)+(-50:6pt) -- ($(C3)+(110:6pt)$); 
\draw[->] (C1)+(-70:6pt) -- ($(C3)+(130:6pt)$); 
\draw[->] (C3)+(170:6pt) -- ($(C2)+(10:6pt)$); 
\draw[->] (C3)+(-170:6pt) -- ($(C2)+(-10:6pt)$); 
\end{tikzpicture}
\end{array} 
&&
\begin{array}{c}
\begin{tikzpicture} [bend angle=25, looseness=1]
\node (C1) at (0,0) {$\scriptstyle S_0$};
\node (C2) at (1.5,0)  {$\scriptstyle S_1$};
\draw[->, bend left] (C1)+(10:8pt) to ($(C2)+(170:8pt)$); 
\draw[->, bend left] (C1)+(25:8pt) to ($(C2)+(155:8pt)$); 
\draw[->, bend left] (C2)+(-170:8pt) to ($(C1)+(-10:8pt)$); 
\draw[->, bend left] (C2)+(-155:8pt) to ($(C1)+(-25:8pt)$); 
\draw[->, bend left] (C2)+(-140:8pt) to ($(C1)+(-40:8pt)$); 
%
%\draw [->,bend left=20,looseness=1] (C2) to (C1);
%\draw [->,bend left=50,looseness=1] (C2) to (C1);
\end{tikzpicture}
\end{array}
\end{array}
\]
For $\frac{1}{5}(1,2)$ the generators of the ring (again, up to conventions) are $x^5, x^3y, xy^2, y^5$.  As $S_0$-modules, $S_1$ is generated by $x$ and $y^3$, $S_2$ by $x^2$ and $y$, $S_3$ by $x^3, xy, y^4$, and $S_4$ by $x^4, x^2y, y^2$.  The quivers are  
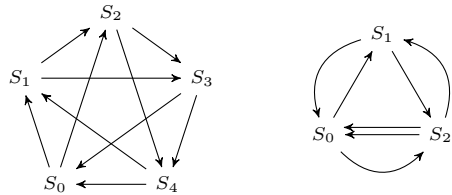
\[
\begin{array}{ccc}
\begin{array}{c}
\begin{tikzpicture}[bend angle=5,looseness=1.2] 
\node[name=s,regular polygon, regular polygon sides=5, minimum size=2.5cm] at (0,0) {}; 
\node (C1) at (s.corner 1)  {$\scriptstyle S_2$};
\node (C2) at (s.corner 2) {$\scriptstyle S_1$};
\node (C3) at (s.corner 3) {$\scriptstyle S_0$};
\node (C4) at (s.corner 4) {$\scriptstyle S_4$};
\node (C5) at (s.corner 5) {$\scriptstyle S_3$};
\draw[->] (C2) -- (C1); 
\draw[->] (C1) -- (C5); 
\draw[->] (C5) -- (C4);
\draw[->] (C4) -- (C3);
\draw[->] (C3) -- (C2); 
\draw[->] (C3) -- (C1); 
\draw[->] (C2) -- (C5); 
\draw[->] (C1) -- (C4);
\draw[->] (C5) -- (C3);
\draw[->] (C4) -- (C2); 
\end{tikzpicture}
\end{array} 
&&
\begin{array}{c}
\begin{tikzpicture} [bend angle=45, looseness=1.2]
%\draw[blue,densely dashed] (0,0) circle(1cm); 
\node[name=s,regular polygon, regular polygon sides=3, minimum size=1.75cm] at (0,0) {}; 
\node (C1) at (s.corner 1) {$\scriptstyle S_1$};
\node (C2) at (s.corner 2) {$\scriptstyle S_0$};
\node (C3) at (s.corner 3) {$\scriptstyle S_2$};
%standard x arrows that don't need to move
%\draw[->] (C3) --  (C2);
\draw[->] (C2) --  (C1);
\draw[->] (C1) --  (C3);
%standard y arrows
\draw [->,bend right] (C1) to  (C2);
\draw [->,bend right] (C2) to (C3);
\draw [->,bend right] (C3) to  (C1);
%extra arrows plus black one that need to move
\draw[->] (C3)++(160:8pt) --  ($(C2) + (20:8pt)$);
\draw[->] (C3)++(180:8pt) --  ($(C2) + (0:8pt)$);
\end{tikzpicture} 
\end{array}
\end{array}
\]
\end{proof}

\begin{proof}[Solution to Exercise \ref{Ex1.14}]
(1), (2) and (4) are CM.  (3) is not.
\end{proof}

\begin{proof}[Solution to Exercise \ref{Ex1.15}]
(1)(a),(b) are not CM, whereas (c) and (d) are.\\
(2)(a),(b),(c) are not CM, whereas (d) is CM, as are $(u,x)$ and $(u,y)$.  The module $(u^2,ux,x^2)$ is not.  Part (3) is similar.
\end{proof}

\begin{proof}[Solution to Exercise \ref{Ex1.16}]
All hypersurfaces are Gorenstein, so this implies that (1),(2) and (3) in Exercise \ref{Ex1.15} are Gorenstein, as are (1),(4) in \ref{Ex1.14}.  The ring in (2) in \ref{Ex1.14} is not Gorenstein.
\end{proof}

\begin{proof}[Solution to Exercise \ref{Ex2.16}]
Label the algebras $A, B, C, D$ from left to right.  Then
\begin{center}
\begin{tabular}{ccccc}
Question & $A$ & $B$& $C$&$D$\\ \hline
(1) & $\mathbb{C}[x,y]$&$\mathbb{C}[x,y]^{\frac{1}{2}(1,1)}$&$\mathbb{C}[x,y,z]^{\frac{1}{3}(1,2)}$ &$\mathbb{C}[x,y,z]^{\frac{1}{3}(1,1)}$\\
(2)a &  $M=R\oplus (x,y)$ & $M=S_0\oplus S_1$ & $M=S_0\oplus S_1$ & $M=S_0\oplus S_1$\\
(2)b & Not CM &Yes CM &Yes CM&Yes CM \\  
(3) & No &Yes &Yes&Yes \\  
(4) & $2, 1$ &$2, 2$ &$\infty, \infty$&$2, 3$ \\  
(5) &&&resolutions are periodic &\\
(6) &$2$ & $2$&$\infty$ & $3$\\
(7) & &NCCR &&\\
(8) & $\mathcal{O}_{\mathbb{P}^1}(-1)$ & $\mathcal{O}_{\mathbb{P}^1}(-2)$ & $X_1$&$\mathcal{O}_{\mathbb{P}^1}(-3)$ \\
\end{tabular}
\end{center}
where for (2)a we use the notation from Exercise \ref{Ex1.13}, and in (8)  $X_1$ is one of the partial resolutions of $\mathbb{C}[x,y]^{\frac{1}{3}(1,2)}$ containing only one curve.  The fact that $X_1$ has only hypersurface singularities is the phenomenon that explains the periodicity in (5).
\end{proof}

\begin{proof}[Solution to Exercise \ref{Ex2.17}]
(1) Set $M_1:=(u,x-1)$, $M_2:=(u,x(x-1))$ and $M_3:=(u,x^2(x-1))$.  The main calculation is 
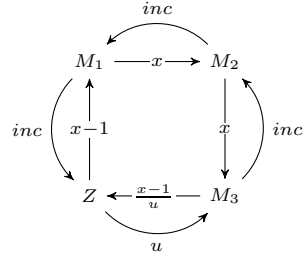
\[
\begin{array}{cc}
\begin{array}{c}
{\begin{tikzpicture} [bend angle=45,looseness=1]
\node[name=s,regular polygon, regular polygon sides=4, minimum size=2.5cm] at (0,0) {}; 
\node (C1) at (s.corner 1) {$\scriptstyle M_2$};
\node (C2) at (s.corner 2) {$\scriptstyle M_1$};
\node (C4) at (s.corner 4) {$\scriptstyle M_3$};
\node (C3) at (s.corner 3) {$\scriptstyle Z$};
\draw[->] (C4) -- node[gap] {$\scriptstyle \frac{x-1}{u}$} (C3); 
\draw[->] (C3) -- node[gap] {$\scriptstyle x-1$} (C2); 
\draw[->] (C2) -- node[gap] {$\scriptstyle x$} (C1); 
\draw[->] (C1) -- node[gap] {$\scriptstyle x$} (C4);
\draw [->,bend right] (C1) to node[above]  {$\scriptstyle inc$} (C2);
\draw [->,bend right] (C2) to node[left]  {$\scriptstyle inc$} (C3);
\draw [->,bend right] (C3) to node[below]  {$\scriptstyle u$} (C4);
\draw [->,bend right] (C4) to node[right]  {$\scriptstyle inc$} (C1);
\end{tikzpicture}}
\end{array}
\end{array}
\]
from which, after some work, the presentation follows.\\
(2) We can check whether $\Lambda\cong\End_Z(Z\oplus M_1\oplus M_2\oplus M_3)$ is a NCCR by localizing to the maximal ideals, and there (up to Morita equivalence) we find the NCCR from \ref{Ex2.16}.\\
(3) There are many.  Set $M:=Z\oplus (u,x-1)\oplus (u,x(x-1))\oplus (u,x^2(x-1))$, the module from (1), then for example taking $N:=Z\oplus (u,x)\oplus (u,x-1)$ we see that $\End_Z(N)$ is Morita equivalent to $\End_Z(M)$, since $\add M=\add N$.
\end{proof}

\begin{proof}[Solution to Exercise \ref{Ex2.18}]
(2) As right modules, the projective resolutions of the simples $S_1, S_2, S_3$ are
\begin{align*}
0\to e_3A \to
(e_1A)^{2}\to e_{3}A\oplus e_{2}A\to
e_1A\to S_1\to0\\
0\to  e_2A\to e_1A\oplus e_3A\to e_2A\to S_2\to0\\
0\to  (e_3A)^{\oplus 2}\to
e_2A\oplus (e_1A)^{\oplus 2}\to e_3A\to S_3\to0
\end{align*}
See \cite[6.9]{WemReconA} for the general form of the projective resolutions.  Part (1) is similar.  For (3), the key point is that $Z(\Lambda)$ is not Gorenstein, and in general we can't apply Auslander--Buchsbaum unless $\Lambda$ is a Gorenstein  $Z(\Lambda)$-order, which it is not.
\end{proof}

\begin{proof}[Solution to Exercise \ref{Ex2.19}]
Label the algebras $A, B, C, D$ from left to right.  Then
\begin{center}
\begin{tabular}{ccccc}
Question & $A$ & $B$& $C$&$D$\\ \hline
(2) & $\dfrac{\mathbb{C}[u,v,x,y]}{(uv-xy)}$ & $\dfrac{\mathbb{C}[u,v,x,y]}{(uv-x^2)}$ & $\dfrac{\mathbb{C}[u,v,x,y]}{(uv-(x-y^2)(x+y^2))}$ & $\mathbb{C}[x,y,z]^{\frac{1}{3}(1,1,1)}$\\ \\
(3) &  e.g.\ consider &e.g.\ consider&e.g.\ consider&Consider \\ 
 & $R\oplus (u,x)$ & $R\oplus (u,x)$ & $R\oplus (u,x+y^2)$ & $S_0\oplus S_1\oplus S_2$\\ \\
(4) & $\mathcal{O}_{\mathbb{P}^1}(-1)^{\oplus 2}$ & $\mathcal{O}_{\mathbb{P}^1}(-2)\oplus \mathcal{O}_{\mathbb{P}^1}$ & $Y_1$&$\mathcal{O}_{\mathbb{P}^2}(-3)$ \\
\end{tabular}
\end{center}
where for (3)$D$ we use the notation from Exercise \ref{Ex1.13}, and in (4)  $Y_1$ is the blowup of the ideal $(u,x+y^2)$, which forms one half of the Pagoda flop.   The first three examples capture the phenomenon of Type $A$ contractions in 3-folds; in example $A$ the curve has width $1$, in Example $C$ the curve has width $2$ (changing the $2$ to $n$ in the relations gives the example with width $n$), and in example $B$ the curve has width $\infty$.
\end{proof}

\begin{proof}[Solution to Exercise \ref{Ex2.20}]
(3) With $R$ a complete local CM ring of dimension three, the general result is that $\Hom_R(M,N)\in\CM R$ if and only if $\depth_R\Ext^1_R(M,N)>0$.  See for example \cite[2.7]{IW3}.
\end{proof}

\begin{proof}[Solution to Exercise \ref{Ex2.21}]
(3) The general theorem due to Watanabe is that if $G$ has no complex reflections, then the invariant ring is always CM, and it is Gorenstein if and only if $G\leq \SL(n,\mathbb{C})$.
\end{proof}

\begin{proof}[Solution to Exercise \ref{Ex3.12}]
Label the algebras $A, B, C, D$ from left to right.  Then
\begin{center}
\begin{tabular}{ccccc}
 & $A$ & $B$& $C$&$D$\\ \hline
$\theta=(-1,1)$ & $\mathcal{O}_{\mathbb{P}^1}(-1)$ & $\mathcal{O}_{\mathbb{P}^1}(-2)$ & $Y_1$&$\mathcal{O}_{\mathbb{P}^1}(-3)$ \\
$\theta=(1,-1)$ & $\mathbb{C}^2$ & $\mathcal{O}_{\mathbb{P}^1}(-2)$ & $Y_2$&$Z_1$ 
\end{tabular}
\end{center}
where $Y_1$ is one of the partial resolutions of the $\frac{1}{3}(1,2)$ singularity containing only one curve, $Y_2$ is the other, and $Z_1$ is a scheme with two components, one of which is $\mathcal{O}_{\mathbb{P}^1}(-3)$.  
\end{proof}

\begin{proof}[Solution to Exercise \ref{Ex3.13}]
Label the algebras $A, B, C$ from left to right.  Then
\begin{center}
\begin{tabular}{cccc}
 & $A$ & $B$& $C$\\ \hline
$\theta=(-1,1)$ & $\mathcal{O}_{\mathbb{P}^1}(-1)^{\oplus 2}$ & $\mathcal{O}_{\mathbb{P}^1}(-2)\oplus\mathcal{O}_{\mathbb{P}^1}$ & $Y_1$ \\
$\theta=(1,-1)$ & $\mathcal{O}_{\mathbb{P}^1}(-1)^{\oplus 2}$ & $\mathcal{O}_{\mathbb{P}^1}(-2)\oplus\mathcal{O}_{\mathbb{P}^1}$ & $Y_2$
\end{tabular}
\end{center}
where $Y_1$ is blowup of $\Spec\frac{\mathbb{C}[u,v,x,y]}{(uv-(x-y^2)(x+y^2))}$ at the ideal $(u,x+y^2)$ and $Y_2$ is blowup at the ideal $(u,x-y^2)$.   In examples $A$ and $C$, the two spaces are abstractly isomorphic, but not isomorphic in a compatible way over the base; they are examples of flops. 
\end{proof}

\begin{proof}[Solution to Exercise \ref{Ex3.14}]
Since $\theta_0=-\theta_1-\theta_2$, the stability condition is determined by the pair $(\theta_1,\theta_2)$.  The chamber structure is
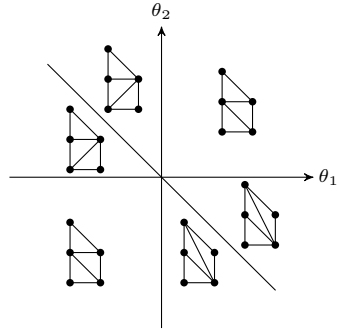
\[
\begin{tikzpicture}
\draw[->] (-2,0) -- (2,0);
\node at (2.2,0) {$\scriptstyle \theta_1$};
\node at (0,2.2) {$\scriptstyle \theta_2$};
\draw[->] (0,-2) -- (0,2);
\draw[-] (-1.5,1.5) -- (1.5,-1.5);
%%top right and bottom left
\node at (1,1) {\begin{tikzpicture}[scale=0.4]
\node at (0,2) [vertex] {};
\node at (0,1) [vertex] {};
\node at (0,0) [vertex] {};
\node at (1,1) [vertex] {};
\node at (1,0) [vertex] {};
\coordinate (A) at (0,2);
\coordinate  (B) at (0,1);
\coordinate (C) at (0,0);
\coordinate (D) at (1,1);
\coordinate (E) at (1,0);
\draw (A) --(B)-- (C)--(E)--(D)--(A);
\draw (B)--(D);
\draw (B)--(E);
\end{tikzpicture}};
\node at (-1,-1) {\begin{tikzpicture}[scale=0.4]
\node at (0,2) [vertex] {};
\node at (0,1) [vertex] {};
\node at (0,0) [vertex] {};
\node at (1,1) [vertex] {};
\node at (1,0) [vertex] {};
\coordinate (A) at (0,2);
\coordinate  (B) at (0,1);
\coordinate (C) at (0,0);
\coordinate (D) at (1,1);
\coordinate (E) at (1,0);
\draw (A) --(B)-- (C)--(E)--(D)--(A);
\draw (B)--(D);
\draw (B)--(E);
\end{tikzpicture}};
%%bottom rights
\node at (1.3,-0.5) {\begin{tikzpicture}[scale=0.4]
\node at (0,2) [vertex] {};
\node at (0,1) [vertex] {};
\node at (0,0) [vertex] {};
\node at (1,1) [vertex] {};
\node at (1,0) [vertex] {};
\coordinate (A) at (0,2);
\coordinate  (B) at (0,1);
\coordinate (C) at (0,0);
\coordinate (D) at (1,1);
\coordinate (E) at (1,0);
\draw (A) --(B)-- (C)--(E)--(D)--(A);
\draw (A)--(E);
\draw (B)--(E);
\end{tikzpicture}};
\node at (0.5,-1) {\begin{tikzpicture}[scale=0.4]
\node at (0,2) [vertex] {};
\node at (0,1) [vertex] {};
\node at (0,0) [vertex] {};
\node at (1,1) [vertex] {};
\node at (1,0) [vertex] {};
\coordinate (A) at (0,2);
\coordinate  (B) at (0,1);
\coordinate (C) at (0,0);
\coordinate (D) at (1,1);
\coordinate (E) at (1,0);
\draw (A) --(B)-- (C)--(E)--(D)--(A);
\draw (A)--(E);
\draw (B)--(E);
\end{tikzpicture}};
%%top rights
\node at (-1,0.5) {\begin{tikzpicture}[scale=0.4]
\node at (0,2) [vertex] {};
\node at (0,1) [vertex] {};
\node at (0,0) [vertex] {};
\node at (1,1) [vertex] {};
\node at (1,0) [vertex] {};
\coordinate (A) at (0,2);
\coordinate  (B) at (0,1);
\coordinate (C) at (0,0);
\coordinate (D) at (1,1);
\coordinate (E) at (1,0);
\draw (A) --(B)-- (C)--(E)--(D)--(A);
\draw (B)--(D);
\draw (C)--(D);
\end{tikzpicture}};
\node at (-0.5,1.3) {\begin{tikzpicture}[scale=0.4]
\node at (0,2) [vertex] {};
\node at (0,1) [vertex] {};
\node at (0,0) [vertex] {};
\node at (1,1) [vertex] {};
\node at (1,0) [vertex] {};
\coordinate (A) at (0,2);
\coordinate  (B) at (0,1);
\coordinate (C) at (0,0);
\coordinate (D) at (1,1);
\coordinate (E) at (1,0);
\draw (A) --(B)-- (C)--(E)--(D)--(A);
\draw (B)--(D);
\draw (C)--(D);
\end{tikzpicture}};
\end{tikzpicture}
\]
where in each chamber is the toric fan corresponding to the quiver GIT for that chamber.  There are thus three crepant resolutions.
\end{proof}

\begin{proof}[Solution to Exercise \ref{Ex3.15}]
The invariants are generated by $R_1:=as$, $R_2:=at=bs$, $R_3:=bt$, $v$ and $w$, and abstractly the invariant ring is isomorphic to
\[
\frac{\mathbb{C}[R_1,R_2,R_3,v,w]}{\begin{array}{l} R_2^2-R_1R_3\\
(v-w)R_1\\
(v-w)R_2\\
(v-w)R_3
 \end{array}}
\]
\end{proof}

\begin{proof}[Solution to Exercise~\ref{Ex4.30}]
(1) No.  Take for example $R=\mathbb{C}[x,y]$ and $M=R\oplus (x,y)$.  (2) See \cite[4.1]{AG} 
\end{proof}

\begin{proof}[Solution to Exercise~\ref{Ex4.31}]
(1) Note that $\End_R(R\oplus M)\in\CM R$ implies that $M\in\CM R$.  Because of this, $\Lambda_\p\cong\End_{R_\p}(R_\p\oplus M_\p)$ implies that for any prime $\p$ not in the singular locus, $\Lambda_\p\cong M_n(R_\p)$ for some $n$.  On the other hand, finite global dimension is preserved under localization, so if $\p$ is in the singular locus then $\Lambda_\p$ cannot be a matrix algebra over $R_\p$ (since they have infinite global dimension).  Thus the Azumaya locus equals the nonsingular locus.  (2) Note in (1) that both loci equal the locus on which $M$ is not free. This can be described as the support of the module $\Ext^1_R(M,\Omega M)$, and hence is closed.  (3) The $K_0$ group of $\End_R(R\oplus M)$, which is finitely generated since the global dimension is finite (and $R$ is now local), surjects onto the class group of $R$.
\end{proof}

\begin{proof}[Solution to Exercise~\ref{Ex4.32}]
$\mathbb{C}\oplus \mathbb{C}[x]$.
\end{proof}

\begin{proof}[Solution to Exercise~\ref{Ex4.33}]
See \cite[5.4]{IW3}.
\end{proof}

\begin{proof}[Solution to Exercise~\ref{Ex4.34}]
The most direct way is to establish (using for example the snake lemma) that $\End_Y(\mathcal{O}_Y\oplus \mathcal{L})\cong\End_R(R\oplus (x,y))$ via the global sections functor, where $R=\mathbb{C}[x,y]$.  From there, in the presentation the arrow $a$ corresponds to multiplication by $x$ taking an element from $R$ to $(x,y)$, similarly the arrow $b$ to multiplication by $y$, and the arrow $t$ corresponds to the inclusion of the ideal $(x,y)$ into $R$.
\end{proof}

\begin{proof}[Solution to Exercise~\ref{Ex5.22}]
(2) is the singularity $uv=x^2y$ from Example~\ref{SSP text} and Exercise~\ref{Ex3.14}.  For the remainder, see \cite[\S16]{Y}. 
\end{proof}

\begin{proof}[Solution to Exercise~\ref{Ex5.24}]
By Artin--Verdier theory, taking the torsion free lift $\mathcal{M}$ of $M$ to the minimal resolution, there are short exact sequences $0\to \mathcal{O}^{\oplus r}\to \mathcal{M}\to \mathcal{O}_D\to 0$ and $0\to \mathcal{M}^{*}\to {\mathcal{O}}^{\oplus r}\to \mathcal{O}_D\to 0$ where $D$ is a divisor transversal to the curve corresponding to $\mathcal{M}$.  Taking the pullback of these sequences, the middle exact sequence splits giving an exact sequence $0\to \mathcal{M}^{*}\to {\mathcal{O}}^{\oplus 2r}\to \mathcal{M}\to 0$, proving the statement.
\end{proof}

\begin{proof}[Solution to Exercise~\ref{Ex5.26}]
The Ext groups are isomorphic, since the  Kn\"orrer functor gives an equivalence of categories $\underline{\CM}R\simeq \underline{\CM}R'$, known as Kn\"orrer periodicity.  See for example \cite[\S12]{Y}.
\end{proof}

\begin{proof}[Solution to Exercise~\ref{Ex6.21}]
(1) $8$, (2) $11$, (3) $\infty$.  The general result is that the path algebra is finite dimensional if and only if there is no oriented cycle.
\end{proof}

\begin{proof}[Solution to Exercise~\ref{Ex6.23}]
(1) $k[x]/(x^3-1)$.  When $k=\mathbb{C}$ this is isomorphic to $ k\oplus k\oplus k$ which can be viewed as just a quiver with three dots. (2) $k[x,y]$, (3) $k\oplus k[x]$. 
\end{proof}

\begin{proof}[Solution to Exercise~\ref{Ex6.24}]
(1) One vertex, $n$ loops, no relations.  (2) One vertex, $n$ loops, the commutativity relations.  (3) Draw the quivers side-by side, and take the union of the relations.  (4) Since $\mathbb{C}G$ is semisimple, it is a direct product of matrix rings.  Combine answers for (3) above and \ref{Ex6.25}(1) below.  Alternatively, work up to Morita equivalence, where $\mathbb{C}G$ is then just a finite number of dots, with no relations.  (5)  One vertex, number of loops=number of generators, then the finite number of relations.
\end{proof}

\begin{proof}[Solution to Exercise~\ref{Ex6.25}]
(1) is an easy extension of \ref{ex2}(2).  (2) then follows as in \ref{quiver example decompose}(2) and \ref{ex5}.  (3) Consider the functor $\mod k\to \mod M_n(k)$ 
\[
V\mapsto
\begin{array}{c}
\begin{tikzpicture}[bend angle=20]
\node (a) at (-1,0)  {$\scriptstyle V$};
\node (b) at (0,0) {$\scriptstyle V$};
\node (c) at (1,0) {$\scriptstyle V$};
\node (c2) at (1.5,0) {$\hdots$};
\node (d) at (2,0) {$\scriptstyle V$};
\node (e) at (3,0) {$\scriptstyle V$};
\draw[->,bend left] (a) to node[above] {$\scriptstyle 1$} (b);
\draw[->,bend left] (b) to  node[above] {$\scriptstyle 1$}(c);
\draw[->,bend left] (d) to  node[above] {$\scriptstyle 1$}(e);
\draw[->,bend left] (b) to node[below] {$\scriptstyle 1$} (a);
\draw[->,bend left] (c) to  node[below] {$\scriptstyle 1$}(b);
\draw[->,bend left] (e) to  node[below] {$\scriptstyle 1$}(d);
 \end{tikzpicture}
\end{array}
\]
This is clearly fully faithful, and is essentially surjective by (2).
\end{proof}

\begin{proof}[Solution to Exercise~\ref{Ex6.26}]
Take an arbitrary simple module, which is necessarily finite dimensional.  View as a quiver representation, with vector space $V$ and loop corresponding to a linear map $f\colon V\to V$ such that $f^n=0$.  Consider the kernel, then this gives a submodule, so necessarily the kernel must be everything (since $f$ cannot be injective).  Thus $f$ must be the zero map, so $V$ must be one-dimensional else the representation decomposes. 
\end{proof}

\begin{proof}[Solution to Exercise~\ref{Ex6.27}]
View any such simple as a finite dimensional representation of the quiver.  Thus $X$ and $Y$ are linear maps from a finite dimensional vector space to itself, and they must satisfy the relation $XY-YX=1$.  Taking the trace of this equation gives $0=n$, which is a contradiction.
\end{proof}


\newpage
\begin{thebibliography}{VdB2}

\bibitem[AV85]{AV}
M.~Artin and J.-L.~Verdier, \emph{Reflexive modules over rational double points}, Math.\ Ann.\ \textbf{270} (1985), no.~1, 79--82. 

\bibitem[A71]{A71}
M. Auslander,
\emph{Representation dimension of Artin algebras}. Lecture notes,
Queen Mary College, London, 1971.

\bibitem[A78]{Aus78}
M. Auslander,
\emph{Functors and morphisms determined by objects}.  Representation theory of algebras (Proc.\ Conf., Temple Univ., Philadelphia, Pa., 1976), pp.\ 1--244. Lecture Notes in Pure Appl.\ Math., Vol. 37, Dekker, New York, 1978.


\bibitem[A84]{Aus2}
M.~Auslander,
\emph{Isolated singularities and existence of almost split sequences}. Representation theory, II (Ottawa, Ont., 1984), 194--242, Lecture Notes in Math., 1178, Springer, Berlin, 1986. 

\bibitem[A86]{Aus3}
M.~Auslander, \emph{Rational singularities and almost split sequences}. Trans.\ Amer.\ Math.\ Soc.\ \textbf{293} (1986), no.~2, 511--531.

\bibitem[AB]{AB}
M.~Auslander and D.~A.~Buchsbaum, \emph{Homological dimension in local rings}, Trans.\ Amer.\ Math.\ Soc.\ \textbf{85} (1957), 390--405.

\bibitem[AG60]{AG}
M.~Auslander and O.~Goldman, \emph{Maximal orders}, Trans.\ Amer.\ Math.\ Soc.\ \textbf{97} (1960), 1--24.


\bibitem[ARS]{ARS}
M.~Auslander, I.~Reiten, S.~O.~Smalo, \textit{Representation theory of Artin algebras},  Cambridge Studies in Advanced Mathematics, \textbf{36}. Cambridge  University Press, Cambridge, 1995.

\bibitem[B63]{Bass}
H.~Bass, \emph{On the ubiquity of Gorenstein rings}, Math.\ Z.\ \textbf{82} (1963), 8--28.

\bibitem[B83]{Beilinson}
A.~A.~Beilinson, \emph{The derived category of coherent sheaves on $\mathbb{P}^n$}, Selected translations, Selecta Math.\ Soviet.\ \textbf{3} (1983/84), no.~3, 233--237.

\bibitem[BSW]{BSW}
R.~Bocklandt, T.~Schedler and M.~Wemyss, \emph{Superpotentials and Higher Order Derivations}, J.\ Pure Appl.\ Algebra \textbf{214} (2010), no.\ 9, 1501--1522.


\bibitem[B99]{BridgelandLMS}
T.~Bridgeland, \emph{Equivalences of triangulated categories and Fourier-Mukai transforms}, Bull.\ London Math.\ Soc.\ \textbf{31} (1999), no.~1, 25--34.

\bibitem[B02]{Bridgeland}
T.~Bridgeland, \emph{Flops and derived categories}. Invent. Math. \textbf{147} (2002), no.~3, 613--632.

\bibitem[BI]{BI}
T.~Bridgeland and S.~Iyengar, \emph{A criterion for regularity of local rings}, C.\ R.\ Math.\ Acad.\ Sci.\ Paris \textbf{342} (2006), no.~10, 723--726.

\bibitem[BKR]{BKR}
T.~Bridgeland, A.~King and M.~Reid, \emph{The McKay correspondence as an equivalence of derived categories}, J. Amer. Math. Soc. \textbf{14} (2001), no. 3, 535--554. 

\bibitem[BM]{BM}
T.~Bridgeland and A.~Maciocia, \emph{Fourier-Mukai transforms for K3 and elliptic fibrations}, J.\ Algebraic Geom.\ \textbf{11} (2002), no.~4, 629--657.

\bibitem[BH84]{BrownHar}
K.~A. Brown and C.~R. Hajarnavis, \emph{Homologically homogeneous rings},  Trans.\ Amer.\ Math.\ Soc.\, \textbf{281} (1984), no.~1,  197--208.

\bibitem[BH]{BH}
W.~Bruns and J.~Herzog, \emph{Cohen-Macaulay rings. Rev. ed.}, Cambridge Studies in Advanced Mathematics. \textbf{39} xiv+453 pp.

\bibitem[Bu86]{Buch}
R.~O.~Buchweitz, \emph{Maximal Cohen--Macaulay modules and Tate--cohomology over Gorenstein rings}, preprint, 1986. \href{https://tspace.library.utoronto.ca/bitstream/1807/16682/1/maximal_cohen-macaulay_modules_1986.pdf}{link}


\bibitem[CS98]{CS}
H.~Cassens and P.~Slodowy, \emph{On {K}leinian singularities and quivers},  Progr. Math. \textbf{162} (1998), 263--288.


\bibitem[C08]{Crawnotes}
A.~Craw, \emph{Quiver representations in toric geometry}, {\sf ArXiv:0807.2191}.



\bibitem[CI04]{CI}
A.~Craw and A.~Ishii, \emph{Flops of $G$-Hilb and equivalences of derived categories by variation of GIT quotient}. Duke Math.\ J.\ \textbf{124} (2004), no.\ 2, 259--307. 


\bibitem[DH]{DH}
H.~Dao and C.~Huneke, \emph{Vanishing of Ext, cluster tilting modules and finite global dimension of endomorphism rings}, Amer.\ J.\ Math.\ \textbf{135} (2013), no.~2, 561--578.

\bibitem[E85]{Esnault}
H.~Esnault, \emph{Reflexive modules on quotient surface singularities}, J.
  Reine Angrew. Math. \textbf{362} (1985), 63--71.

\bibitem[EG02]{EG}
P.~Etingof and V.~Ginzburg, \emph{Symplectic reflection algebras, Calogero-Moser space, and deformed Harish-Chandra homomorphism}, Invent.\ Math.\ \textbf{147} (2002), no.~2, 243--348.

\bibitem[G06]{GinzCY}
V.~Ginzburg, \emph{Calabi--Yau algebras}, \href{http://arxiv.org/pdf/math/0612139v3.pdf}{\textsf{arXiv:math/0612139}}.

\bibitem[H77]{H}
R.~Hartshorne, \emph{Algebraic geometry}. Graduate Texts in Mathematics, No.\ 52.\ Springer-Verlag, New York-Heidelberg, 1977. xvi+496 pp.


\bibitem[HdB]{HVdB}
L.~Hille and M.~Van den Bergh, \emph{{F}ourier-{M}ukai transforms}, Handbook of tilting theory,  
LMS Lecture Note Ser., \textbf{332} (2007), 147--177.


\bibitem[H06]{HuybrechtsFM}
D.~Huybrechts, \emph{Fourier-Mukai transforms in algebraic geometry}. Oxford Mathematical Monographs. The Clarendon Press, Oxford University Press, Oxford, 2006. viii+307 pp.

\bibitem[I07]{IyamaAR}
O. Iyama, \emph{Higher-dimensional Auslander--Reiten theory on maximal orthogonal subcategories}, Adv.\ Math.\ \textbf{210} (2007), no. 1, 22--50.

\bibitem[IR08]{IR}
O. Iyama and I. Reiten, \emph{Fomin-Zelevinsky mutation and tilting modules over Calabi-Yau algebras}, Amer.\ J.\ Math.\ \textbf{130} (2008), no. 4, 1087--1149.


\bibitem[IT]{IT}
O.~Iyama and R.~Takahashi, \emph{Tilting and cluster tilting for
quotient singularities},  Math.\ Ann.\ \textbf{356} (2013), no.~3, 1065--1105.

\bibitem[IW10]{IW3}
O.~Iyama and M.~Wemyss, \emph{Auslander--Reiten duality and maximal modifications for non-isolated singularities}, Invent.\ Math.\ \textbf{197} (2014), no.~3, 521--586. 

\bibitem[IW12]{IW4}
O.~Iyama and M.~Wemyss, \emph{On the Noncommutative Bondal--Orlov Conjecture}, J.\ Reine Angew.\ Math.\ \textbf{683} (2013), 119--128.

\bibitem[IW11]{IW5}
O.~Iyama and M.~Wemyss, \emph{Singular derived categories of $\mathds{Q}$-factorial terminalizations and maximal modification algebras}, Adv.\ Math.\ \textbf{261} (2014), 85--121.


\bibitem[IW13]{IW6}
O. Iyama and M. Wemyss, \emph{Reduction of triangulated categories and Maximal Modification Algebras for $cA_n$ singularities}, 
\href{http://arxiv.org/pdf/1304.5259.pdf}{{\sf arXiv:1304.5259}}. 


\bibitem[KV]{KV}
M.~Kapranov and E.~Vasserot, \emph{Kleinian singularities, derived categories and hall algebras}, Math.\  Ann.\ \textbf{316}, (2000) no.~3, 565--576.

\bibitem[K96]{Keller}
B.~Keller, \emph{Derived categories and their uses}, Handbook of algebra, Vol.~1, 671--701, North-Holland, Amsterdam, 1996.
\href{http://www.math.jussieu.fr/~keller/publ/dcu.pdf}{link}.

\bibitem[K94]{King}
A.~D.~King, \emph{Moduli of representations of finite-dimensional algebras}, Quart.\ J.\ Math.\ Oxford Ser.\ (2) \textbf{45} (1994), no.\ 180, 515--530.

\bibitem[K86]{Kronheimer}
P.~Kronheimer, \emph{{A}{L}{E} gravitational instantons}, D.\ Phil.\ thesis, University of Oxford, (1986).


\bibitem[LeB]{leB}
L.~Le~Bruyn, \emph{Quotient singularities and the conifold algebra}, \href{http://win.ua.ac.be/~lebruyn/LeBruyn2004e.pdf}{link}.


\bibitem[L07]{Leuschke}
G.~Leuschke, \emph{Endomorphism rings of finite global dimension}, Canad.\ J.\ Math.\ \textbf{59} (2007), no.~2, 332--342. 

\bibitem[M96]{Martin}
R.~Martin, \emph{Skew group rings and maximal orders}, 
Glasgow Math.\ J.\ \textbf{37} (1995), no.~2, 249--263.

\bibitem[M80]{McKay}
J.~McKay, \emph{Graphs, singularities, and finite groups}, Proc.
Sympos. Pure  Math. \textbf{37} (1980), 183--186.


\bibitem[M]{Milicic}
D.~Milicic, \emph{Lectures on derived categories}, \href{http://www.math.utah.edu/~milicic/Eprints/dercat.pdf}{link}.



\bibitem[N12]{Nag}
K.~Nagao, \emph{Derived categories of small toric Calabi-Yau 3-folds and curve counting invariants}, Q.\ J.\ Math.\ \textbf{63} (2012), no.~4, 965--1007.


\bibitem[N96]{Neeman}
A.~Neeman, \emph{The {G}rothendieck duality theorem via {B}ousfield's  techniques and {B}rown representability}, J.\ Amer.\ Math.\ Soc.\ \textbf{9} (1996), no.~1, 205--236.


\bibitem[O03]{Orlov1}
D.~Orlov, \emph{Triangulated categories of singularities and $D$-branes in Landau-Ginzburg models},  Proc.\ Steklov Inst.\ Math.\ 2004, no. 3 (246), 227--248.

\bibitem[O05]{Orlov2}
D.~Orlov, \emph{Triangulated categories of singularities and equivalences between Landau-Ginzburg Models},  Sb.\ Math.\ \textbf{197} (2006), no.\ 11--12, 1827--1840.


\bibitem[RV89]{RV89}
I.~Reiten and M.~Van den Bergh, \emph{{T}wo-dimensional tame and maximal orders of finite representation type.}, Mem. Amer. Math. Soc. \textbf{408} (1989),
 vii+72pp.

\bibitem[R89]{Rickard}
J.~Rickard, \emph{Morita theory for derived categories}, J. London Math. Soc. (2) \textbf{39} (1989), 436--456.

\bibitem[R08]{Rouq}
R.~Rouquier, \emph{Dimensions of triangulated categories}, J. K-Theory \textbf{1} (2008), no.\ 2, 193--256.

\bibitem[TT]{TT}
R.~W.~Thomason and T.~Trobaugh, \emph{Higher algebraic K-theory of schemes and of derived categories}, The Grothendieck Festschrift, Vol.\ III, 247--435, Progr. Math., \textbf{88}, Birkh\"auser Boston, Boston, MA, 1990. 

\bibitem[VdB1]{VdB1d}
M.~Van den Bergh, \emph{Three-dimensional flops and noncommutative rings}, 
Duke Math. J. \textbf{122} (2004), no.~3, 423--455. 

\bibitem[VdB2]{VdBNCCR}
M.~Van den Bergh, \emph{Non-commutative crepant resolutions}, The legacy of Niels Henrik Abel, 749--770, Springer, Berlin, 2004.

\bibitem[W11]{WemReconA}
M.~Wemyss, \emph{Reconstruction Algebras of Type $A$}, Trans.\ Amer.\ Math.\ Soc.\ \textbf{363} (2011), 3101--3132.

\bibitem[W88]{Wunram}
J.~Wunram, \emph{Reflexive modules on quotient surface singularities}, Math.\ Ann.\ \textbf{279} (1988), no.~4, 583--598.

\bibitem[Y90]{Y}
Y. Yoshino, \emph{Cohen-Macaulay modules over Cohen-Macaulay rings},
London Mathematical Society Lecture Note Series, 146. Cambridge University Press, Cambridge, 1990.

\end{thebibliography}
\end{document}